\documentclass{scrreprt}
\usepackage[utf8]{inputenc}
\usepackage{amssymb,amsthm}
\usepackage{parskip}
\usepackage[ngerman]{babel}
\usepackage{amsmath}
\usepackage{amsfonts}
\usepackage{mathtools}
\usepackage{tikz}
\usetikzlibrary{matrix,arrows}
\usepackage{units}
\usepackage[square,numbers]{natbib}

\newtheoremstyle{jonas}{.25\baselineskip}{.25\baselineskip}{}{\parindent}{\bfseries}{.}{1em}{}
\theoremstyle{jonas}
\newtheorem{statement}{Blub}[chapter]
\newtheorem{Satz}[statement]{Satz}
\newtheorem{Def}[statement]{Definition}
\newtheorem{Lemma}[statement]{Lemma}
\newtheorem{Proposition}[statement]{Proposition}
\newtheorem{Konstruktion}[statement]{Konstruktion}
\newtheorem{Korollar}[statement]{Korollar}
\newtheorem{Theorem}[statement]{Theorem}

\newtheorem{Bem}[statement]{Bemerkung}
\newtheorem{Bem+Def}[statement]{Bemerkung und Definition}
\newtheorem{Bsp}[statement]{Beispiel}
\newtheorem{Konvention}[statement]{Konvention}

\DeclareMathOperator{\Komod}{-Komod}
\DeclareMathOperator{\ev}{ev}
\DeclareMathOperator{\Rep}{Rep}
\DeclareMathOperator{\Spec}{Spec }
\DeclareMathOperator{\Vect}{-Vect}
\DeclareMathOperator{\MM}{\mathcal{MM}^{eff}_{Nori}}
\DeclareMathOperator{\C}{{\mathcal{C}}}
\DeclareMathOperator{\Hom}{Hom}
\DeclareMathOperator{\Frei}{-Frei}
\DeclareMathOperator{\End}{End}
\DeclareMathOperator{\Mod}{-Mod}
\DeclareMathOperator{\Mor}{Mod-}

\begin{document}

\title{\vspace{3cm}Nori-Motive und Tannaka-Theorie }
\author{Jonas von Wangenheim}

\maketitle

\tableofcontents

\chapter*{Einleitung}
\addcontentsline{toc}{chapter}{Einleitung}
Es gibt viele verschiedene Kohomologietheorien über der Kategorie der Varietäten. Alle diese Theorien haben ähnliche Eigenschaften.  Grothendieck formulierte als erster die Idee, diese Theorien in gewisser Weise zu vereinheitlichen. Er träumte von einer einzigen "`Theorie kohomologischer Natur"', welche das "`Motiv"' für all die Kohomologietheorien beschreiben sollte. Er schreibt:\\
"`Anders als in der gewöhnlichen Topologie findet man sich [in der algebraischen Geometrie] mit einer beunruhigenden Fülle an Kohomologietheorien konfrontiert. Man hat deutlich den Eindruck (wenn auch in einem vagen Sinne), dass all diese Theorien `zum gleichen Ding' führen, `die selben Ergebnisse' liefern."' (\cite{G}, frei übersetzt)\\
Für einen festen Körper $ k $ wünschen wir uns also eine  Kategorie der \emph{gemischten Motive} über $ k $. Sie soll eine abelsche Tensorkategorie $ \mathcal{MM}(k) $ sein, zusammen mit einem kontravarianten "`universellen"' Kohomologiefunktor 
\[ Var(k)\longrightarrow \mathcal{MM}(k). \]
Der Funktor soll in dem Sinne universell sein, dass jede "`gute"' Kohomologietheorie über ihn faktorisiert. Unter einer guten Kohomologietheorie verstehen wir eine Kohomologietheorie, welche die Bloch-Ogus Axiome erfüllt \cite[Part 3]{MR1439046}.
Die Existenz einer solchen Kategorie ist eine offene Frage. Bisher gibt es nur einen möglichen Kandidaten für eine Kategorie der \emph{reinen Motive}, welche auf der vollen Unterkategorie der glatten, projektiven Varietäten eine universelle Kohomologietheorie für alle Weil-Kohomologien darstellen soll. Ein neuerer Versuch zur Konstruktion einer Kategorie der gemischten Motive kam Mitte der 90er Jahre von Nori. Zentral für den Ansatz ist folgendes Resultat:\\
\begin{Theorem}[Nori]
Sei $ D $ ein gerichteter Graph, $ R $ ein noetherscher, kommutativer, unitärer Ring und $R$-Mod die Kategorie der endlich erzeugten R-Moduln. Sei 
\[
T:D\longrightarrow R\Mod 
\]
eine Darstellung. \\
Dann existiert eine $R$-lineare abelsche Kategorie $ \mathcal{C}(T) $, genannt \emph{Diagrammkategorie}, mit einer Darstellung 
\[
\tilde{T}:D\longrightarrow\mathcal{C}(T) ,
\]
sowie ein treuer, exakter, $ R $-linearer Funktor $ ff_T $, sodass:
\begin{enumerate}
\item T durch $ D\overset{\tilde{T}}{\longrightarrow}\mathcal{C}(T)\xrightarrow{ff_T}R\Mod  $ faktorisiert.
\item $  \tilde{T} $  folgende universelle Eigenschaft erfüllt:
Gegeben eine weitere $R$-lineare, abelsche Kategorie $ \mathcal{A} $ mit einem $R$-linearen, treuen, exakten Funktor $ f:\mathcal{A}\rightarrow R\Mod $ und $ F:D\rightarrow \mathcal{A} $ einer Darstellung, die $ T=f\circ F $ erfüllt, so existiert ein bis auf Isomorphie eindeutiger Funktor $ L(F) $, so dass folgendes Diagramm bis auf Isomorphie kommutiert:
\end{enumerate}

\begin{center}
\begin{tikzpicture}[description/.style={fill=white,inner sep=2pt}]
    \matrix (m) [matrix of math nodes, row sep=5em,
    column sep=4.5em, text height=1.5ex, text depth=0.25ex]
    {  & \mathcal{C}(T) &  \\
       D & &  R\Mod  \\
 		& \mathcal{A} &  \\
	};
    \path[->,font=\scriptsize]
    (m-2-1) edge node[auto] {$ \tilde{T} $} (m-1-2)
    (m-2-1) edge node[auto] {$ F $} (m-3-2)
    (m-3-2) edge node[auto] {$f$} (m-2-3)
    (m-1-2) edge node[auto] {$ff_T$} (m-2-3)
    (m-2-1) edge node[pos=0.3,above right] {$T$} (m-2-3)
    (m-1-2) edge node[pos=0.4,above right] {$ L(F)$} (m-3-2);
\end{tikzpicture}
\end{center}
\end{Theorem}
Leider existiert für diese Aussage kein veröffentlichter Beweis, sondern nur eine dreiseitige Beweisskizze aus einem Seminar am TIFR \cite{Nori}. Hauptziel meiner Arbeit ist es, diese Lücke zu schließen und einen ausführlichen Beweis für dieses Theorem zu liefern.\\
Mit Hilfe des Theorems lässt sich die Kategorie der effektiven gemischten Nori-Motive konstruieren: Sei $ X $ eine Varietät über einem Körper $ k $ der Charakteristik Null, $ Y\subseteq X $ eine abgeschlossene Untervarietät und $ i $ eine natürliche Zahl. Wir definieren in Kapitel 7 einen Graphen $ D $ mit Ecken, bestehend aus Tripeln $ (X,Y,i) $, sowie geeigneten Kanten. Singuläre Kohomologie $ H^\bullet(X(\mathbb{C}),Y(\mathbb{C}),\mathbb{Q}) $ definiert dann eine Darstellung 
\[
\begin{array}{cccc}
T: & D & \longrightarrow & \mathbb{Q}\Vect \\
& (X,Y,i) & \longmapsto &  H^i(X(\mathbb{C}),Y(\mathbb{C}),\mathbb{Q})\\
\end{array}
\]
in die endlich-dimensionalen $ \mathbb{Q} $-Vektorräume. Die Kategorie der effektiven gemischten Nori-Motive $ \MM $ ist definiert als die Diagrammkategorie dieser Darstellung. Ob diese Kategorie universell für alle Kohomologietheorien ist, welche die Bloch-Ogus Axiome erfüllen, ist unklar. Die universelle Eigenschaft der Diagrammkategorie liefert uns jedoch treue, exakte Funktoren aus der Kategorie $ \MM $ in jede Kohomologietheorie, welche Vergleichsisomorphismen in die singuläre Kohomologie besitzt. Dies gilt unter anderem für $ l $-adische etale Kohomologie und algebraische de Rham Kohomologie.\\
\vspace*{0.2cm} \\
Ein weiteres Ziel dieser Arbeit ist es, die Diagrammkategorie dieser Darstellung besser zu verstehen. Hierfür zeigen wir:
\begin{Theorem}
Nimmt eine Darstellung $ T $ Werte in freien, endlich erzeugten Moduln über einem Hauptidealring $ R $ an, so ist die Diagrammkategorie äquivalent zu der Kategorie endlich erzeugter Moduln über einer Koalgebra $ A(T) $.
\end{Theorem}
Darüber hinaus lässt sich zeigen, dass die Kategorie $ \MM $ eine natürliche Struktur als abelsche Tensorkategorie besitzt. Sie wird dadurch zu einer neutralen Tannaka-Kategorie mit Faserfunktor in die endlich-dimensionalen $ \mathbb{Q} $-Vektorräume. Tannaka-Dualität nach Deligne \cite{MR654325} besagt dann, dass die Kategorie $ \MM $ äquivalent zu der Kategorie der Darstellungen eines affinen Gruppenmonoids ist. Genauer erhalten wir durch die Tensorstruktur auf $ \MM $ eine Hopf-Algebrenstruktur auf der Koalgebra $ A(T) $. Das Gruppenschema ist dann durch das Spektrum dieser Algebra gegeben. \\
\vspace{0.2cm}
\\
Nachdem wir in Kapitel 1 einige grundlegenden Begriffe und Konzepte klären, werden wir in Kapitel 2 Noris Diagrammkategorie konstruieren. Der Beweis der universellen Eigenschaft dieser Kategorie wird Kapitel 3 und 4 vollständig einnehmen. Diese Kapitel sind der Kern dieser Arbeit. Kapitel 5 klärt die Dualität zwischen Modulstrukturen über einer Algebra und Komodulstrukturen auf der dualen Koalgebra. Am Ende zeigen wir, dass die Diagrammkategorie $ \C(T) $ unter gewissen Voraussetzungen eine Kategorie von Komoduln über einer Koalgebra ist. Kapitel 6 enthält einen kurzen Exkurs über Tannaka-Dualität. In Kapitel 7 werden wir schließlich die Kategorie der effektiven, gemischten Nori-Motive $ \MM $ definieren und skizzieren, wie diese Kategorie durch Tannaka-Dualität zu den Darstellungen von $ \Spec A(T) $ korrespondiert. \\

\section*{Danksagung}
Mein Dank gilt in erster Linie Frau Prof. Huber-Klawitter, die mir mit ihrer außerordentlich guten Betreuung diese Diplomarbeit ermöglicht hat. Weiterer Dank für mathematische Hilfestellungen geht an Dr. Matthias Wendt. Für inhaltliche Diskussionen und Anregungen danke ich meinen Kommilitonen Jan Weidner, Konrad Völkel, David Stotz und Clemens Jörder. Für unzählige Hilfe mit LaTex danke ich Jochen Kiene und Konrad Völkel. Alex Koenen danke ich für die Korrektur des Manuskripts. Meiner Freundin Jana Ditter danke ich ganz herzlich für den moralischen Aufbau während der Entstehungszeit dieser Arbeit. Der Studienstiftung des deutschen Volkes danke ich für die ideelle und finanzielle Unterstützung während meines Studiums. 

\chapter{Grundlagen}
\label{ch:kap_1}

Wir beginnen mit einigen grundlegenden Definitionen und Resultaten.
\begin{Def}
Sei $ R $ ein unitärer, kommutativer Ring. Ein $ R $-Modul heißt \emph{noethersch}, wenn jede unendliche, per Inklusion aufsteigende Kette von Untermoduln stationär wird. Der Ring $ R $ heißt \emph{noethersch}, wenn er als Modul über sich selbst noethersch ist.
\end{Def}

\begin{Lemma}\label{noether}
Sei R ein noetherscher Ring und $ M $ ein endlich erzeugter $ R $-Modul. Dann  sind alle Untermoduln und Quotienten von $M$ endlich erzeugt. 
\end{Lemma}

\begin{proof}
Der Beweis ist elementar und steht in jedem Standardlehrbuch über kommutative Algebra, zum Beispiel \cite[S.\,76]{atiyahmacdonaldintroduction69}.
\end{proof}

\begin{Lemma}\label{hom_noethersch}
Sei $ R $ ein noetherscher Ring, sowie $ M $ und $ N $ endlich erzeugte $ R $-Moduln. Dann ist auch $ \Hom(M,N) $ als $ R $-Modul endlich erzeugt.
\end{Lemma}

\begin{proof}
Sei $ (v_1,..,v_m) $ ein Erzeugendensystem von $ M $ und $ (w_1,..w_n) $ ein Erzeugendensystem von $N$. Weiter seien $ R^m $ und $ R^n $ freie $ R $-Moduln und
\[
\begin{array}{ccccccccc}
 p_M: & R^m & \rightarrow & M & \hspace{1,5cm} &  p_N: & R^n & \rightarrow & N \\
      & e_i & \mapsto    & v_i &  &      & e_i & \mapsto    & w_i \\
\end{array}
\]
die natürlichen Projektionen. Sei

\[ A:=\{f\in\Hom(R^m,R^n)|\exists\tilde{f}\in\Hom(M,N):p_N\circ f=\tilde{f}\circ p_M\}, \]

also die Menge aller Abbildungen $ f $, für die wir ein $ \tilde{f} $ finden, so dass folgendes Diagramm kommutiert:

\begin{center}
\begin{tikzpicture}[description/.style={fill=white,inner sep=2pt}]
    \matrix (m) [matrix of math nodes, row sep=3.0em,
    column sep=5.0em, text height=1.5ex, text depth=0.25ex]
    { R^m &  M  \\
   	   R^n & N  \\
	};
    \path[->,font=\scriptsize]
    (m-1-1) edge node[description] {$ p_M $} (m-1-2)
    (m-2-1) edge node[description] {$ p_N $} (m-2-2)
    (m-1-1) edge node[description] {$ f $} (m-2-1)
    (m-1-2) edge node[description] {$ \tilde{f} $} (m-2-2)
    ;   
\end{tikzpicture}
\end{center}

Da $ \Hom(M,N) $ ein Modul ist, ist $ A $ abgeschlossen unter Addition und $ R $-Multiplikation, also selbst ein Modul. Die Abbildung $ p_M $ ist surjektiv, also gibt es für jedes $ f\in A $ genau ein $ \tilde{f} $, so dass das Diagramm kommutiert. Wir erhalten also einen wohldefinierten \\$ R $-Modulhomomorphismus 
\[
\begin{array}{ccc}
 A & \twoheadrightarrow & \Hom(M,N) \\
 f & \mapsto &   \tilde{f}. \\
\end{array}
\]
$ A $ ist ein Untermodul des freien, endlich erzeugten Moduls $ \Hom(R^m,R^n)\cong R^{m\cdot n} $, also nach Lemma \ref{noether} selbst endlich erzeugt.\\
Sei $ (f_1,..f_k) $ ein Erzeugendensystem von $ A $. Dann ist $ (\tilde{f}_1,..\tilde{f}_k) $ ein Erzeugendensystem von $ \Hom(M,N) $, also ist auch $ \Hom(M,N) $ endlich erzeugt. 
\end{proof}

Wir arbeiten bis auf weiteres mit folgender Definition für Algebren aus \cite[S.121]{MR1878556}.
\begin{Def}\label{algebra1}
Sei $ R $ ein kommutativer, unitärer Ring. Eine \emph{$ R $-Algebra} ist ein Ring $ A $, zusammen mit einem Ringhomomorphismus $ f:R\rightarrow A $, so dass das Bild des Ringhomomorphismus im Zentrum $ Z(A) $ von $ A $ liegt.\\
Ist A unitär, so fordert man, dass $ \varphi $ ein unitärer Ringhomomorphismus ist und wir sprechen von einer \emph{unitären Algebra}.
Ein \emph{Algebrenhomomorphismus} zwischen zwei $ R $-Algebren $ f_1:R\rightarrow A_1 $, $ f_2:R\rightarrow A_2 $ ist ein $ R $-linearer Ringhomomorphismus $ \varphi:A_1\rightarrow A_2 $, also ein Ringhomomorphismus, so dass folgendes Diagramm kommutiert. 

\begin{center}
\begin{tikzpicture}[description/.style={fill=white,inner sep=2pt}]
    \matrix (m) [matrix of math nodes, row sep=3.0em,
    column sep=2.0em, text height=1.5ex, text depth=0.25ex]
    { A_1 & & A_2  \\
   	  & R &   \\
	};
    \path[->,font=\scriptsize]
    (m-1-1) edge node[description] {$ \varphi $} (m-1-3)
    (m-2-2) edge node[description] {$ f_1 $} (m-1-1)
    (m-2-2) edge node[description] {$ f_2 $} (m-1-3)
    ;   
\end{tikzpicture}
\end{center}

\end{Def}

\begin{Def}
Eine $ R $-Algebra heißt \emph{endlich}, wenn sie als $ R $-Modul endlich erzeugt ist. Eine $ R $-Algebra $ A $ heißt \emph{endlich erzeugt}, wenn sie als Algebra endlich erzeugt ist, das heißt, wenn es eine Surjektion 
\[ R[X_1,...,X_n]\twoheadrightarrow A \]
von $ R $-Algebren gibt.
\end{Def}

\begin{Satz}[Hilbertscher Basissatz] Der Polynomring $ R[x] $ über einem noetherschen Ring ist noethersch.
\end{Satz}

\begin{proof}
Siehe z.B. \cite[S.27]{MR1322960}
\end{proof}

\begin{Korollar}\label{noethersch_Algebra}
Jede endliche $R$-Algebra, sowie jede endlich erzeugte $ R $-Algebra über einem noetherschen Ring $ R $, ist als Ring wieder noethersch. 
\end{Korollar}

\begin{proof}
Eine endliche $ R $-Algebra ist insbesondere endlich erzeugt. Sei also $ A $ eine endlich erzeugte $ R $-Algebra. Sei $ (a_1\dots a_n) $ ein Erzeugendensystem von $ A $ als $ R $-Algebra. Dann erhalten wir eine natürliche Projektion $ R[a_1\dots a_n]\xrightarrow{\pi}A $, also einen Isomorphismus $ A \cong {R[a_1\dots a_n]}/{I}$, wobei $ I=ker(\pi) $ ein Ideal von $ R[a_1\dots a_n] $ ist. \\
Die Algebra $ R[a_1\cdots a_n] $ ist nach induktivem Anwenden des hilbertschen Basissatzes noethersch, also nach Lemma \ref{noether} auch $ {R[a_1\dots a_n]}/{I} $ und damit auch $ A $.
\end{proof}

\begin{Def}
\cite[I.8]{MR0354799} Eine Kategorie $ \mathcal{C} $ heißt \emph{präadditiv} oder auch \emph{Ab-Kategorie}, wenn sie die folgenden zwei Bedingungen erfüllt: 
\begin{enumerate}
\item Für zwei Objekte $ A,B\in \mathcal{C} $ hat die Morphismenmenge $ \Hom_{\mathcal{C}}(A,B) $ die Struktur einer abelschen Gruppe.
\item Die Verkettung von Morphismen ist bilinear bezüglich der Gruppenstruktur.
\end{enumerate}
\end{Def}

\begin{Def}\label{abKat}
Eine Kategorie $ \mathcal{C} $ heißt \emph{abelsch}, wenn sie die folgenden Bedingungen erfüllt: 
\begin{enumerate}
\item Es gibt ein Objekt, das sowohl initial als auch terminal ist, genannt Nullobjekt. 
\item Endliche Produkte und Koprodukte existieren.
\item Jeder Morphismus besitzt Kern und Kokern.
\item Jeder Monomorphismus ist ein Kern, jeder Epimorphismus ist ein Kokern.
\end{enumerate}
\end{Def}

\begin{Bem}
Viele Autoren fordern für eine abelsche Kategorie auch die Präadditivität. Tatsächlich kann man aber zeigen, dass sich die Gruppenstruktur und die Bilinearität der Morphismen aus den hier angegebenen Axiomen gewinnen lassen. Eine abelsche Kategorie ist also immer präadditiv \cite[S.\,45\,ff, Kap.\,2.3]{Freyd1964}.\\
Außerdem kann man zeigen, dass in abelschen Kategorien endliche Produkte und Koprodukte kanonisch isomorph sind \cite[Thm.2.35]{Freyd1964}.
\end{Bem}

\begin{Bsp}
Beispiele für abelsche Kategorien sind
\begin{itemize}
\item Die Kategorie der abelschen Gruppen
\item Die Kategorie der Vektorräume über einem Körper $ K $
\item Die Kategorie der Moduln über einem Ring $ R $. 
\end{itemize}
Die Kategorie der endlich erzeugten Moduln über einem beliebigen Ring ist nicht abelsch, da Kerne und Kokerne von endlich erzeugten Moduln nicht endlich erzeugt sein müssen. Für noethersche Ringe gilt dies:
\end{Bsp}

\begin{Lemma}\label{noethersch_kategorie}
Endlich erzeugte Moduln über noetherschen Ringen $ R $ bilden eine abelsche Kategorie.
\end{Lemma}

\begin{proof}
Da beliebige Moduln über $ R $ eine abelsche Kategorie bilden, gelten Axiome (1) und (4) auch für die Unterkategorie der endlich erzeugten Moduln. Endliche Produkte und Summen endlich erzeugter Moduln sind endlich erzeugt, also gilt (3). Da Kerne Untermoduln und Kokerne Quotienten sind, gilt nach Lemma \ref{noether}, dass sie endlich erzeugt, also ebenfalls Objekte der Unterkategorie der endlich erzeugten Moduln sind. 
\end{proof}

\begin{Def}
Eine präadditive Kategorie heißt \emph{$ R $-linear}, wenn zu je zwei Objekten $ A,B\in\mathcal{C} $ die Gruppe $ \Hom_{\mathcal{C}}(A,B) $ ein $ R $-Modul ist und die Verkettung von Morphismen $ R $-bilinear ist.
\end{Def}

\begin{Def}
Seien $ \mathcal{C} $, $ \mathcal{D} $ Kategorien und sei $ F:\mathcal{C}\rightarrow\mathcal{D} $ ein Funktor. Seien $ A $ und $ B $ Objekte aus $ \C $. Der Funktor $ F $ heißt

	\begin{itemize}
	\item \emph{treu}, wenn die induzierte Abbildung $ \Hom_{\mathcal{C}}(A,B)\rightarrow \Hom_{\mathcal{D}}(F(A),F(B)) $ injektiv ist.
	\item \emph{voll}, wenn die induzierte Abbildung $ \Hom_{\mathcal{C}}(A,B)\rightarrow \Hom_{\mathcal{D}}(F(A),F(B)) $ surjektiv 				ist.
	\item \emph{volltreu}, wenn die induzierte Abbildung $ \Hom_{\mathcal{C}}(A,B)\rightarrow \Hom_{\mathcal{D}}(F(A),F(B)) $ bijektiv 			ist.
	\end{itemize}

Sind $ \mathcal{C} $, $ \mathcal{D} $ abelsche Kategorien, so heißt $ F $ 

	\begin{itemize}
	\item \emph{additiv}, wenn die induzierte Abbildung $ \Hom_{\mathcal{C}}(A,B)\rightarrow \Hom_{\mathcal{D}}(F(A),F(B)) $ ein Gruppenhomomorphismus ist.
	\item \emph{exakt}, wenn $ F $ kurze exakte Sequenzen auf kurze exakte Sequenzen abbildet, also Kerne und Kokerne erhält.
	\end{itemize}
	
Sind $ \mathcal{C} $, $ \mathcal{D} $ $ R $-linear und abelsch, so nennen wir $ F $

	\begin{itemize}
	\item \emph{$ R $-linear}, wenn die induzierte Abbildung $ \Hom_{\mathcal{C}}(A,B)\rightarrow \Hom_{\mathcal{D}}(F(A),F(B)) $ ein 		$ R $-Modulhomomorphismus ist und die Verkettung von Morphismen $ R $-bilinear ist. 
	\end{itemize}
	
\end{Def}

\begin{Proposition}
Ein Funktor $ \mathcal{F}:\C\rightarrow\mathcal{D} $ zwischen abelschen Kategorien ist genau dann additiv, wenn er direkte Summen erhält, also wenn für zwei beliebige Objekte $ X,Y\in\C $ gilt
\[
\mathcal{F}(X\oplus Y)\cong\mathcal{F}(X)\oplus \mathcal{F}(Y).
\]
\end{Proposition}

\begin{proof}
\cite[Thm.3.11]{Freyd1964}
\end{proof}

\begin{Lemma}
Ein exakter Funktor zwischen abelschen Kategorien ist additiv.
\end{Lemma}

\begin{proof}
Sei $ \mathcal{F}:\C\rightarrow\mathcal{D} $ exakt. Für zwei Objekte $ X,Y\in\C $ haben wir die kanonische, spaltende kurze exakte Sequenz
\[
0\rightarrow X\hookrightarrow X\oplus Y\twoheadrightarrow Y\rightarrow 0.
\]
Durch Anwenden von $ \mathcal{F} $ erhalten wir die spaltende kurze exakte Sequenz
\[
0\rightarrow \mathcal{F}(X)\hookrightarrow\mathcal{F}(X\oplus Y)\twoheadrightarrow \mathcal{F}(Y)\rightarrow 0.
\]
Also ist $ \mathcal{F}(X\oplus Y)=\mathcal{F}(X)\oplus \mathcal{F}(Y) $.

\end{proof}
\newpage
\begin{Lemma}\label{Nullobjekt}
Ein additiver Funktor zwischen abelschen Kategorien bildet Nullobjekte auf Nullobjekte ab.
\end{Lemma}

\begin{proof}
Wir definieren für jedes Objekt $A$ den Nullmorphismus als den eindeutigen Endomorphismus von $A$, der über $ A\rightarrow 0\rightarrow  A$ faktorisiert. Er ist das neutrale Element der Gruppe $ \Hom(A,A) $, da $ \Hom(A,0) $ und $ \Hom(0,A) $ triviale Gruppen sind und die Verkettung von Morphismen bilinear ist.\\
Wir zeigen zunächst, dass ein Objekt genau dann Null ist, wenn die Identitätsabbildung gleich der Nullabbildung ist.\\
Da das Nullobjekt nur einen Endomorphismus besitzt, ist der Nullmorphismus schon die Identität auf dem Nullobjekt. Sei umgekehrt $A$ ein Objekt mit $ id_A=0_{\End(A)} $. Sei $ A\xrightarrow{f} B $ ein beliebiger Morphismus. Dann gilt 
\[ f=f\circ id=f\circ 0_{\End(A)}.\] 
Der Morphismus $ f $ faktorisiert also über 
\[ A\rightarrow 0\rightarrow B,\]
ist demnach eindeutig, da die Null initial und terminal ist. Analog zeigt man, dass es nur einen Morphismus $ B\xrightarrow{f} A  $ gibt. $A$ ist also das Nullobjekt.\\
Additive Funktoren induzieren Gruppenhomomorphismen auf Morphismenmengen. Der eindeutige Morphismus in $\End(0)$ muss unter einem additiven Funktor $ F $ also sowohl auf die Identität, als auch auf die Nullabbildung abgebildet werden. Damit ist in $\End(F(0))$ die Identität gleich der Nullabbildung, also ist $ F(0) $ das Nullobjekt.
\end{proof}

\begin{Def}\label{Ind_Kat}
Sei $ I $ eine partiell geordnete Indexmenge. Dann definiert $ I $ auf natürliche Weise eine Kategorie $ \mathcal{I} $: 
\begin{itemize}
\item $Ob(\mathcal{I}):=I$
\item $\Hom_{\mathcal{I}}(i,j)$ besitzt genau ein Element, wenn $ i\le j $ ist und ist sonst leer. Da $ \le $ transitiv ist, erhalten wir eine wohldefinierte Verkettung von Morphismen.
\end{itemize}
\end{Def}

\begin{Def}
Eine partielle Ordnung auf $ I $ heißt \emph{filtrierend} oder \emph{gerichtet}, wenn es für zwei beliebige Elemente $ i,j\in I $ eine gemeinsame obere Schranke $ k\in I $ gibt, das heißt, wenn es ein $ k\in I $ gibt, sodass $ i\le k $ und $ j\le k $.
\end{Def}

\begin{Bem+Def}
Sei $ I $ eine partiell geordnete Menge und $ \mathcal{I} $ die zugehörige Kategorie. Ein kovarianter Funktor
\[
\begin{array}{cccc}
F: & \mathcal{I} & \rightarrow & \mathcal{C} \\
& i & \mapsto & X_i
\end{array}
\]
ist gegeben durch eine eine Familie $ (X_i)_{i\in I} $ von Objekten in $ \mathcal{C} $ zusammen mit Übergangsmorphismen $ \phi_{ji}:X_i\rightarrow X_j $ für alle $ i\le j $, sodass für $ i\le j \le k $ gilt 
\[ \phi_{kj}\circ\phi_{ji}=\phi_{ki}. \] Eine solche Familie, zusammen mit Übergangsabbildungen nennen wir ein \emph{induktives System} in $\mathcal{C}$. Ist $ I $ eine filtrierende Indexmenge, so nennen wir $ (\{ X_i\}_{i\in I},\phi_{ji}) $ ein \emph{gerichtetes} oder \emph{filtrierendes System}.
\end{Bem+Def}

\begin{Def}
Sei $ \C $ eine Kategorie und $ I $ eine partiell geordnete Indexmenge. Sei \\
$ (\{ X_i\}_{i\in I},\phi_{ji}) $ ein induktives System in $ \C $ bezüglich $ I $. Der \emph{direkte Limes} oder auch \emph{Kolimes} des induktiven Systems $ (\{ X_i\}_{i\in I},\phi_{ji}) $ in $ \C $ besteht aus einem Objekt $ L $, zusammen mit Morphismen $ \psi_i\in \Hom_{\C}(X_i,L) $ für alle $ i\in I $, sodass
\begin{enumerate}
\item Für alle $ i,j$ aus $ I $ mit $ i\le j $ gilt: $ \psi_j\circ \phi_{ji}=\psi_i. $ 
\item Folgende universelle Eigenschaft erfüllt ist: Für jedes weitere Objekt $ M $ mit Morphismen $(\chi_i)_{i\in I}\subseteq \Hom_{\C}(X_i,M)$, welche die Eigenschaft 1 erfüllen, existiert ein eindeutiger Morphismus $ u:L\rightarrow M $, sodass für alle $ i\in I$ die Komposition $u\circ\psi_i=\chi_i $ ist.
\end{enumerate}
Der direkte Limes wird üblicherweise mit $ \lim\limits_{\rightarrow I}X_i  $ notiert.
\end{Def}

\begin{Bem}
Für ein induktives System 

\begin{center}
\begin{tikzpicture}[description/.style={fill=white,inner sep=2pt}]
    \matrix (m) [matrix of math nodes, row sep=2.5em,
    column sep=5.5em, text height=1.5ex, text depth=0.25ex]
    {  \cdots X_i & & X_j \cdots \\
	};
    \path[->,font=\scriptsize]
    (m-1-1) edge node[auto] {$ \phi_{ji} $} (m-1-3)
    ;   
\end{tikzpicture}
\end{center}

lässt sich die zweite Bedingungen durch die Existenz des folgenden kommutativen Diagramms formulieren:

\begin{center}
\begin{tikzpicture}[description/.style={fill=white,inner sep=2pt}]
    \matrix (m) [matrix of math nodes, row sep=3.5em,
    column sep=5.5em, text height=1.5ex, text depth=0.25ex]
    { \cdots X_i & & X_j \cdots \\
    	      & L &   \\
   	   	      &  M &  \\
	};
    \path[->,font=\scriptsize]
    (m-1-1) edge node[description] {$ \phi_{ji} $} (m-1-3)
    (m-1-1) edge node[description] {$ \psi_{i} $} (m-2-2)
    (m-1-1) edge node[description] {$ \chi_{i} $} (m-3-2)
    (m-1-3) edge node[description] {$ \psi_{j} $} (m-2-2)
    (m-1-3) edge node[description] {$ \chi_{j} $} (m-3-2)
    (m-2-2) edge node[description] {$\exists ! u $} (m-3-2)
    ;   
\end{tikzpicture}
\end{center}

\end{Bem}

\begin{Bem}
Der direkte Limes muss nicht in jeder Kategorie existieren. Existiert er jedoch, kann man mit dem üblichen Argument zeigen, dass er eindeutig bis auf eindeutigen Isomorphismus ist. 
\end{Bem}

\chapter{Konstruktion von Noris Diagrammkategorie}
\label{ch:kap_2}

\begin{Def}
\begin{itemize}
\item Ein \emph{Diagramm} $ D $ besteht aus einer Menge von \emph{Objekten} $ O(D) $, sowie für jedes Paar $ (p,q) $ von Objekten einer Morphismenmenge $ M(p,q) $. Im Unterschied zu einer Kategorie ist in Diagrammen also a priori die Verknüpfung von Morphismen nicht erklärt.
\item Ein endliches Diagramm ist ein Diagramm mit nur endlich vielen Objekten.
\item Ein volles Unterdiagramm eines Diagramms $ D $ ist ein Diagramm $ F $ mit $ O(F)\subseteq O(D) $ sodass für alle Objekte $ p,q\in O(F) $ die Morphismenmenge $ M(p,q) $ in $ F $ der Morphismenmenge $ M(p,q) $ in $ D $ entspricht.
\end{itemize}
\end{Def}

\begin{Bem}
Ein Diagramm ist nichts anderes als ein gerichteter Graph. Wir werden trotzdem, konsistent zu den Quellen \cite{paper}, \cite{MR2182598} und \cite{Nori} den Begriff "`Diagramm"' verwenden. 
\end{Bem}

\begin{Def}
Sei $ \mathcal{C} $ eine Kategorie. Eine \emph{Darstellung} $ F $ eines Diagramms $ D $ besteht aus einer Abbildung $ F:O(D)\rightarrow Ob(\mathcal{C}) $ sowie für alle $ p,q\in O(D) $ einer Abbildung $ F:M(p,q)\rightarrow \Hom(F(p),F(q)) $.\\
Ist $ \C $ keine Kategorie, sondern nur ein Diagramm so sprechen wir von einer \emph{Abbildung von Diagrammen}.
\end{Def}

Ziel dieses und der beiden folgenden Kapitel ist der Beweis des folgenden Theorems:

\begin{Theorem}[Nori]\label{thm1}
Sei $ D $ ein Diagramm, $ R $ ein noetherscher, kommutativer, unitärer Ring und $R$-Mod die Kategorie der endlich erzeugten $R$-Moduln. Sei 
\[
T:D\longrightarrow R\Mod 
\]
eine Darstellung. \\
Dann existiert eine $R$-lineare, abelsche Kategorie $ \mathcal{C}(T) $ mit einer Darstellung 
\[
\tilde{T}:D\longrightarrow\mathcal{C}(T) ,
\]
sowie ein treuer, exakter, $ R $-linearer Funktor $ ff_T $, sodass:
\begin{enumerate}
\item T durch $ D\overset{\tilde{T}}{\longrightarrow}\mathcal{C}(T)\xrightarrow{ff_T}R\Mod  $ faktorisiert.
\item $  \tilde{T} $  folgende universelle Eigenschaft erfüllt:
Gegeben eine weiter $R$-lineare, abelsche Kategorie $ \mathcal{A} $ mit einem R-linearen, treuen, exakten Funktor $ f:\mathcal{A}\rightarrow R $-Mod und $ F:D\rightarrow \mathcal{A} $ einer Darstellung, die $ T=f\circ F $ erfüllt, so existiert ein bis auf Isomorphie eindeutiger Funktor $ L(F) $, so dass folgendes Diagramm bis auf Isomorphie kommutiert:
\end{enumerate}

\begin{center}
\begin{tikzpicture}[description/.style={fill=white,inner sep=2pt}]
    \matrix (m) [matrix of math nodes, row sep=5em,
    column sep=4.5em, text height=1.5ex, text depth=0.25ex]
    {  & \mathcal{C}(T) &  \\
       D & &  R\Mod  \\
 		& \mathcal{A} &  \\
	};
    \path[->,font=\scriptsize]
    (m-2-1) edge node[auto] {$ \tilde{T} $} (m-1-2)
    (m-2-1) edge node[auto] {$ F $} (m-3-2)
    (m-3-2) edge node[auto] {$f$} (m-2-3)
    (m-1-2) edge node[auto] {$ff_T$} (m-2-3)
    (m-2-1) edge node[pos=0.3,above right] {$T$} (m-2-3)
    (m-1-2) edge node[pos=0.4,above right] {$ L(F)$} (m-3-2);
\end{tikzpicture}
\end{center}

\end{Theorem}

\begin{Bem}
Existiert eine solche Kategorie, so ist sie eindeutig bis auf Äquivalenz: Sei $ \C'(T) $ eine weitere Kategorie, welche die universelle Eigenschaft erfüllt, so haben wir treue, exakte $ R $-lineare Funktoren $ \C(T)\underset{L'(F)}{\overset{L(F)}{\rightleftarrows}} \C'(T) $, deren Verkettung ein treuer, exakter, $ R $-linearer Funktor 
\[ L'(F)\circ L(F):\C(T)\longrightarrow\C(T) \]
ist. Ein solcher Funktor ist nach universeller Eigenschaft von $ \C(T) $ eindeutig bis auf Isomorphie, entspricht also bis auf natürliche Isomorphie der Identität. Gleiches gilt für $ L(F)\circ L'(F) $. Die Kategorien $ \C(T) $ und $ \C'(T) $ sind also äquivalent.
\end{Bem}

\section[Konstruktion für endliche Diagramme]{Konstruktion für endliche Diagramme}

Wir zeigen zunächst 1. aus Theorem \ref{thm1} für endliche Diagramme:

\begin{Proposition}\label{thm1_endlich}
Sei $ D $ ein endliches Diagramm und $ T: D\longrightarrow R\Mod $ eine Darstellung. Dann existiert eine $R$-lineare abelsche Kategorie $ \mathcal{C}(T) $, zusammen mit einer Darstellung 
\[
 \tilde{T}:D\longrightarrow\mathcal{C}(T), 
\]
sowie ein treuer, exakter Funktor $ ff_T $,  sodass $T$ durch 
\[
 D\overset{\tilde{T}}{\longrightarrow}\mathcal{C}(T)\overset{ff_T}{\longrightarrow}R\Mod  
\]
faktorisiert.
\end{Proposition}

\begin{proof}
Der Beweis basiert auf der Ausführung in \cite[1.2]{Nori}.
Für $ p\in O(D) $ ist $ Tp $ ein endlich erzeugter R-Modul. $ \End_R(Tp) $ ist dann nach Lemma \ref{hom_noethersch} ebenfalls ein endlich erzeugter $R$-Modul. Mit Hintereinanderausführung als innere Multiplikation wird $ \End_R(Tp) $ eine endlich erzeugte $ R $-Algebra. Wir definieren 
\[
\End(T):=\big \{ e \in \prod_{p \in Ob(D)}\End_R(T_p)\big | \forall p,q \in  O(D), \forall m \in M(p,q): Tm \circ e_p=e_q \circ Tm \big \} ,
\]
wobei $ e_p, e_q $ Projektion von $ e $ in die entsprechende Komponente 
darstellen. \\
In Diagrammform gesprochen besteht ein Element in $ \End(T) $ aus Tupeln $(e_p)_{p\in O(D)}$ von Endomorphismen aus $ \End_R(Tp) $, sodass alle entsprechenden Diagramme der Form

\begin{center}
\begin{tikzpicture}[description/.style={fill=white,inner sep=2pt}]
    \matrix (m) [matrix of math nodes, row sep=3em,
    column sep=2.5em, text height=1.5ex, text depth=0.25ex]
    {  T(p) & & T(q) \\
      T(p) & & T(q) \\
	};
    \path[->,font=\scriptsize]
    (m-1-1) edge node[auto] {$ e_{p_1} $} (m-2-1)
    (m-1-3) edge node[auto] {$ e_{p_2} $} (m-2-3)
    (m-1-1) edge node[auto] {$T(m)$} (m-1-3)
    (m-2-1) edge node[auto] {$T(m)$} (m-2-3);
\end{tikzpicture}
\end{center}

kommutieren. Die Menge $ \prod_{p \in Ob(D)}\End_R(Tp)$ hat mit komponentenweiser Verknüpfung die Struktur einer Produktalgebra. Da $\End(T)$ abgeschlossen ist unter $ (+)$, komponentenweiser Verkettung und Multiplikation mit Elementen aus $ R $, bildet $ \End(T) $ eine $R$-Unteralgebra. Da die Identität mit allen Morphismen kommutiert, ist $ \End(T) $ eine unitäre Algebra. $ O(D) $ ist endlich, also ist  $  \prod_{p \in Ob(D)}\End_R(Tp) $ als $ R $-Modul endlich erzeugt. Da $ R $ noethersch ist, ist nach Lemma \ref{noether} auch die Unteralgebra \[ \End(T) \subset \prod_{p \in Ob(D)}\End_R(Tp) \]endlich erzeugt.\\
Wir können für alle $ p\in O(D) $ den $R$-Modul $ Tp $ als $\End(T)$-Linksmodul auffassen, denn die Projektion 
\[ pr:\End(T)\rightarrow \End_R(Tp) \]
induziert eine Wirkung von $\End(T)$ auf $ Tp $, welche die Modulaxiome erfüllt. Genauer gilt für zwei beliebige Elemente $ (e_p)_{p\in O(D)} $, $(e'_p)_{p\in O(D)}$ aus $ \End(T) $, sowie zwei beliebige Element $ a,b $ aus $ Tp $

\begin{align*}
(e_p\circ e'_p)(a) &=e_p(e'_p(a)) \quad \quad &\text{(Assoziativität)} \\
e_p(Tp)(a+b) &=e_p(a)+e_p(b) \quad \quad &\text{(Distributivität)}
\end{align*}

Das Distributivgesetz gilt, da $ e_p $ insbesondere ein $ R $-Modulhomomorphismus ist.\\
Da $ Tp $ als $ R $-Modul endlich erzeugt ist, gilt dies auch für $ Tp $ als $ \End(T) $-Linksmodul.\\
Wir setzen nun 
\[ \mathcal{C}(T):=\End(T)\Mod ,\]
die Kategorie der endlich erzeugten $\End(T)$-Linksmoduln. Da $\End(T)$ eine endlich erzeugte $ R $-Algebra ist, ist $ \End(T) $ nach Korollar \ref{noethersch_Algebra} noethersch, also ist $\End(T)\Mod$ nach Korollar \ref{noethersch_kategorie} eine abelsche Kategorie.\\
Die Zuordnung 
\[
\begin{array}{ccc}
 D & \rightarrow & \End(T)\Mod \\
 p & \mapsto & Tp\text{ als End(}\textit{T}\text{)-Linksmodul}
\end{array}
\]
soll nun eine Darstellung liefern. Um das zu sehen, müssen wir prüfen, ob für alle $p,q\in O(D), m\in M(p,q) $ der Morphismus $ Tm\in \Hom_R(Tp,Tq) $ nicht nur ein Homomorphismus von $ R $-Moduln, sondern auch ein Homomophismus von $ \End(T)$-Linksmoduln ist, also ob die Abbildung $ Tm $ auch $ \End(T) $-linear ist. Dies entspricht aber gerade der Kommutativitätsbedingung des obigen Diagramms: Für $ e\in \End(T) $ gilt 
\[ Tm(e.Tp)=(Tm\circ e_p)(Tp)=(e_q\circ Tm)(Tp)=e.(Tm(Tp)). \] 
Sei $ ff_T:\mathcal{C}(T)\rightarrow R\Mod $ der Vergissfunktor, der die $ \End(T) $-Struktur vergisst. Dann faktorisiert $ T $ durch
\[
\begin{array}{ccccc}
D & \xrightarrow{\tilde{T}} & \mathcal{C}(T) & \xrightarrow{ff_T} & R\Mod \\
p & \mapsto & Tp & \mapsto & Tp
\end{array}
\]
Der Vergissfunktor ist exakt, treu und $ R $-linear.
\end{proof}

\section[Konstruktion für beliebige Diagramme]{Konstruktion für beliebige Diagramme}

Der Beweis der Proposition lässt sich in der Form nicht ohne Weiteres für beliebige (nicht endliche) Diagramme $ D $ durchführen. Im Fall beliebig großer Objektmengen $ O(D) $ ist  $\prod_{p \in Ob(D)}\End_R(Tp)$ nicht mehr endlich erzeugt, also ist $ \End(T) $ nicht notwendigerweise noethersch, $ \End(T)\Mod $ und damit nicht unbedingt eine abelsche Kategorie. Außerdem wäre eine derartige Kategorie zu groß, um die universelle Eigenschaft zu erfüllen.\\
Wir müssen daher unser Diagramm durch endliche, bezüglich Inklusion aufsteigende Unterdiagramme ausschöpfen: Für jedes endliche, volle Unterdiagramm $ F\subset O(D) $ erhalten wir nach Proposition \ref{thm1_endlich} eine Faktorisierung 
\[
F \rightarrow\End(T_F)\Mod\rightarrow R\Mod. 
\]

Nori folgend möchten wir $ \C (T):=\varinjlim_{F\subset D\text{ endlich}}\End(T_F)\Mod $ setzen und zeigen, dass dieses Objekt die universelle Eigenschaft erfüllt.

Wir beweisen nun einige Aussagen, die uns später garantieren werden, dass \\
$\varinjlim_{F\subset D\text{ endlich}}\End(T_F)\Mod   $ eine abelsche Kategorie ist.

\begin{Proposition}\label{Freyd}
Sei $ \mathcal{F}:\C\longrightarrow\mathcal{D} $ ein additiver Funktor zwischen abelschen Kategorien. Dann sind folgende Aussagen äquivalent: \vspace*{0.5cm}
\begin{itemize}
\item Der Funktor $ \mathcal{F} $ ist treu.
\item Ist eine Sequenz $ A\rightarrow B\rightarrow C $ in $ \mathcal{C} $ nicht exakt, so ist auch $ \mathcal{F}(A)\rightarrow \mathcal{F}(B)\rightarrow \mathcal{F}(C) $ in $ \mathcal{D} $ nicht exakt.
\end{itemize}
\end{Proposition}
\begin{proof}
\cite[Thm.3.21]{Freyd1964}
\end{proof}

\begin{Korollar}\label{treu_exakt}
Seien $ \C $, $ \mathcal{D} $ abelsche Kategorien, $ F:\C\rightarrow \mathcal{D} $ ein kovarianter, exakter, treuer Funktor. Sei $ f:X\longrightarrow Y $ ein Morphismus in $ \C $. Dann ist\\
\begin{itemize}
\item $ K\xrightarrow{k}X $ ein Kern von $ X\xrightarrow{f}Y $ genau dann, wenn $ F(K)\xrightarrow{F(k)}F(X) $ ein Kern von $ F(X)\xrightarrow{F(f)}F(Y) $ ist. 
\item $ Y\xrightarrow{cok}C $ ein Kokern von $ X\xrightarrow{f}Y $ genau dann, wenn $ F(Y)\xrightarrow{F(cok)}F(C) $ ein Kokern von $ F(X)\xrightarrow{F(f)}F(Y) $ ist. 

\end{itemize}
\end{Korollar}

\begin{proof}
Die Aussage, dass  $ K\xrightarrow{k}X $ ein Kern und $ Y\xrightarrow{cok}C $ ein Kokern von $ X\xrightarrow{f}Y $ sind, ist äquivalent dazu, dass folgende Sequenz exakt ist:\\
\[0\longrightarrow K\overset{x}{\longrightarrow}X\overset{f}{\longrightarrow}Y\overset{cok}{\longrightarrow}C \longrightarrow 0. \] 
Da $ F $ exakt und treu ist, gilt dies nach Proposition \ref{Freyd} genau dann, wenn
\[0\rightarrow F(K)\xrightarrow{F(k)}F(X)\xrightarrow{F(f)}F(Y)\xrightarrow{F(cok)}F(C) \rightarrow 0 \] 
exakt ist. Dies ist wiederum äquivalent dazu , dass $ F(K)\xrightarrow{F(k)}F(X) $ ein Kern und $ F(Y)\xrightarrow{F(cok)}F(C) $ ein Kokern von $ F(X)\xrightarrow{F(f)}F(Y) $ ist.\\
\end{proof}

\begin{Korollar}\label{iso}
Sei $ \mathcal{F}:\C\longrightarrow\mathcal{D} $ ein treuer, exakter Funktor zwischen abelschen Kategorien. Dann ist $ f:X\rightarrow Y $ genau dann ein Isomorphismus in $ \C $, wenn $ F(f):F(X)\rightarrow F(Y) $ ein Isomorphismus in $ \mathcal{D} $ ist. 
\end{Korollar}

\begin{proof}
$ f $ ist ein Isomorphismus genau dann, wenn 
\[ 0\rightarrow X \xrightarrow{f}Y\rightarrow 0 \]
exakt ist, also genau dann, wenn
\[ 0\rightarrow F(X) \xrightarrow{F(f)}F(Y)\rightarrow 0 \]
exakt ist. Dies ist äquivalent dazu, dass $ F(f) $ ein Isomorphismus ist.
\end{proof}

\begin{Proposition}\label{Kategorienproposition}
Sei $ (\{\mathcal{C}_i\}_{i\in I}, \phi_{i,j})$ ein gerichtetes, filtriertes System $ R $-linearer, abelscher Kategorien, wobei $ \phi_{i,j} $ $ R$-lineare, treue, exakte Funktoren sind. Dann ist $ \C:= \lim\limits_{\rightarrow I}\C_i $ eine $ R $-lineare, abelsche Kategorie.
\end{Proposition}

\begin{proof}
Wir müssen eine Kategorie $ \C $ konstruieren, von der wir zeigen können, dass sie abelsch und $ R $-linear ist und dass sie die universelle Eigenschaft des direkten Limes erfüllt. Die Konstruktion von $ \C $ entspricht der Konstruktion des direkten Limes in der Kategorie der Mengen: \\
Auf der Menge $ \coprod_{i\in I}Ob(\C_i) $ definieren wir für je zwei Elemente $ A \in \C_i $, $ B \in \C_j $ die Äquivalenzrelation 
\[ A\sim B\Leftrightarrow \exists k\in I\text{, sodass } \phi_{ki}(A)=\phi_{kj}(B). \] 
Die Relation ist offensichtlich symmetrisch und reflexiv. Transitivität gilt, weil die Indexmenge gerichtet ist. Seien genauer $ A\in\C_h, B\in\C_i, C\in\C_j $, so dass $ A\sim B\sim C $. Dann gibt es geeignete Indizes $ k,l\in I $, so dass $ \phi_{kh}=\phi_{ki} $ und $ \phi_{li}=\phi_{lj} $. Da das System filtrierend ist, finden wir $ m\in I $, so dass $ l\le m \ge k $. Dann gilt $ \phi_{mh}(A)=\phi_{mi}(B)=\phi_{mj}(C) $, also $ A\sim C $. \\
Wir setzen 
\[ Ob(\C):=\frac{\coprod_{i\in I}Ob(\C_i)}{\sim}. \]
Für Morphismen definieren wir analog auf $ \coprod_{i\in I;A,B\in Ob(\C_i)}\Hom_{\mathcal{C}_i}(A,B) $ folgende Äquivalenzrelation: Seien $ \varphi\in \Hom_{\mathcal{C}_i}(A,B),\psi\in \Hom_{\mathcal{C}_j}(C,D)\ $. Wir definieren
\[ \varphi\sim\psi\Leftrightarrow\exists k\in I: \phi_{ik}(\varphi)=\phi_{jk}(\psi). \]
Für $ A, B \in Ob(\C)$ setzen wir 
\[\Hom_{\C}(A,B):=\frac{\coprod_{i\in I}\{\varphi\in \Hom_{\C_i}(A_i,B_i)|A_i,B_i \text{ Repräsentanten von }A, B \text{ in }\C_{i}\}}{\sim}. \]
Da $ I $ eine Menge ist, ist also auch $ \Hom_{\C}(A,B) $ eine Menge.

\begin{itemize}

\item Verknüpfung von Morphismen:  Seien $ A, B,C\in Ob(\C) $, $ \varphi\in \Hom_{\C}(A,B) $ und\\
$ \psi\in \Hom_{\C}(B,C) $. Für beliebige Repräsentanten $ \varphi_i\in \Hom_{\C_{i}}(A_i,B_i) $, \\$ \psi_j\in \Hom_{\C_{j}}(B_j,C_j)$ von $ \varphi $ bzw. $ \psi $, finden wir nach Voraussetzung ein $ k\in I $ und entsprechende Übergangsabbildungen, sodass $ \phi_{ki}(B_i)=\phi_{kj}(B_j) $. Wir definieren $\psi\circ\varphi=[\phi_{kj}(\psi_j)\circ\psi_{ki}(\phi_i)]$. Wir müssen nun zeigen, dass diese Definition unabhängig von der Wahl der Repräsentanten, sowie von der Wahl von $ k $ ist. Dies prüft man leicht, indem man zu einem hinreichend großen Index übergeht.\\
Seien also $ \varphi_{i'}\in \Hom_{\C_{i'}}(A_{i'},B_{i'}) $, $ \psi_{j'}\in \Hom_{\C_{j'}}(B_{j'},C_{j'})$ andere Repräsentanten von $ \varphi $ bzw. $ \psi $. Sei $ k'\in I  $ beliebig, sodass die Übergangsabbildungen $ \phi_{k'i'} $ und $ \phi_{k'j'} $ existieren. Da $ [\varphi_i]=[\varphi_{i'}] $, existiert nach Definition von $ \sim $ ein $ i''\in I $, so dass 
\[ \phi_{i''i'}(\varphi_{i'})=\phi_{i''i'}(\varphi_{i'})\] 
gilt. Analog existiert ein $ j''\in I $, so dass 
\[\phi_{j''j'}(\psi_{j'})=\phi_{j''j'}(\psi_{i'})\]
gilt. Wir vergleichen $ \phi_{kj}(\psi_j)\circ\phi_{ki}(\varphi_i) $ mit $\phi_{k'j'}(\psi_{j'})\circ\phi_{k'i'}(\varphi_{i'})$ durch Übergang zu einem entsprechend "`größerem"' Index. Sei $ l\in I $, so dass folgende Übergangsabbildungen existieren.

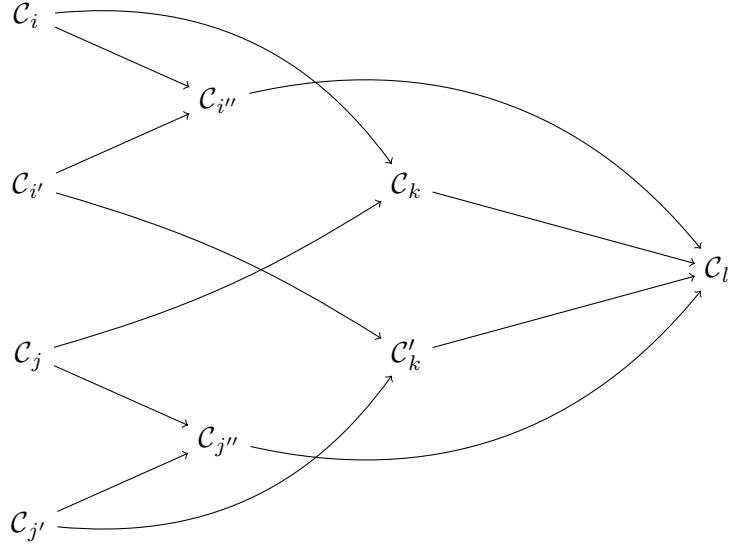
\begin{center}
\begin{tikzpicture}[description/.style={fill=white,inner sep=2pt}]
    \matrix (m) [matrix of math nodes, row sep=1.5em,
    column sep=4.5em, text height=1.5ex, text depth=0.25ex]
    {  \C_i & & & \\
       & \C_{i''} & & \\
       \C_{i'} & & \C_k & \\
       & & &  &  \C_l \\
	  \C_j & & \C_k' &  \\
       & \C_{j''} &  \\
       \C_{j'} & & & \\       
	};
    \path[->,font=\scriptsize]
  	(m-1-1) edge [bend left=30] (m-3-3)  
  	(m-7-1) edge [bend right=30] (m-5-3)  
	(m-3-1) edge node[auto] {}  (m-2-2)  
	(m-5-1) edge [bend right=8]  (m-3-3) 
	(m-3-1) edge [bend left=8]  (m-5-3)
	(m-5-1) edge node[auto] {}  (m-6-2)   
	(m-7-1) edge node[auto] {}  (m-6-2)   
    (m-1-1) edge node[auto] {}  (m-2-2)
    (m-3-3) edge node[auto] {}  (m-4-5)
    (m-5-3) edge node[auto] {}  (m-4-5)
  	(m-2-2) edge [bend left=32]  (m-4-5) 
	(m-6-2) edge [bend right=32]  (m-4-5)
    ;   
\end{tikzpicture}
\end{center}

Es gilt nun:
\begin{align*}
\phi_{lk}(\phi_{kj}(\psi_{j})\circ\phi_{ki}(\varphi_{i})) &=\phi_{lj}(\psi_{j})\circ\phi_{li}(\varphi_{i})\\
&=(\phi_{lj''}\circ\phi_{j''j})(\psi_{j})\circ(\phi_{li''}\circ\phi_{i''i})(\varphi_{i})\\
&=(\phi_{lj''}\circ\phi_{j''j'})(\psi_{j'})\circ(\phi_{li''}\circ\phi_{i''i'})(\varphi_{i'})\\
&=\phi_{lj'}(\psi_{j'})\circ\phi_{li'}(\varphi_{i'})\\
&=\phi_{lk'}(\phi_{k'j'}(\psi_{j'})\circ\phi_{k'i'}(\varphi_{i'}))
\end{align*}
Also gilt:
\[
[\phi_{kj}(\psi_j)\circ\phi_{ki}(\varphi_i)]=[\phi_{k'j'}(\psi_{j'})\circ\phi_{k'i'}(\varphi_{i'})]
\]

\item Identität: \\
Für $ A\in\C $ ist der Identitätsmorphismus offensichtlich gegeben durch $ [id_i] $, wobei $ id_i $ die Identität aus $ \End_{\C_i}(A) $ beschreibt. Die Definition ist wohldefiniert, da Funktoren Identität auf Identität abbilden.\\

\item Nullobjekt:\\
Nach Lemma \ref{Nullobjekt} bilden additive Funktoren Nullobjekte auf Nullobjekte ab. Da das Nullobjekt das einzige Objekt ist, dessen Endomorphismenmenge aus nur einem Element besteht und die Übergangsfunktoren zudem treu sind, werden nur Nullobjekte auf Nullobjekte abgebildet. Es gibt also eine wohldefinierte Äquivalenzklasse $ [0_{\C_i}] $. Wir möchten überprüfen, dass $ [0_{\C_i}] $ das Nullobjekt in $ \C $ ist. Sei dafür $A$ ein weiteres Objekt in $ \C $ und seien $ f,g\in \Hom([0_{\C_i}],A) $. Da unsere Indexmenge gerichtet ist, finden wir ein $ i\in I $ mit Repräsentanten $ f_i,g_i\in \Hom_{\C_i}(0_{\C_i},A_i) $. Da $ 0_{\C_i} $ das Nullobjekt in $ \C_{i} $ ist, gilt $ f_i=g_i $, also auch $ f=g $. Dies zeigt, dass $ [0_{\C_i}] $ initial ist. Analog zeigt man, dass es nur eine Äquivalenzklasse von Morphismen in $ \Hom(A,[0_{\C_i}]) $ gibt, i.e. $ [0_{\C_i}] $ terminal ist. $ [0_{\C_i}] $ ist also das Nullobjekt in $ \C $.

\item Produkte und Koprodukte:\\
Additive Funktoren vertauschen mit endlichen Produkten und Koprodukten. \cite[Thm.3.11]{Freyd1964}. Für zwei Äquivalenzklassen $ [A],[B] $ in $ \C $  ist die Definition eines Objekts $ [A]\oplus [B]:=[A\oplus B] $ also unabhängig von der Repräsentantenwahl und wohldefiniert. Wir müssen die universelle Eigenschaft des Koporodukts überprüfen. Die zugehörigen Morphismen auf Repräsentantenebene $ A\xrightarrow{i} A\oplus B \xleftarrow{j}B $ liefern uns Abbildungen $ [A]\xrightarrow{[i]} [A\oplus B] \xleftarrow{[j]}[B] $, das heißt wir müssen für ein weiteres Objekt $ C $ mit Morphismen $ [A]\xrightarrow{[\tilde{i}]} C \xleftarrow{[\tilde{j}]}[B] $ zeigen, dass es einen eindeutigen Morphismus $ u $ gibt, sodass folgendes Diagramm kommutiert. 

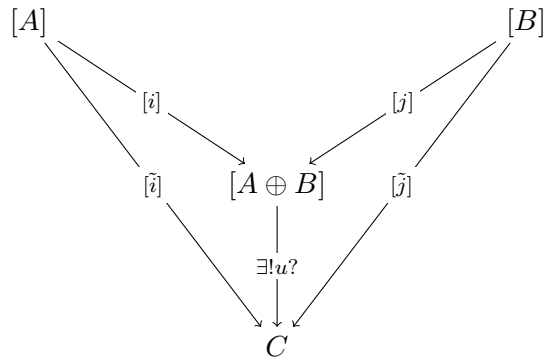
\begin{center}
\begin{tikzpicture}[description/.style={fill=white,inner sep=2pt}]
    \matrix (m) [matrix of math nodes, row sep=4.2em,
    column sep=5.5em, text height=1.5ex, text depth=0.25ex]
    {   {[A]} & & {[B]}  \\
    	      & {[A\oplus B]} &   \\
   	   	      &  C &  \\
	};
    \path[->,font=\scriptsize]
    (m-1-1) edge node[description] {$[i]$} (m-2-2)
    (m-1-3) edge node[description] {$[j]$} (m-2-2)
    (m-1-1) edge node[description] {$[\tilde{i}]$} (m-3-2)
    (m-1-3) edge node[description] {$[\tilde{j}]$} (m-3-2)
    (m-2-2) edge node[description] {$\exists ! u? $} (m-3-2)
    ;   
\end{tikzpicture}
\end{center}

Die Existenz ist klar, da wir auf Repräsentantenebene eine entsprechende Abbildung finden, deren Äquivalenzklasse unser Diagramm kommutativ macht. Falls es mehr als einen solchen Morphismus gäbe, so gäbe es auch in $ \C_i $ für hinreichend großes $ i\in I $ mehr als eine solche Abbildung auf Repräsentantenebene. Dies widerspräche der universellen Eigenschaft des Koprodukts in $ \C_i $. \\
Analog zeigt man die Existenz von Produkten. 

\item Kerne und Kokerne:\\
Wir bezeichnen im Weiteren der Übersicht wegen Objekte aus $ \C $ mit Großbuchstaben, für deren Repräsentanten in einer Kategorie $ \C_i $ versehen wir den Großbuchstaben mit einem Index i. Seien also wieder $ A,B $ Objekte in $ \C $, $ f\in \Hom_{\C}(A,B)$ ein Morphismus mit  Repräsentant $ f_i\in \Hom_{\C_i}(A_i,B_i) $. Dann existieren in der abelschen Kategorie $ \C_i $ zum Morphismus $ f_i $ ein Kern $  Ker(f_i) $. Da unsere Übergangsabbildungen aus treuen, exakten Funktoren bestehen, werden unter ihnen nach Korollar \ref{treu_exakt} genau Kerne auf Kerne und genau Kokerne auf Kokerne abgebildet. Die Definition $ ker(f):=[ker(f_i)] $ ist also unabhängig von Repräsentantenwahl. Die universelle Eigenschaft von $ ker(f) $ überprüft man analog zu der Vorgehensweise bei Produkten und Koprodukten. \\
Für Kokerne funktioniert das Argument analog.

\item Monomorphismen sind Kerne, Epimorphismen sind Kokerne:\\
Seien $ C,D $ Objekte in $ \C$, $ f:C\rightarrow D $ ein Monomorphismus. Wir wollen zunächst zeigen, dass dann schon $f_i:C_i\rightarrow D_i $ für beliebige Repräsentanten $ C_i,D_i\in \C_i $ ein Monomorphismus ist. Seien $ B_i\in\C_i $ ein beliebiges Objekt und $ g_i,h_i:B_i\rightarrow C_i $ beliebige Morphismen, so dass $ f_i\circ g_i=f_i \circ h_i $. Dann ist $ f \circ  g=f \circ h $. Da $f$ ein Monomorphismus ist, gilt $ g=h $, also $ g_i\sim h_i $. Es gibt also ein $ k\in I $, sodass $ \phi_{ki}(g_i)=\phi_{ki}(h_i) $. Da die Übergangsfunktoren treu sind, gilt also $ g_i=h_i $. Also ist $ f_i $ ein Monomorphismus. Da $ \C_i $ eine abelsche Kategorie ist, ist $ f_i$ also ein Kern. Wir haben bereits nachgerechnet, dass dann auch $ f\in\Hom_{\C}(C,D) $ ein Kern ist.\\
Für Epimorphismen analog.

\item $ R $-linear:\\
Seien $ C,D\in\C$. Da $ \C $ abelsch ist, ist $ \Hom_{\C}(C,D) $ eine Gruppe. Wir müssen zeigen, dass $ \Hom_{\C}(C,D) $ sogar ein $ R $-Modul ist. Sei $f\in \Hom_{\C}(C,D) $. Wir wählen einen Repräsentanten $ f_i\in \Hom_{\C_i}(C_i,D_i) $. Nun definieren wir für alle $ r\in R $ die Multiplikation $ r\cdot f:=[r\cdot f_i] $. Diese Konstruktion ist unabhängig von der Repräsentantenwahl, da die Übergangsabbildungen $ R $-linear sind. Die so definierte $ R $-Multiplikation erfüllt Assoziativ- und Distributivgesetz, da sie das auf Repräsentantenebene erfüllt.

\end{itemize}

$ \C $ ist also eine $ R $-lineare abelsche Kategorie. Wir möchten nun zeigen, dass die Kategorie $ \C $ zusammen mit den Funktoren

\[
\begin{array}{cccc}
\psi_i: & \C_i & \rightarrow & \mathcal{C}\\
 & A_i & \mapsto & [A_i]
\end{array}
\]

die universelle Eigenschaft des direkten Limes erfüllt. Nach Konstruktion kommutieren die Funktoren $ \psi_i $ mit den Übergangsabbildungen $ \phi_{ji} $. Für alle $ i\le j $ gilt also $ \psi_i=\psi_j\circ\phi_{ji} $.\\
Sei nun $ (\mathcal{D},\{\chi_i\}_{i\in I}) $ eine weiterer Kandidat für den direkten Limes, also eine abelsche Kategorie $ \mathcal{D} $ zusammen mit $ R $-linearen additiven, treuen, exakten Funktoren \\
$ \chi_i:\C_i\rightarrow\mathcal{D} $, die mit allen Übergangsabbildungen kommutieren. Dann definieren wir

\[
\begin{array}{cccc}
 u: & \C & \rightarrow & \mathcal{D}\\
 & A & \mapsto & \chi_i(A_i).
\end{array}
\]

Die Definition dieser Abbildung ist wegen $ \chi_i=\chi_j\circ\phi_{ji} $ unabhängig von der Wahl von $ i $ und, da $ \chi_i=u\circ\psi_i $ sein soll, auch die einzige Abbildung, die folgendes Diagramm kommutativ macht:

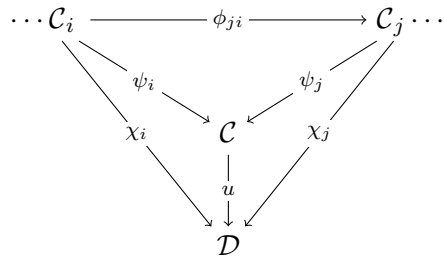
\begin{center}
\begin{tikzpicture}[description/.style={fill=white,inner sep=2pt}]
    \matrix (m) [matrix of math nodes, row sep=2.5em,
    column sep=4.0em, text height=1.5ex, text depth=0.25ex]
    { \cdots \C_i & & \C_j \cdots \\
    	      & \C &   \\
   	   	      &  \mathcal{D} &  \\
	};
    \path[->,font=\scriptsize]
    (m-1-1) edge node[description] {$ \phi_{ji} $} (m-1-3)
    (m-1-1) edge node[description] {$ \psi_{i} $} (m-2-2)
    (m-1-1) edge node[description] {$ \chi_{i} $} (m-3-2)
    (m-1-3) edge node[description] {$ \psi_{j} $} (m-2-2)
    (m-1-3) edge node[description] {$ \chi_{j} $} (m-3-2)
    (m-2-2) edge node[description] {$u $} (m-3-2)
    ;   
\end{tikzpicture}
\end{center}

Wir müssen überprüfen, ob es sich bei $ u $ um einen $ R $-linearen, additiven, treuen, exakten Funktor handelt.\\
Seien $ A,B\in\C $, $ f,g\in\Hom_{\C}(A,B) $, $ f \neq g $. Wir wählen nun Repräsentanten \\
$ f_i,g_i\in\Hom_{\C_i}(A_i,B_i) $, für die dann ebenfalls $ f_i \neq g_i $ gilt. Die Additivität, Treuheit und \\
$ R $-Linearität von $u$ folgen nun aus der Tatsache, dass $ \chi_i $ additiv, treu und $ R $-linear ist, sowie aus der Kommutativität von
\begin{center}
\begin{tikzpicture}[description/.style={fill=white,inner sep=2pt}]
    \matrix (m) [matrix of math nodes, row sep=3.5em,
    column sep=5.5em, text height=1.5ex, text depth=0.25ex]
    { \C_i & \C  \\
   	  &  \mathcal{D}   \\
	};
    \path[->,font=\scriptsize]
    (m-1-1) edge node[description] {$ \chi_i $} (m-2-2)
    (m-1-2) edge node[description] {$ u $} (m-2-2)
    (m-1-1) edge node[description] {$ \psi_{i} $} (m-1-2)
    ;   
\end{tikzpicture}
\end{center}
Genauer gilt 
\begin{itemize}
\item treu: Da $ \chi_i $ treu ist, ist $ u(f)=(u\circ\psi)(f_i)=\chi (f_i)\neq\chi (g_i)=(u\circ\psi)(g_i)=u(g)$
\item additiv: $ u(f+g)=(u\circ\psi_i)(f_i+g_i)=\chi_i(f_i+g_i)=\chi_i(f_i)+\chi_i(g_i)\\=(u\circ\phi)(f_i)+(u\circ\phi)(g_i)=u(f)+u(g) $
\item $ R $-linear: $ u(r\cdot f)=(u\circ\psi_i )(r\cdot f_i)=\chi_i(r\cdot f_i)=r\cdot (\chi_i(f_i))=r\cdot(u\circ\psi_i)(f_i)=r\cdot u(f) $
\end{itemize}
Für die Exaktheit gehen wir genauso vor. Für eine exakte Sequenz
\[ 0\rightarrow A\xrightarrow{f} B\xrightarrow{g} C \rightarrow 0  \]
finden wir Repräsentanten
\[ 0\rightarrow A_i\xrightarrow{f_i} B_i\xrightarrow{g_i} C_i \rightarrow 0,  \]
die eine exakte Sequenz in $ \C_i $ bilden. Die Exaktheit bleibt durch Anwenden von $ \chi_i $ erhalten.
\end{proof}

\begin{Def}\label{Restriktion_Skalare}
Seien $ A,B $  Ringe und $ \varphi:A\rightarrow B $ ein Ringhomomorphismus. Sei $ M $ ein $ B $-Modul Dann definiert $ \varphi $ auf natürliche Weise eine $ A $-Modulstruktur auf $ M $: \\
Für $ a\in A,m\in M $ sei
\[ a\odot_A m:=\varphi(a)\odot_B m.\]
Diese Multiplikation auf $ A $ ist offensichtlich assoziativ und distributiv. \\
Sei $ N $ ein weiterer $ B $-Modul, $ f:M\rightarrow N $ ein $ B $-Modulhomomorphismus. Dann wird $ f $ durch 
\[ a\odot_A f(m)=\varphi(a)\odot_B f(m)=f(\varphi(a)\odot_B m)=f(a\odot_A m) \] 
zum $ A $-Modulhomomorphismus. \\
$ \varphi $ definiert also einen Funktor $ \phi $ von der Kategorie der $ B $-Moduln in die Kategorie der $ A $-Moduln. Diesen Funktor nennt man \emph{Restriktion der Skalare}. 
\end{Def}

\begin{Bem}\label{algebra_als_modul}
Sei $f:R\rightarrow A $ eine $ R $-Algebra und $ M $ ein Modul über $ A $. Sei $ m\in M, r\in R $. Dann wird $ M $ durch Restriktion der Skalare auf natürliche Weise zum $ R $-Modul.\\
Ist $ A $ als $ R $-Modul endlich erzeugt und $ M $ als $ A $-Modul endlich erzeugt, so ist offensichtlich auch $ M $ als $ R $-Modul endlich erzeugt.
\end{Bem}

\begin{Lemma}\label{Lemma_Skalare}
Seien $f_1:R\rightarrow A $, $f_2:R\rightarrow B $ zwei $ R $-Algebren und als $ R $-Moduln jeweils endlich erzeugt. Sei 

\begin{center}
\begin{tikzpicture}[description/.style={fill=white,inner sep=2pt}]
    \matrix (m) [matrix of math nodes, row sep=3.0em,
    column sep=2.0em, text height=1.5ex, text depth=0.25ex]
    { A & & B  \\
   	  & R &   \\
	};
    \path[->,font=\scriptsize]
    (m-1-1) edge node[description] {$ \varphi $} (m-1-3)
    (m-2-2) edge node[description] {$ f_1 $} (m-1-1)
    (m-2-2) edge node[description] {$ f_2 $} (m-1-3)
    ;   
\end{tikzpicture}
\end{center}

ein $ R $-Algebrenhomomorphismus. Dann definiert die Restriktion der Skalare einen exakten, treuen, additiven, $ R $-linearen Funktor 
\[ \phi:B\Mod\rightarrow A\Mod \]
von der Kategorie der endlich erzeugten $ B $-Moduln in die Kategorie der endlich erzeugten $ A $-Moduln. Ein $ B $-Modul M trägt die gleiche natürliche Struktur als $ R $-Modul wie sein Bild $\phi(B)$.
\end{Lemma}

\begin{proof}
Sei $ M $ ein endlich erzeugter $ B $-Modul. Durch Anwenden des Funktors $\phi(B)$ wird nur die multiplikative Struktur des Moduls verändert, die Gruppenstruktur bleibt gleich. Also ist der Funktor additiv. Kern und Kokern eines Modulhomomorphismus sind gleich dem Kern und Kokern als Gruppenhomomorphismus. Also erhält der Funktor Kerne und Kokerne und ist somit exakt.\\
Zwei Morphismen $ f,g:M\rightarrow N $ sind genau dann als Modulhomomorphismen gleich, wenn sie als Abbildungen von Mengen gleich sind. Die multiplikative Struktur spielt hier keine Rolle. Also ist der Funktor treu.\\
Sei $ m\in M $, $ r\in R $. Die $ R $-Modulstruktur auf dem $ A $-Modul $ \phi(M) $ ist gegeben durch
\[ r \odot_R m= f_1(r) \odot_A m=\varphi(f_1(r))\odot_B m=f_2(r)\odot_B m. \]
Sie ist also gleich der $ R $-Modulstruktur auf dem $ B $-Modul $ M $. Also ist der Funktor $ \phi $ insbesondere $ R $-linear.
\end{proof}

Wir sind nun in der Lage, (1) des Theorems \ref{thm1} für beliebige, nicht notwendig endliche, Diagramme zu zeigen:

\begin{Proposition}\label{1}
Sei $ D $ ein gerichtetes Diagramm, $ R $ ein noetherscher Ring und $ T\rightarrow R\Mod $ eine Darstellung. Dann existiert eine $R$-lineare, abelsche Kategorie $ \mathcal{C}(T) $ mit einer Darstellung 
\begin{equation}
 \tilde{T}:D\rightarrow\mathcal{C}(T), 
\end{equation}
sowie ein treuer, exakter Funktor $ ff_T $,  sodass $T$ durch 
\begin{equation}
 D\xrightarrow{\tilde{T}}\mathcal{C}(T)\xrightarrow{ff_T}R\Mod  
\end{equation}
faktorisiert.
\end{Proposition}

\begin{proof}
Der Beweis basiert auf \cite[II.5.3]{MR2182598}.
Wir erinnern uns an die Situation aus Proposition \ref{thm1_endlich}. Für jedes endliche volle Unterdiagramm $ F\subset D $ können wir die $ R $-Algebra $ \End(T_F)\subset \prod_{p \in F}\End_R(T_p) $ bilden. Sie besteht aus Tupeln $ (e_p)_{p\in F} \in \prod_{p \in F}\End_R(T_p) $, so dass für alle $ p,q\in F $ und $ m\in M(p,q) $ das Diagramm 

\begin{center}
\begin{tikzpicture}[description/.style={fill=white,inner sep=2pt}]
    \matrix (m) [matrix of math nodes, row sep=3em,
    column sep=2.5em, text height=1.5ex, text depth=0.25ex]
    {  T(p) & & T(q) \\
      T(p) & & T(q) \\
	};
    \path[->,font=\scriptsize]
    (m-1-1) edge node[auto] {$ e_{p} $} (m-2-1)
    (m-1-3) edge node[auto] {$ e_{q} $} (m-2-3)
    (m-1-1) edge node[auto] {$T(m)$} (m-1-3)
    (m-2-1) edge node[auto] {$T(m)$} (m-2-3);
\end{tikzpicture}
\end{center}

kommutiert. Für eine endliche Obermenge $ F'\supset F $ liefert die Projektion 
\[ pr:\prod_{p \in F'}\End_R(T_p)\twoheadrightarrow\prod_{p \in F}\End_R(T_p) \]
einen surjektiven $ R $-Algebrenhomomorphismus. Dieser lässt sich einschränken auf\\
 $ \End(T_{F'})\subset\prod_{p \in F'}\End_R(Tp) $ und liefert einen $ R $-Algebrenhomomorphismus 
\[ \tilde{pr}:\End(T_{F'})\rightarrow\End(T_F). \]
Wie wir in Proposition \ref{thm1_endlich} gesehen haben, ist für alle endlichen Teilmengen $ F\subset D $ die Algebra $ \End(T_F) $ als $ R $-Modul endlich erzeugt. Nach Lemma \ref{Lemma_Skalare} erhalten wir also für alle endlichen Teilmengen $ F,F'\subset D $ mit $ F\subset F' $ durch Restriktion der Skalare einen treuen, exakten, additiven, $ R $-linearen Funktor
\[ \tilde{Pr}:\End(T_F)\Mod\rightarrow\End(T_{F'})\Mod \]
von den endlich erzeugten $ \End(T_F) $-Moduln in die endlich erzeugten $ \End(T_{F'}) $-Moduln. Die endlichen Teilmengen von $ D $ bilden bezüglich Inklusion ein partiell geordetes Mengensystem. Dieses System ist filtrierend, da für zwei endliche Teilmengen $ E,F\subset D $ auch $ E\cup F $ endlich ist und $ E\subset E\cup F $ sowie $ F\subset E\cup F $ gilt.\\
Wir haben also ein filtrierendes System abelscher Kategorien mit exakten, treuen, additiven, $ R $-linearen Funktoren als Übergangsabbildungen. Nach Proposition \ref{Kategorienproposition} können wir den direkten Limes 
\[  \C(T):=\varinjlim_{F\subset D\text{ endlich}}\End(T_F)\Mod \]
bilden, der selbst wieder eine $ R $-lineare abelsche Kategorie ist. Jeder $ R $-Modul $ Tp $ ist also ein $ \End(T_F) $-Modul für geeignetes endliches $ F $ und repräsentiert ein Objekt in der Kategorie $ \C(T) $. Durch Nachschalten des Vergissfunktors, der $ Tp $ wieder als $ R $-Modul auffasst, erhalten wir die gewünschte Faktorisierung 
\[
\begin{array}{ccccc}
D & \xrightarrow{\tilde{T}} & \mathcal{C}(T) & \xrightarrow{ff_T} & R\Mod \\
p & \mapsto & Tp & \mapsto & Tp.
\end{array} 
\]

\end{proof}

\chapter{Diagrammkategorie abelscher Kategorien}
\label{ch:kap_3}
Sei für dieses Kapitel $ R $ stets ein kommutativer, noetherscher Ring und $ R\Mod $ die Kategorie der endlich erzeugten $ R $-Moduln. Ist $ R' $ ein nicht kommutativer Ring so bezeichnen wir mit $ R'\Mod $ die endlich erzeugten $ R' $-Linksmoduln und mit $ \Mor R' $ die endlich erzeugten $ R' $-Rechtsmoduln.\\\\
Wir haben konstruiert, wie die Darstellung eines Diagramms 
\[ D\overset{T}{\longrightarrow}R\Mod \]
über eine abelsche Kategorie $ \C(T) $ faktorisiert. Wir bezeichnen diese Kategorie im Weiteren als die Diagrammkategorie von $ T $. \\
Wir betrachten nun den Fall, dass das Diagramm selbst schon eine abelsche Kategorie $ \mathcal{A} $ und die Darstellung $ T $ ein treuer, $ R $-linearer, exakter Funktor ist. Die Darstellung
\[ \mathcal{A}\overset{T}{\longrightarrow}R\Mod \]
faktorisiert über
\[ \mathcal{A}\overset{\tilde{T}}{\longrightarrow}\C(T)\overset{ff_T}{\longrightarrow}R\Mod. \]
Wir werden zeigen, dass $ \tilde{T} $ dann schon eine Äquivalenz von Kategorien ist. Dieser Satz spielt die entscheidende Rolle für den Beweis der universellen Eigenschaft der Diagrammkategorie. Der Beweis wird den Rest des Kapitels einnehmen. Einen ähnlichen alternativen Beweis findet man in \cite{MR2793022}. Da wir viele technische Details entwickeln müssen, möchten wir zunächst eine Motivation für die Vorgehensweise geben.\\
Sei $ p $ ein Objekt aus $ \mathcal{A} $. Um zu zeigen, dass der Funktor
\[ \mathcal{A}\overset{\tilde{T}}{\longrightarrow}\C(T)=\varinjlim_{E\subset\mathcal{A}\text{ endlich}}\End(T_{|E})\Mod  \]
wesentlich surjektiv ist, müssen wir insbesondere für jeden $ \End(T_{|\{p\}}) $-Modul ein Urbild in $ \mathcal{A} $ finden. Da $ \End(T_{|\{p\}}) $ als Modul die Kategorie der endlich erzeugten \\ $ \End(T_{|\{p\}}) $-Moduln erzeugt, stellt es sich als zentraler Schritt heraus, in $ \mathcal{A} $ ein Urbild von $ \End(T_{|\{p\}})\subset \End_R(Tp,Tp) $ zu finden. 
Wir werden in Abschnitt 3.1 zeigen, wie man für $ p\in\mathcal{A} $ und einen beliebigen endlich erzeugten $ R $-Modul $ M $ funktoriell ein Objekt $ \Hom_R(M,p) $ in $ \mathcal{A} $ konstruieren kann, so dass für $ T $ gilt: 
\[
\begin{array}{ccc}
\mathcal{A} & \overset{T}{\longrightarrow} & R\Mod \\
\Hom_R(M,p) & \mapsto & \Hom_R(M,Tp). \\
\end{array} 
\]
Setzen wir $ M=Tp $ so erhalten wir mit $ \Hom_R(Tp,p) $ ein Urbild von $ \End_R(Tp) $ unter $ T $.
Die Algebra $ \End(T_{|\{p\}})\subset \End_R(Tp) $ ist definiert durch
\[ \End(T_{|\{p\}})=\big \{ \varphi\in\End_R(Tp)|\forall \alpha\in\End_{\mathcal{A}}(p):\varphi\circ T\alpha=T\alpha\circ \varphi\big \}. \]
Eine ähnliche Konstruktion möchten wir nun für die Elemente von $ \Hom_R(Tp,p) $ in $ \mathcal{A} $ wiederholen, um ein Urbild von $  \End(T_{|\{p\}}) $ zu erhalten. Zu diesem Zweck konstruieren wir in Abschnitt 3.2 Operationen auf $ \Hom_R(Tp,p) $, welche durch $ T $ auf die gewöhnliche Rechts- und Linksoperation auf $ \End_R(Tp) $ abgebildet werden. Die Idee dieser Konstruktion stammt aus \cite{Nori}.

\section{Funktorkonstruktionen nach $ \mathcal{A} $}
\label{Homs}

\begin{Def}
Sei $ A $ ein beliebiger Ring. Ein $ A $-Modul P heißt \emph{projektiv}, wenn für alle surjektiven $ A $-Modulhomomorphismen $ e:L\twoheadrightarrow L' $ und für alle $ A $-Modulhomomor- phismen $ f:P\rightarrow L $ ein $ A $-Modulhomomorphismus $ g:P\rightarrow L $ existiert, so dass folgendes Diagramm kommutiert:
\begin{center}
\begin{tikzpicture}[description/.style={fill=white,inner sep=2pt}]
    \matrix (m) [matrix of math nodes, row sep=3.5em,
    column sep=3.5em, text height=1.5ex, text depth=0.25ex]
    {  &  P & \\
     L & & L'  \\
	};
    \path[->,font=\scriptsize]
    (m-2-1) edge node[description] {$ e $} (m-2-3)
    (m-1-2) edge node[description] {$ f $} (m-2-3)
    ;   
    \path[->,dotted]
    (m-1-2) edge node[description] {$ \exists g $} (m-2-1)
    ;
\end{tikzpicture}
\end{center}
\end{Def}

\begin{Bem}
Freie Moduln sind stets projektiv.
\end{Bem}

\begin{Lemma}\label{tensor}
Sei $ \mathcal{A} $ eine $ R $-lineare, abelsche Kategorie, $ R' $ eine nicht notwendig kommutative endlich erzeugte $ R $-Algebra.
Sei weiter $ p\in\mathcal{A} $ ein ausgezeichnetes Objekt und
\begin{center}
\begin{tikzpicture}[description/.style={fill=white,inner sep=2pt}]
    \matrix (m) [matrix of math nodes, row sep=3em,
    column sep=2.5em, text height=1.5ex, text depth=0.25ex]
    { R' & & \End_{\mathcal{A}}(p)  \\
	  & R & \\
	};
    \path[->,font=\scriptsize]
    (m-1-1) edge node[description] {$f$} (m-1-3)
    (m-2-2) edge node[description] {$ i_1 $} (m-1-1)
    (m-2-2) edge node[description] {$ i_2 $} (m-1-3);

\end{tikzpicture}
\end{center}

ein endlicher $ R $-Algebrenhomomorphismus. Bezeichne $ \Mor R'$ die Kategorie der endlich erzeugten $ R' $-Rechtsmoduln. Dann existiert ein bis auf Isomorphie eindeutiger exakter, $ R $-linearer, kovarianter Funktor

\[
\begin{array}{ccc}
\mathcal{F}:\Mor R'& \longrightarrow & \mathcal{A},\\
\end{array}
\]

so dass für $ R'\in \Mor R'$
\[
\mathcal{F}(R')=p
\]
und für Morphismen
\[
\begin{array}{ccc}
\mathcal{F}:\End_{R'}(R')\cong R'&\longrightarrow & \End_{\mathcal{A}}(p)\\
a & \longmapsto & f(a)
\end{array}
\]
erfüllt sind.
\end{Lemma}

\begin{Bem}
Durch den $ R $-Algebrenhomomorphismus $ f $ bekommt $ p $ eine Struktur als $ R' $-Modul. Für einen beliebigen $ R' $-Rechtsmodul $ M $ ähnelt die Konstruktion des Objektes $ \mathcal{F}(M) $ in $ \mathcal{A} $ der Konstruktion des Tensorproduktes $ M\otimes_{R'}p $. Wir bezeichnen den Funktor $ \mathcal{F} $ daher oft mit $ \_\otimes_{R'}p $. Für den Fall $ R'=R $ und die kanonische Abbildung $ f:R\rightarrow \End_{\mathcal{A}}(p) $, liefert die Verkettung
\[ R\Mod\xrightarrow{\_\otimes_{R}p}\mathcal{A}\overset{T}{\longrightarrow} R\Mod\]
das gewöhnliche Tensorieren $ \_\otimes_R Tp $ in $ R\Mod $. Eine analoge Aussage für den ähnlich definierten kontravarianten Funktor $ \Hom_R(\_,p) $ zeigen wir in Satz \ref{rechts}.
\end{Bem}

\begin{Bem}\label{Summen}
Bevor wir das Lemma beweisen, erinnern wir uns an ein Resultat über direkte Summen in abelschen Kategorien. Direkte Summen sind Biprodukte, also gleichermaßen Produkte und Koprodukte. Eine endliche direkte Summe besteht demnach aus einem System von Objekten und Abbildungen 
\begin{center}
\begin{tikzpicture}[description/.style={fill=white,inner sep=2pt}]
    \matrix (m) [matrix of math nodes, row sep=3em,
    column sep=2.5em, text height=1.5ex, text depth=0.25ex]
    { X_1 &  & X_1  \\
	  \vdots & \bigoplus_{i=1}^{n}X_i & \vdots \\
	  X_n &  & X_n  \\
	};
    \path[->,font=\scriptsize]
    (m-1-1) edge node[description] {$i_1$} (m-2-2)
    (m-3-1) edge node[description] {$ i_n $} (m-2-2)
    (m-2-2) edge node[description] {$ \pi_1 $} (m-1-3)
    (m-2-2) edge node[description] {$ \pi_n $} (m-3-3);

\end{tikzpicture}
\end{center}
so dass die Einbettungen $ i_j $ und die Projektionen $ \pi_j $ die universelle Eigenschaft des Koproduktes bzw. Produktes erfüllen. Es ist ein elementares Resultat, dass dies genau dann der Fall ist, wenn 
\[ \operatorname{\pi_i\circ i_j}=\begin{cases} id_{X_i}\ , & i=j \\ 0_{Abb}\ , & i\neq j \end{cases}\]
(siehe zB. \cite[Thm2.41]{Freyd1964}).\\
Ein solches Diagramm nennt man auch direktes Summensystem.
Man kann außerdem leicht zeigen, dass ein Funktor genau dann additiv ist, wenn er direkte Summensysteme auf direkte Summensysteme abbildet \cite[Thm3.11]{Freyd1964}. Tatsächlich verwenden einige Autoren diese Eigenschaft als Definition für Additivität bei Funktoren. \\
Ist ein Funktor $\mathcal{F}$ kontravariant und additiv, so muss wegen $ \pi_i\circ i_i=id_{X_i} $ auch $ \mathcal{F}(i_i)\circ \mathcal{F}(\pi_i)=id_{X_i} $ gelten. Also wird durch Anwenden von $ \mathcal{F} $ eine Inklusion zur Projektion und eine Projektion zur Inklusion im direkten Summensystem.
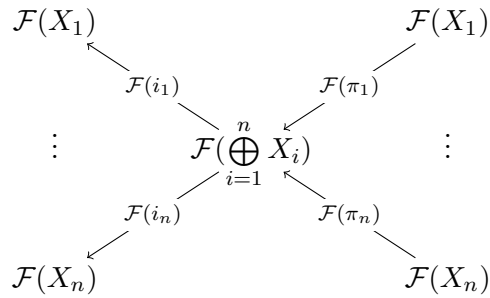
\begin{center}
\begin{tikzpicture}[description/.style={fill=white,inner sep=2pt}]
    \matrix (m) [matrix of math nodes, row sep=3em,
    column sep=2.5em, text height=1.5ex, text depth=0.25ex]
    { \mathcal{F}(X_1) &  & \mathcal{F}(X_1)  \\
	  \vdots & \mathcal{F}(\bigoplus\limits_{i=1}^{n}X_i) & \vdots \\
	  \mathcal{F}(X_n) &  & \mathcal{F}(X_n)  \\
	};
    \path[->,font=\scriptsize]
    (m-2-2) edge node[description] {$\mathcal{F}(i_1)$} (m-1-1)
    (m-2-2) edge node[description] {$ \mathcal{F}(i_n) $} (m-3-1)
    (m-1-3) edge node[description] {$ \mathcal{F}(\pi_1)$} (m-2-2)
    (m-3-3) edge node[description] {$ \mathcal{F}(\pi_n) $} (m-2-2);

\end{tikzpicture}
\end{center}
\end{Bem}

\begin{proof}[Beweis von Lemma \ref{tensor}]
Wir werden zunächst zeigen, dass die geforderten Bedingungen den Funktor schon für freie $ R' $-Rechtsmoduln eindeutig festlegen und dies dann auf beliebige $ R' $-Rechtsmoduln verallgemeinern.\\
\begin{itemize}
\item Da additive, also insbesondere exakte, Funktoren mit direkten Summen vertauschen, gilt 
\[
\mathcal{F}(R'^n)=\mathcal{F}\big ( \bigoplus_{i=1}^{n}R'\big )=\bigoplus\limits_{i=1}^{n}\mathcal{F}(R')=\bigoplus\limits_{i=1}^{n}p=p^n.
\]

\item Endliche direkte Summen in abelschen Kategorien sind gleichermaßen Produkt und Koprodukt. Kombiniert man die universelle Eigenschaft des Koprodukts mit seinen Inklusionen 
\[ p_i\hookrightarrow \bigoplus_{j=1}^n p_j \]
und die universelle Eigenschaft des Produkts mit seinen Projektionen 
\[ \bigoplus_{j=1}^n p_j\twoheadrightarrow p_i, \]
so sieht man, dass ein Morphismus 
\[\bigoplus_{j=1}^n p_j\rightarrow \bigoplus_{j=1}^n p_j\]
gerade gegeben ist durch eine Familie von Abbildungen 
\[ (a_{i,j}:p_j\rightarrow p_i)_{i,j=1}^n,\]
wobei $ a_{ij} $ der Verkettung $ \pi_j\circ f\circ i_i $ entspricht.\\
Es ist also legitim, sich einen Morphismus aus $ \Hom_{\mathcal{A}}(p^n,p^m) $ als eine Matrix \\
$ Mat_{\End(p)}(m\times n) $ vorzustellen.\\
Da additive Funktoren endliche direkte Summen erhalten, ist für \\$ A\in Mat_{R'}(n \times m) $ der Morphismus $ \mathcal{F}(A) $ schon durch die Bilder der einzelnen Abbildungen \\$ a_{i,j}:R'\rightarrow R' $ festgelegt und es gilt:
\[
\begin{array}{ccc}
\mathcal{F}:\Hom_R(R'^m,R'^n)&\rightarrow & \Hom_{\mathcal{A}}(p^n,p^m)\\
A \in Mat_{R'}(n \times m) & \mapsto & (f(a_{i,j}))_{i,j} \in Mat_{\End(p)}(n \times m)
\end{array}
\]
Die Matrix $ (f(a_{i,j})) $ bezeichnen wir ab sofort mit $ \tilde{A} $.
\item Für einen beliebigen endlich erzeugten $ R' $-Rechtsmodul fixieren wir zunächst eine freie Auflösung
\[
\cdots R'^{a_1}\xrightarrow{A}R'^{a_0}  \stackrel{\pi_a}{\twoheadrightarrow}M \rightarrow 0. 
\]
Da $ \mathcal{F} $ exakt ist, gilt
\[
  M\otimes_{R'}p  =\mathcal{F}(M)=\mathcal{F}(coker(A))=coker(\mathcal{F}(A))=coker(\tilde{A}).
\]
\item Um $\mathcal F$ auf Morphismen $\varphi :M \rightarrow N$ zu berechnen, suchen wir zunächst Lifts zwischen den Auflösungen für $ M $ und $ N $, so dass folgendes Diagramm kommutiert.

\begin{center}
\begin{tikzpicture}[description/.style={fill=white,inner sep=2pt}]
    \matrix (m) [matrix of math nodes, row sep=3em,
    column sep=2.5em, text height=1.5ex, text depth=0.25ex]
    { \cdots R'^{a_1}& R'^{a_0}& M & 0 \\
      \cdots R'^{b_1}& R'^{b_0}& N & 0 \\
	};
    \path[->,font=\scriptsize]
    (m-1-1) edge node[auto] {$ A $} (m-1-2)
    (m-1-2) edge node[auto] {$ \pi_{a} $} (m-1-3)
    (m-1-3) edge node[auto] {} 			  (m-1-4)
    
    (m-2-1) edge node[auto] {$ B $} (m-2-2)
    (m-2-2) edge node[auto] {$ \pi_{b} $} (m-2-3)
    (m-2-3) edge node[auto] {} 			  (m-2-4)

    (m-1-1) edge node[auto] {$ \varphi^{1} $} (m-2-1)
    (m-1-2) edge node[auto] {$ \varphi^{0} $} (m-2-2)
    (m-1-3) edge node[auto] {$\varphi$} 	  (m-2-3);

\end{tikzpicture}
\end{center}

Zur Existenz der Lifts:\\
Da $\pi_b$ surjektiv und $R'^{a_0}$ frei, also projektiv ist, existiert nach der universellen Eigenschaft für Projektive eine Abbildung $\varphi^0$, so dass $ \pi_b\circ\varphi^0=\varphi\circ\pi_a $.\\
Da für $ x \in \ker(\pi_a) $ 
\[ \pi_b \circ \varphi^0(x)=\varphi \circ \pi_a(x)=\varphi(0)=0  \]
gilt, liegt $ \varphi^0(x) $ im Kern von $ \pi_b $. Also lässt sich $\varphi^0$ auf
\[
\varphi^0_{|\ker(\pi_a)}: \ker(\pi_a)\rightarrow \ker(\pi_b)
\]
einschränken. Da 
\[ A : R'^{a_1}\twoheadrightarrow im(A)=\ker(\pi_a) \]
und 
\[ B : R'^{b_1}\twoheadrightarrow im(B)=\ker(\pi_b) \]
surjektiv sind, folgt wegen der Projektivität von $ R'^{a_1} $ mit dem gleichen Argument die Existenz von $ \varphi^1 $.
Die Lifts sind jedoch im Allgemeinen nicht eindeutig! 
Wenn wir die Auflösungen mit $ \mathcal{F} $ abbilden, bekommen wir folgendes kommutative Diagramm mit exakten Zeilen.

\begin{center}
\begin{tikzpicture}[description/.style={fill=white,inner sep=2pt}]
    \matrix (m) [matrix of math nodes, row sep=3em,
    column sep=2.5em, text height=1.5ex, text depth=0.25ex]
    { \cdots p^{a_1}& p^{a_0}& coker \tilde{A}=\mathcal{F}(M) & 0 \\
      \cdots p^{b_1}& p^{b_0}& coker \tilde{A}=\mathcal{F}(N) & 0 \\
	};
    \path[->,font=\scriptsize]
    (m-1-1) edge node[auto] {$ \tilde{A} $} (m-1-2)
    (m-1-2) edge node[auto] {$ \tilde{\pi_{a}} $} (m-1-3)
    (m-1-3) edge node[auto] {} 			  (m-1-4)
    
    (m-2-1) edge node[auto] {$ \tilde{B} $} (m-2-2)
    (m-2-2) edge node[auto] {$ \tilde{\pi_{b}} $} (m-2-3)
    (m-2-3) edge node[auto] {} 			  (m-2-4)

    (m-1-1) edge node[auto] {$ \tilde{\varphi^{1}} $} (m-2-1)
    (m-1-2) edge node[auto] {$ \tilde{\varphi^{0}} $} (m-2-2);
    \path[->,dotted](m-1-3) edge node[auto] {} 	  (m-2-3);

\end{tikzpicture}
\end{center}
Wir suchen eine eindeutige Abbildung zwischen $ coker(\tilde{A}) $ und $ coker(\tilde{B}) $, so dass dieses Diagramm kommutiert. Es gilt:  
\[ \tilde{\pi_b}\circ \tilde{\varphi}^0\circ \tilde{A} =\tilde{\pi_b}\circ \tilde{B}\circ \tilde{\varphi}^1 =0_{Abb}\circ \tilde{\varphi}^1=0_{Abb} \]
Also existiert nach der universellen Eigenschaft von $ coker(\tilde{A}) $ eine eindeutige Abbildung 
\[ 
\mathcal{F}(\varphi):\mathcal{F}(M)\rightarrow \mathcal{F}(N),
\] 
welche das Diagramm kommutativ lässt. \\
Wir zeigen nun, dass diese Abbildung wohldefniert ist:
\item Die Abbildung ist unabhängig von der Wahl der Lifts: Seien $ \psi^1:R'^{a_0}\rightarrow R'^{b_0} $ und $ \psi^0:R'^{a_1}\rightarrow R'^{b_1} $ weitere Lifts. Dann gilt
\[ \pi_b \circ \varphi^0 = \varphi \circ \pi_a = \pi_b \circ \psi^0,  \]
nach Anwenden des Funktors also 
\[ \tilde{\pi_b} \circ \tilde{\varphi}^0 = \tilde{\pi_b} \circ \tilde{\psi}^0 \]
und dementsprechend auch
\[ \tilde{\pi_b} \circ \tilde{\psi}^0\circ \tilde{A}=\tilde{\pi_b} \circ \tilde{\varphi}^0\circ \tilde{A}.\]
Wir erhalten also die gleiche eindeutige Abbildung
\[ \mathcal{F}(M)\rightarrow \mathcal{F}(N).\]

\item Wir prüfen nun, ob die Verkettung von Morphismen wohldefiniert ist. Wir müssen zeigen, dass für $ \varphi \in \Hom_R(L,M) $ und $ \psi \in \Hom_R(M,N) $  
\[ \mathcal{F}(\psi\circ\varphi)=\mathcal{F}(\psi)\circ\mathcal{F}(\varphi) \]
gilt. Die Definition von $ \mathcal{F}(\varphi) $ und $ \mathcal{F}(\psi) $ liefert uns folgendes kommutative Diagramm:

\begin{center}
\begin{tikzpicture}[description/.style={fill=white,inner sep=2pt}]
    \matrix (m) [matrix of math nodes, row sep=3em,
    column sep=2.5em, text height=1.5ex, text depth=0.25ex]
    { \cdots p^{a_1}& p^{a_0}& \mathcal{F}(L)  \\
      \cdots p^{b_1}& p^{b_0}& \mathcal{F}(M)  \\
      \cdots p^{c_1}& p^{c_0}& \mathcal{F}(N)  \\
	};
    \path[->,font=\scriptsize]
    (m-1-1) edge node[auto] {$ \tilde{A} $} (m-1-2)
    (m-1-2) edge node[auto] {$ \tilde{\pi_{a}} $} (m-1-3)
    
    (m-2-1) edge node[auto] {$ \tilde{B} $} (m-2-2)
    (m-2-2) edge node[auto] {$ \tilde{\pi_{b}} $} (m-2-3)

    (m-3-1) edge node[auto] {$ \tilde{\gamma} $} (m-3-2)
    (m-3-2) edge node[auto] {$ \tilde{\pi_{c}} $} (m-3-3)

    (m-1-1) edge node[auto] {$ \tilde{\varphi^{1}} $} (m-2-1)
    (m-1-2) edge node[auto] {$ \tilde{\varphi^{0}} $} (m-2-2)
    (m-1-3) edge node[auto] {$ \mathcal{F}(\varphi) $} (m-2-3)

    (m-2-1) edge node[auto] {$ \tilde{\psi^{1}} $} (m-3-1)
    (m-2-2) edge node[auto] {$ \tilde{\psi^{0}} $} (m-3-2)
    (m-2-3) edge node[auto] {$ \mathcal{F}(\psi) $} (m-3-3);

\end{tikzpicture}
\end{center}
Die Verkettung $ \mathcal{F}(\psi)\circ\mathcal{F}(\varphi) $ lässt das Diagramm also kommutieren. Da $ \psi^0\circ\varphi^0 $ bzw. $ \psi^1\circ\varphi^1 $ Lifts von  $ \psi\circ\varphi $  sind, ist eine derartige Abbildung nach universeller Eigenschaft von $ coker(A) $ eindeutig. Also gilt:
\[ \mathcal{F}(\psi\circ\varphi)=\mathcal{F}(\psi)\circ\mathcal{F}(\varphi).\]
\end{itemize}
\end{proof}

\begin{Bem+Def}\label{op}
Sei wieder $ \mathcal{A} $ eine $ R $-lineare, abelsche Kategorie, $ p\in \mathcal{A} $ ein Objekt, $ R' $ eine nicht-kommutative $ R $-Algebra endlichen Typs und 
\[ R'\overset{f}{\longrightarrow}\End_{\mathcal{A}}(p) \]
ein endlicher $ R $-Algebrenhomomorphismus. Die Kategorie der $ R' $-Rechtsmoduln ist natürlich isomorph zur Kategorie der $ (R')^{op} $-Linksmoduln über dem opponierten Ring $ (R')^{op} $. Man beachte, dass sich die Objekte und Morphismen von $ \Mor R' $ und $ (R')^{op}\Mod $ mengentheoretisch nicht unterscheiden; nur der zugrunde liegende Ring und dessen Skalarmultiplikation auf den Moduln ist in den beiden Kategorien formal verschieden. Über die natürliche Isomorphie erhalten wir durch die Verkettung
\[ (R')^{op}\Mod \overset{\sim}{\longrightarrow} \Mor R' \xrightarrow{ \_\otimes_{R'} p} \mathcal{A} \]
einen treuen, exakten $ R $-linearen Funktor von der Kategorie der $ (R')^{op} $-Linksmoduln nach $ \mathcal{A} $. Diesen Funktor bezeichnen wir mit $ p\otimes_{(R')^{op}} \_ $. \\
ist $ R' $ kommutativ so sind Rechtsmoduln und Linksmoduln das gleiche und auch die Definition der beiden Funktoren $ p\otimes_{(R')^{op}} \_  $ und $ \_ \otimes_{R'} p  $ wird gleich. 

%

\end{Bem+Def}

Analog zur Tensorkonstruktion in Lemma \ref{tensor}, können wir für ein $ p\in\mathcal{A} $ einen kontravarianten Funktor $ \Hom_{R'}(\_,p) $ definieren:

\begin{Lemma}\label{Hom}
Sei $ \mathcal{A} $ eine $ R $-lineare, abelsche Kategorie und $ R' $ eine nicht notwendig kommutative $ R $-Algebra endlichen Typs. Sei $ p\in\mathcal{A} $ und $ f:R'\rightarrow \End_{\mathcal{A}}(p) $ ein endlicher $ R $-Algebrenhomomorphismus. Dann existiert ein bis auf Isomorphie eindeutiger exakter, $ R $-linearer, kontravarianter Funktor

\[
\begin{array}{ccc}
\Hom_{R'}(\_,p):R'\Mod & \longrightarrow & \mathcal{A},\\
\end{array}
\]

so dass für $ R'\in R' $-Mod
\[
\Hom_{R'}(R',p)=p
\]
und für Morphismen
\[
\begin{array}{ccc}
\Hom_{R'}(\_,p):\End_{R'}(R')\cong R'&\rightarrow & \End_{\mathcal{A}}(p) \\
a & \mapsto & f(a) \\
\end{array}
\]
erfüllt sind.
\end{Lemma} 
 
\begin{proof}
Der Beweis funktioniert dual zu vorigem und wir werden ihn daher nicht in der gleichen Ausführlichkeit führen. Zur Vereinfachung der Notation bezeichnen wir den Funktor $ \Hom_{R'}(\_,p) $ für den Rest des Beweises mit $ \mathcal{G} $.\\
Wegen Additivität gilt wieder
\[
\mathcal{G}(R'^n)=\mathcal{G}\big (\bigoplus_{i=1}^{n}R'\big )=\bigoplus\limits_{i=1}^{n}\mathcal{G}(R')=\bigoplus\limits_{i=1}^{n}p=p^n.
\]
Wie in Lemma \ref{tensor} bereits erklärt, ist ein Morphismus $ A:R'^m\rightarrow R'^n $ zwischen freien Moduln das Gleiche wie eine Familie von Endomorphismen $ a_{i,j}\in\End_{R'}(R') $, wobei $ a_{i,j} $ durch die Verkettung 
\[ R'\xrightarrow{i_i}R'^m\xrightarrow{A}R'^n\xrightarrow{\pi_j}R' \]
gegeben wird. Durch Anwenden des kontravarianten, exakten Funktors $ \mathcal{G} $ erhalten wir eine Sequenz
\[ 
\mathcal{G}(R')\xleftarrow{\mathcal{G}(i_i)}\mathcal{G}(R')^m\xleftarrow{\mathcal{G}(A)}\mathcal{G}(R')^n\xleftarrow{\mathcal{G}(\pi_j)}\mathcal{G}(R'), \]
wobei nach Bemerkung \ref{Summen} die Abbildungen $ \mathcal{G}(\pi_j) $ Inklusionen in die direkte Summe und die Abbildungen $ \mathcal{G}(i_i) $ Projektionen aus der direkten Summe sind. Also entspricht die Verkettung  
\[ p\xleftarrow{\pi_i}p^m\xleftarrow{\mathcal{G}(A)}p^n\xleftarrow{i_j}p. \]
dem Endomorphismus $ f(a_{ij}):p\rightarrow p $.
Da $ \mathcal{G} $ additiv ist, ergibt sich für Morphismen zwischen freien Moduln 
\[
\begin{array}{ccc}
\mathcal{G}:\Hom_{R'}(R'^m,R'^n) & \rightarrow & \Hom_{\mathcal{A}}(p^m,p^n))\\
A & \mapsto & f(A)^T:=(f(a_{ij})_{i,j})^T.
\end{array}
\]
Für einen beliebigen $ R'$-Linksmodul $ M $ fixieren wir wieder eine freie Auflösung
\[
\cdots \rightarrow R'^{a_1}\xrightarrow{A}R'^{a_0}  \stackrel{\pi_a}{\twoheadrightarrow}M\cong coker(A)\rightarrow 0. 
\]
Da $ \mathcal{G} $ exakt und kontravariant ist, erhalten wir eine exakte Sequenz
\[
0\rightarrow \mathcal{G}(M)\hookrightarrow \mathcal{G}(R'^{a_0})\xrightarrow{\mathcal{G}(A)}\mathcal{G}(R'^{a_1}),
\]
bzw.
\[
0\rightarrow \Hom_{R'}(M,p)\rightarrow p^{a_0}\xrightarrow{\tilde{A}^T:=(f(A_{ij}))_{ij}^T}p^{a_1}.
\]
Also ist $ \mathcal{G}(M):=\Hom_{\mathcal{A}}(coker(A),p)=ker({\tilde{A}}^T) $.\\
Um $ \mathcal{G} $ auf Morphismen $ \varphi : M \rightarrow N $ zu berechnen, wählen wir wieder Lifts

\begin{center}
\begin{tikzpicture}[description/.style={fill=white,inner sep=2pt}]
    \matrix (m) [matrix of math nodes, row sep=3em,
    column sep=2.5em, text height=1.5ex, text depth=0.25ex]
    { R'^{a_1}& R'^{a_0}& M\cong coker(A) & 0 \\
      R'^{b_1}& R'^{b_0}& N\cong coker(B) & 0 \\
	};
    \path[->,font=\scriptsize]
    (m-1-1) edge node[auto] {$ A $} (m-1-2)
    (m-1-2) edge node[auto] {$ \pi_{a} $} (m-1-3)
    (m-1-3) edge node[auto] {} 			  (m-1-4)
    
    (m-2-1) edge node[auto] {$ B $} (m-2-2)
    (m-2-2) edge node[auto] {$ \pi_{b} $} (m-2-3)
    (m-2-3) edge node[auto] {} 			  (m-2-4)

    (m-1-1) edge node[auto] {$ \varphi^{1} $} (m-2-1)
    (m-1-2) edge node[auto] {$ \varphi^{0} $} (m-2-2)
    (m-1-3) edge node[auto] {$\varphi$} 	  (m-2-3);

\end{tikzpicture}
\end{center}

Durch Anwenden von $ \mathcal{G} $ erhalten wir

\begin{center}
\begin{tikzpicture}[description/.style={fill=white,inner sep=2pt}]
    \matrix (m) [matrix of math nodes, row sep=3em,
    column sep=3.5em, text height=2.5ex, text depth=0.25ex]
    { p^{a_1}& & p^{a_0}& ker \tilde{A}\simeq\mathcal{G}(M) & 0 \\
      p^{b_1}& & p^{b_0}& ker \tilde{B}\simeq\mathcal{G}(N) & 0 \\
	};
    \path[->,font=\scriptsize]
    (m-1-3) edge node[auto] {$ \tilde{A}^T $} (m-1-1)
    (m-1-4) edge node[auto] {$ \tilde{\pi_{a}}^T $} (m-1-3)
    (m-1-5) edge node[auto] {} 			  (m-1-4)
    
    (m-2-3) edge node[auto] {$ \tilde{B}^T $} (m-2-1)
    (m-2-4) edge node[auto] {$ \tilde{\pi_{b}}^T $} (m-2-3)
    (m-2-5) edge node[auto] {} 			  (m-2-4)

    (m-2-1) edge node[auto] {$ \tilde{\varphi^{1}}^T $} (m-1-1)
    (m-2-3) edge node[auto] {$ \tilde{\varphi^{0}}^T $} (m-1-3);
    \path[->,dotted](m-2-4) edge node[auto] {$\exists ! \mathcal{G}(\varphi)$} 	  (m-1-4);

\end{tikzpicture}
\end{center}

Die universelle Eigenschaft von $ ker (\tilde{A}^T)$ liefert eine eindeutige Abbildung \\
$ \mathcal{G}(\varphi):\mathcal{G}(N)\rightarrow \mathcal{G}(M) $, so dass das Diagramm kommutiert. Um zu prüfen, dass\\
$ \mathcal{G}(\psi\circ\varphi)=\mathcal{G}(\psi)\circ\mathcal{G}(\varphi) $ gilt, führt man das gleiche Argument durch wie in Lemma \ref{tensor}.
\end{proof} 

\begin{Bem}
Wir möchten an dieser Stelle noch einmal explizit darauf hinweisen, dass die Konstruktion der Funktoren 
\[ \_\otimes_{R'} p:\Mor R'\longrightarrow\mathcal{A} \]
und 
\[ \Hom_{R'}(\_,p):R'\Mod\longrightarrow\mathcal{A} \]
von den gewählten Auflösungen der $ R' $-Moduln abhängt. Verschiedene Auflösungen generieren aber äquivalente Funktoren. Für unser langfristiges Ziel, die wesentliche Surjektivität des Funktors $ \tilde{T}:\mathcal{A}\rightarrow\C(T) $ zu zeigen, genügt uns dies.
\end{Bem}

\begin{Lemma}\label{rechts}
Sei $ \mathcal{A} $ eine $ R $-lineare, abelsche Kategorie und 
\[ T:\mathcal{A}\longrightarrow R\Mod \]
ein $ R $-linearer, treuer, exakter Funktor. Sei $ p\in \mathcal{A} $ ein ausgezeichnetes Objekt. Die durch $ R $-Linearität gegebene Abbildung
\[R\longrightarrow \End_{\mathcal{A}}(p) \]
definiert nach Lemma \ref{Hom} einen kontravarianten Funktor 
\[ \Hom_R(\_,p):R\Mod\longrightarrow \mathcal{A}.\]
Die Verkettung
\[
\begin{array}{ccccc}
R\Mod & \overset{\Hom_R(\_,p)}{\longrightarrow} & \mathcal{A} & \overset{T}{\longrightarrow} & R\Mod\\
M & \longmapsto & \Hom_R(M,p) & \longmapsto & \Hom_R(M,Tp)\\
\end{array}
\]
entspricht bis auf Isomorphie dem gewöhnlichen kontravarianten Funktor $\Hom_R(\_,Tp)$ in $ R\Mod $. 
\end{Lemma}

\begin{proof}
Wir beginnen mit einer Vorbemerkung:
Da $ \mathcal{A} $ $ R $-linear ist, definiert eine Matrix  $ A\in Mat_R(n\times m) $ eine Abbildung $ A\in\Hom_{\mathcal{A}}(p^m,p^n) $, gegeben durch
\[
\begin{array}{ccccc}
A: & p^m & \longrightarrow & p^n\\
& (x_1,..x_m) & \longmapsto & (\sum_{i=1}^m(a_{1i}\cdot x_i),..,\sum_{i=1}^m (a_{ni}\cdot x_i)).
\end{array}
\]
Wenden wir den Funktor $ T $ auf diese Abbildung an, so erhalten wir, wegen der $ R $-Linearität von $ T $, in $ R\Mod $ die gleichermaßen definierte Abbildung
\[
\begin{array}{ccccc}
TA: & Tp^m & \longrightarrow & Tp^n\\
& (x_1,..x_m) & \longmapsto & (\sum_{i=1}^m(a_{1i}\cdot x_i),..,\sum_{i=1}^m (a_{ni}\cdot x_i)).
\end{array}
\]
\\
Seien nun $ M,N\in R\Mod $ endlich erzeugte $ R $-Moduln, $ \varphi\in\Hom_R(M,N) $ ein Homomorphismus und
\[
\begin{split}
R^{a_1}\overset{A}{\longrightarrow}R^{a_0}\longrightarrow M\longrightarrow 0 \\
R^{b_1}\overset{B}{\longrightarrow}R^{b_0}\longrightarrow N\longrightarrow 0  \\
\end{split}
\]
Präsentationen. Dann erhalten wir nach der Wahl von Lifts wie im Beweis von Lemma \ref{Hom} ein kommutatives Diagramm 

\begin{center}
\begin{tikzpicture}[description/.style={fill=white,inner sep=2pt}]
    \matrix (m) [matrix of math nodes, row sep=3em,
    column sep=2.5em, text height=1.5ex, text depth=0.25ex]
    { R^{a_1}& R^{a_0}& M & 0 \\
      R^{b_1}& R^{b_0}& N & 0 \\
	};
    \path[->,font=\scriptsize]
    (m-1-1) edge node[auto] {$ A $} (m-1-2)
    (m-1-2) edge node[auto] {$ $} (m-1-3)
    (m-1-3) edge node[auto] {} 			  (m-1-4)
    
    (m-2-1) edge node[auto] {$ B $} (m-2-2)
    (m-2-2) edge node[auto] {$ $} (m-2-3)
    (m-2-3) edge node[auto] {} 			  (m-2-4)

    (m-1-1) edge node[auto] {$ \varphi^{1} $} (m-2-1)
    (m-1-2) edge node[auto] {$ \varphi^{0} $} (m-2-2)
    (m-1-3) edge node[auto] {$\varphi$} 	  (m-2-3);

\end{tikzpicture}
\end{center}

Wir überprüfen, dass durch Anwenden von $ T\circ\Hom_R(\_,p) $ auf dieses Diagramm, das gleiche Diagramm entsteht, wie durch Anwenden von $ \Hom_R(\_,Tp) $.\\
Anwenden von $ \Hom_R(\_,p) $ liefert

\begin{center}
\begin{tikzpicture}[description/.style={fill=white,inner sep=2pt}]
    \matrix (m) [matrix of math nodes, row sep=3em,
    column sep=2.5em, text height=1.5ex, text depth=0.25ex]
    { p^{a_1}& p^{a_0}& \Hom_R(M,p)=ker(A^T) & 0 \\
      p^{b_1}& p^{b_0}& \Hom_R(N,p)=ker(B^T) & 0 \\
	};
    \path[->,font=\scriptsize]
    (m-1-2) edge node[auto] {$ A^T $} (m-1-1)
    (m-1-3) edge node[auto] {$ $} (m-1-2)
    (m-1-4) edge node[auto] {} 			  (m-1-3)
    
    (m-2-2) edge node[auto] {$ B^T $} (m-2-1)
    (m-2-3) edge node[auto] {$ $} (m-2-2)
    (m-2-4) edge node[auto] {} 	(m-2-3)

    (m-2-1) edge node[auto] {$ {\varphi^{1}}^T $} (m-1-1)
    (m-2-2) edge node[auto] {$ {\varphi^{0}}^T $} (m-1-2)
    (m-2-3) edge node[auto] {$\Hom_R(\_,p)(\varphi)$} 	  (m-1-3);

\end{tikzpicture}
\end{center}

Nach der Vorbemerkung, sowie der Tatsache, dass exakte Funktoren mit Kernen vertauschen, erhalten wir nun durch Anwenden von $ T $ folgendes kommutative Diagramm:

\begin{center}
\begin{tikzpicture}[description/.style={fill=white,inner sep=2pt}]
    \matrix (m) [matrix of math nodes, row sep=3em,
    column sep=2.5em, text height=1.5ex, text depth=0.25ex]
    { Tp^{a_1}& Tp^{a_0}& T(\Hom_R(M,p))=Ker(A^T) & 0 \\
      Tp^{b_1}& Tp^{b_0}& T(\Hom_R(N,p))=Ker(B^T) & 0 \\
	};
    \path[->,font=\scriptsize]
    (m-1-2) edge node[auto] {$ A^T $} (m-1-1)
    (m-1-3) edge node[auto] {$ $} (m-1-2)
    (m-1-4) edge node[auto] {} 			  (m-1-3)
    
    (m-2-2) edge node[auto] {$ B^T $} (m-2-1)
    (m-2-3) edge node[auto] {$ $} (m-2-2)
    (m-2-4) edge node[auto] {} 	(m-2-3)

    (m-2-1) edge node[auto] {$ {\varphi^{1}}^T $} (m-1-1)
    (m-2-2) edge node[auto] {$ {\varphi^{0}}^T $} (m-1-2)
    (m-2-3) edge node[auto] {$(T\circ\Hom_R(\_,p))(\varphi)$} 	  (m-1-3);

\end{tikzpicture}
\end{center}

Wenden wir dagegen den linksexakten Funktor $ \Hom_R(\_,Tp) $ in $ R\Mod $ auf die Präsentationen und Lifts an, erhalten wir folgendes kommutative Diagramm.

\begin{center}
\begin{tikzpicture}[description/.style={fill=white,inner sep=2pt}]
    \matrix (m) [matrix of math nodes, row sep=3em,
    column sep=2.5em, text height=1.5ex, text depth=0.25ex]
    { \Hom_R(R^{a_1},Tp)& \Hom_R(R^{a_0},Tp) & \Hom_R(M,Tp)=Ker(A^T) & 0 \\
	  \Hom_R(R^{b_1},Tp)& \Hom_R(R^{b_0},Tp) & \Hom_R(N,Tp))=Ker(B^T) & 0 \\
	};
    \path[->,font=\scriptsize]
    (m-1-2) edge node[auto] {$ A^T $} (m-1-1)
    (m-1-3) edge node[auto] {$ $} (m-1-2)
    (m-1-4) edge node[auto] {} 			  (m-1-3)
    
    (m-2-2) edge node[auto] {$ B^T $} (m-2-1)
    (m-2-3) edge node[auto] {$ $} (m-2-2)
    (m-2-4) edge node[auto] {} 	(m-2-3)

    (m-2-1) edge node[auto] {$ {\varphi^{1}}^T $} (m-1-1)
    (m-2-2) edge node[auto] {$ {\varphi^{0}}^T $} (m-1-2)
    (m-2-3) edge node[auto] {$(\Hom_R(\_,Tp))(\varphi)$} 	  (m-1-3);

\end{tikzpicture}
\end{center}

Verwenden wir nun die kanonische Isomorphie von $ \Hom_R(R^n,Tp)\cong (Tp)^n $, so sehen wir, dass sich die beiden letzten Diagramme bis auf Isomorphie entsprechen. Insbesondere müssen $(T\circ\Hom_R(\_,p))(\varphi)$ sowie $ \Hom_R(\_,Tp)(\varphi) $ dem eindeutigen Morphismus von $ Ker(B^T) $ nach $ Ker(A^T) $ entsprechen, der das Diagramm kommutativ macht, also gilt
\[(T\circ\Hom_R(\_,p))(\varphi)\cong\Hom_R(\_,Tp)(\varphi). \] 
\end{proof}

\section[Die Operationen auf $ \Hom_R(Tp,p) $]{Die Operationen auf $ \Hom_R(Tp,p) $}

\begin{Bem}\label{quark}
Wir haben über den Funktor $ \Hom_R(\_,p) $ mit dem Objekt\\ $ \Hom_R(Tp,p)\in\mathcal{A} $ ein Urbild unter $ T $ von $ \End_R(Tp)\in R\Mod $ konstruiert. Wir möchten nun in $ \mathcal{A} $ ein Urbild von 
\[ \End(T_{|\{p\}})=\{\varphi\in\End_R(Tp)|\forall \alpha\in\End_{\mathcal{A}}(p):\varphi\circ T\alpha=T\alpha\circ \varphi\} \]
konstruieren. Hierfür brauchen wir Endomorphismen auf $ \Hom_R(Tp,p) $, welche über $ T $ in $ R\Mod $ auf das gewöhnlichen Vor- und Nachschalten auf $ \End_R(Tp) $ abgebildet werden. Sei $ \varphi\in\End_R(Tp) $ ein beliebiger Morphismus. In Lemma \ref{rechts} haben wir gesehen, dass die Verkettung
\[
\begin{array}{ccccc}
R\Mod & \overset{\Hom_R(\_,p)}{\longrightarrow} & \mathcal{A} & \overset{T}{\longrightarrow} & R\Mod\\
Tp & \longmapsto & \Hom_R(Tp,p) & \longmapsto & \Hom_R(Tp,Tp)\\
\end{array}
\]
den gewöhnlichen $ \Hom_R(\_,Tp) $ Funktor in $ R\Mod $ liefert. Der Endomorphismus \\$ \Hom_R(\_,p)(\varphi) $ wird also durch $ T $ auf die gewöhnliche Rechtsoperation $ \circ\varphi $ auf $ \End_R(Tp) $ abgebildet. Wir werden den Morphismus $ \Hom_R(\_,p)(\varphi) $ daher ab sofort ebenfalls mit $ \circ\varphi $ notieren. In Lemma \ref{links} konstruieren wir den Endomorphismus auf $ \Hom_R(Tp,p) $, der unter $ T $ auf die gewöhnliche Linksoperation auf  $ \End_R(Tp) $ abgebildet wird.
\end{Bem}

\begin{Lemma}\label{vertauschen}
Sei $ \mathcal{A} $ eine $ R $-lineare, abelsche Kategorie und $ p $ ein Objekt aus $ \mathcal{A} $. Sei $ A\in Mat_R(n,m) $ eine Matrix mit Koeffizienten in $ R $. Wir fassen $ A $ als Abbildung 
\[ A:p^m\longrightarrow p^n \]
auf. Sei $ \alpha\in\End_{\mathcal{A}}(p) $ ein beliebiger Endomorphismus. Bezeichne $ \alpha^n\in\End_{\mathcal{A}}(p^n) $ den Morphismus, der auf jedem direkten Summanden den Endomorphismus $ \alpha $ ausführt. Dann gilt
\[ \alpha^m\circ A=A\circ\alpha^n. \]
\end{Lemma}

\begin{proof}
Ein Morphismus $ f:p^m\rightarrow p^n $ ist gegeben durch eine Familie $ f_{ij}:p\rightarrow p $, so dass $ f_{ij}=\pi_i\circ f\circ i_j $ ist. Um zu sehen, dass die Abbildungen $ \alpha $ und $ A=(a_{ij})_{ij} $ kommutieren, müssen wir also sehen, dass sich die induzierten Abbildungen $ \pi_k\circ A\circ\alpha^n\circ i_l $ und $\pi_k\circ \alpha^m \circ A\circ i_l $ für alle $ k\in\{1..m\}$ und $l\in\{1..n\} $ entsprechen.\\
Für $ x\in p $ gilt nun aber 
\[ (\pi_k\circ A\circ\alpha^n\circ i_l)(x)=a_{kl}\cdot\alpha(x) \]
\[ (\pi_k\circ \alpha^m \circ A\circ i_l)(x)=\alpha(a_{kl} x). \]
Da $ \alpha $ $ R $-linear ist, sind die Abbildungen gleich.
\end{proof}

\begin{Konstruktion}\label{konstr}
Analog zur Definition der Funktoren 
\[
\begin{array}{cccc}
 \Hom_R(\_,p): & R\Mod & \longrightarrow & \mathcal{A} \\
				& M & \longmapsto & \Hom_R(M,p)
\end{array}
\]
und
\[
\begin{array}{cccc}
 p\otimes_R \_: & R\Mod & \longrightarrow & \mathcal{A} \\
				& M & \longmapsto & p\otimes_R M
\end{array}
\]
aus Lemma \ref{Hom} bzw. Lemma \ref{tensor} können wir für einen beliebigen, endlich erzeugten $ R $-Modul $ M $ wie folgt die Funktoren
\[
\begin{array}{cccc}
 \Hom_R(M,\_): & \mathcal{A} & \longrightarrow & \mathcal{A} \\
				& p & \longmapsto & \Hom_R(M,p)
\end{array}
\]
und
\[
\begin{array}{cccc}
 \_\otimes_R M : & \mathcal{A} & \longrightarrow & \mathcal{A} \\
				& p & \longmapsto & p\otimes_R M
\end{array}
\]
definieren:\\
Sei wieder 
\[ R^n\overset{A}{\longrightarrow}R^m\twoheadrightarrow M\longrightarrow 0 \]
eine feste Auflösung von $ M $. Für $ p\in\mathcal{A} $ definieren wir das Objekt $ \Hom_R(M,p) $ als den Kern der Abbildung 
\[ p^n\overset{A^T}{\longrightarrow}p^m. \]
Sei für $ \alpha:p\rightarrow q$ ein beliebiger Morphismus in $ \mathcal{A} $. Dann hat das Diagramm
\begin{center}
\begin{tikzpicture}[description/.style={fill=white,inner sep=2pt}]
    \matrix (m) [matrix of math nodes, row sep=4.5em,
    column sep=3.5em, text height=1.5ex, text depth=0.25ex]
    { p^m& p^n& \Hom_R(M,p) & 0 \\
      q^m & q^n & \Hom_R(M,q) & 0 \\
	};
    \path[->,font=\scriptsize]
    (m-1-2) edge node[auto] {$ A^T $} (m-1-1)
    (m-1-3) edge node[auto] {$ $} (m-1-2)
    (m-1-4) edge node[auto] {} 			  (m-1-3)
    
    (m-2-2) edge node[auto] {$ A^T $} (m-2-1)
    (m-2-3) edge node[auto] {$ $} (m-2-2)
    (m-2-4) edge node[auto] {} 			  (m-2-3)

    (m-1-1) edge node[auto] {$ \alpha^m $} (m-2-1)
    (m-1-2) edge node[auto] {$ \alpha^n $} (m-2-2);

\end{tikzpicture}
\end{center}
für alle $ p,q\in\mathcal{A} $ exakte Zeilen. Es kommutiert nach Lemma \ref{vertauschen}. Wir definieren $ \Hom_R(M,\_)(\alpha) $ als die eindeutige Abbildung 
\[ \Hom_R(M,p)\longrightarrow\Hom_R(M,q), \]
die das Diagramm kommutativ lässt. Für zwei Morphismen $ p\xrightarrow{\alpha}q\xrightarrow{\beta}r $ gilt offensichtlich 
\[ \Hom_R(M,\_)(\beta\circ\alpha) = \Hom_R(M,\_)(\beta) \circ \Hom_R(M,\_)(\alpha). \]
Also ist $ \Hom_R(M,\_) $ ein Funktor und genauso überprüft man, dass die analoge Konstruktion $ \_\otimes_R M $ einen Funktor definiert. 

\end{Konstruktion}

\begin{Lemma}\label{links}

Sei wieder $ \mathcal{A} $ eine $ R $-lineare, abelsche Kategorie und 
\[ T:\mathcal{A}\longrightarrow R\Mod \]
ein $ R $-linearer, treuer, exakter Funktor und $ p\in \mathcal{A} $ ein ausgezeichnetes Objekt.
Für einen Endomorphismus $ \alpha\in\End_{\mathcal{A}}(p) $ liefert die Verkettung

\[
\begin{array}{ccccc}
 \mathcal{A} & \overset{\Hom_R(Tp,\_)}{\longrightarrow} & \mathcal{A} & \overset{T}{\longrightarrow} & R\Mod  \\
		  p & \longmapsto & \Hom_R(Tp,p) & \longmapsto & \Hom_R(Tp,Tp) \\
		  \alpha & \longmapsto & \Hom_R(Tp,\_)(\alpha) & \longmapsto & (T\alpha)\circ \\
\end{array}
\]

bis auf Isomorphie die gewöhnliche Linkswirkung $ (T\alpha)\circ $ auf $ \End_R(Tp) $.\\

\end{Lemma}

\begin{proof}
Sei 
\[ R^m\xrightarrow{A}R^n\twoheadrightarrow Tp\longrightarrow 0 \]
eine feste Auflösung von $ Tp $.
$ \Hom_R(Tp,\_)(\alpha) $ ist definiert als der Morphismus, der folgendes Diagramm kommutativ macht. 

\begin{center}
\begin{tikzpicture}[description/.style={fill=white,inner sep=2pt}]
    \matrix (m) [matrix of math nodes, row sep=4.5em,
    column sep=3.5em, text height=1.5ex, text depth=0.25ex]
    { p^m& p^n& \Hom_R(Tp,p) & 0 \\
      p^m & p^n & \Hom_R(Tp,p) & 0 \\
	};
    \path[->,font=\scriptsize]
    (m-1-2) edge node[auto] {$ A^T $} (m-1-1)
    (m-1-3) edge node[auto] {$ $} (m-1-2)
    (m-1-4) edge node[auto] {} 			  (m-1-3)
    
    (m-2-2) edge node[auto] {$ A^T $} (m-2-1)
    (m-2-3) edge node[auto] {$ $} (m-2-2)
    (m-2-4) edge node[auto] {} 			  (m-2-3)

    (m-1-1) edge node[auto] {$ \alpha^m $} (m-2-1)
    (m-1-2) edge node[auto] {$ \alpha^n $} (m-2-2);
    \path[->,densely dotted]
    (m-1-3) edge node[description] {$ \Hom_R(Tp,\_)(\alpha) $} 	  (m-2-3);

\end{tikzpicture}
\end{center}

Da $ T $ exakt ist, liefert uns dies folgendes kommutative Diagramm mit exakten Zeilen in $ R $-Mod:

\begin{center}
\begin{tikzpicture}[description/.style={fill=white,inner sep=2pt}]
    \matrix (m) [matrix of math nodes, row sep=4.5em,
    column sep=3.5em, text height=1.5ex, text depth=0.25ex]
    { Tp^m& Tp^n& \Hom_R(Tp,Tp) & 0 \\
      Tp^m & Tp^n & \Hom_R(Tp,Tp) & 0 \\
	};
    \path[->,font=\scriptsize]
    (m-1-2) edge node[auto] {$ A^T $} (m-1-1)
    (m-1-3) edge node[auto] {$ $} (m-1-2)
    (m-1-4) edge node[auto] {} 			  (m-1-3)
    
    (m-2-2) edge node[auto] {$ A^T $} (m-2-1)
    (m-2-3) edge node[auto] {$ $} (m-2-2)
    (m-2-4) edge node[auto] {} 			  (m-2-3)

    (m-1-1) edge node[auto] {$ T\alpha^m $} (m-2-1)
    (m-1-2) edge node[auto] {$ T\alpha^n $} (m-2-2)
    (m-1-3) edge node[description] {$ T\circ\Hom_R(Tp,\_)(\alpha) $} 	  (m-2-3);

\end{tikzpicture}
\end{center}

Andererseits kommutiert in $ R $-Mod das Diagramm

\begin{center}
\begin{tikzpicture}[description/.style={fill=white,inner sep=2pt}]
    \matrix (m) [matrix of math nodes, row sep=4.5em,
    column sep=3.5em, text height=1.5ex, text depth=0.25ex]
    { \Hom_R(R^m,Tp) & \Hom_R(R^n,Tp) & \Hom_R(Tp,Tp) & 0 \\
		\Hom_R(R^m,Tp) & \Hom_R(R^n,Tp) & \Hom_R(Tp,Tp) & 0 \\
	};
    \path[->,font=\scriptsize]
    (m-1-2) edge node[auto] {$ \circ A $} (m-1-1)
    (m-1-3) edge node[auto] {$ $} (m-1-2)
    (m-1-4) edge node[auto] {} 			  (m-1-3)
    
    (m-2-2) edge node[auto] {$ \circ A $} (m-2-1)
    (m-2-3) edge node[auto] {$ $} (m-2-2)
    (m-2-4) edge node[auto] {} 			  (m-2-3)

    (m-1-1) edge node[auto] {$ T\alpha\circ $} (m-2-1)
    (m-1-2) edge node[auto] {$ T\alpha\circ $} (m-2-2)
    (m-1-3) edge node[description] {$ T\alpha\circ $} 	  (m-2-3);

\end{tikzpicture}
\end{center}

welches sich unter Verwendung des kanonischen Isomorphismus $ \Hom_R(R^a,Tp)\cong Tp^a $ ebenfalls umschreiben lässt zu

\begin{center}
\begin{tikzpicture}[description/.style={fill=white,inner sep=2pt}]
    \matrix (m) [matrix of math nodes, row sep=4.5em,
    column sep=3.5em, text height=1.5ex, text depth=0.25ex]
    { Tp^m& Tp^n& \Hom_R(Tp,Tp) & 0 \\
      Tp^m & Tp^n & \Hom_R(Tp,Tp) & 0 \\
	};
    \path[->,font=\scriptsize]
    (m-1-2) edge node[auto] {$  A^T $} (m-1-1)
    (m-1-3) edge node[auto] {$ $} (m-1-2)
    (m-1-4) edge node[auto] {} 			  (m-1-3)
    
    (m-2-2) edge node[auto] {$ A^T $} (m-2-1)
    (m-2-3) edge node[auto] {$ $} (m-2-2)
    (m-2-4) edge node[auto] {} 			  (m-2-3)

    (m-1-1) edge node[auto] {$ T\alpha^m $} (m-2-1)
    (m-1-2) edge node[auto] {$ T\alpha^n $} (m-2-2)
    (m-1-3) edge node[description] {$ T\alpha\circ $} 	  (m-2-3);

\end{tikzpicture}
\end{center}

$ T\circ\Hom_R(Tp,\_)(\alpha) $ und $ T\alpha\circ $ sind also jeweils die nach universeller Eigenschaft von $ Ker(A^T) $ eindeutig induzierte Abbildung $ \Hom_R(Tp,Tp)\longrightarrow \Hom_R(Tp,Tp) $, welche das Diagramm kommutativ machen. Daher gilt
\[ T\circ\Hom_R(Tp,\_)(\alpha)=(T\alpha)\circ.\qedhere\]

\end{proof}

\begin{Konvention}
Wir bezeichnen die von dem Morphismus $ \Hom_R(Tp,\_)(\alpha) $ induzierte Linksoperation auf $ \Hom_R(Tp,p) $ mit $ \alpha\circ $. Dies ist intuitiv, da $ T(\alpha\circ)=(T\alpha)\circ $ ist. 
\end{Konvention}

\begin{Bem}\label{gdg}
Sei $ \mathcal{A} $ eine $ R $-lineare, abelsche Kategorie, 
\[ T:\mathcal{A}\longrightarrow R\Mod \]
ein $ R $-linearer, treuer, exakter Funktor und  $ p\in \mathcal{A} $ ein Objekt. Für jeden Endomorphismus  $\alpha\in\End_{\mathcal{A}}( p) $ erhalten wir nach Lemma \ref{links} einen Endomorphismus $ (\alpha\circ) $ auf $ \Hom_R(Tp,p) $. Dies induziert eine $ R $-lineare Linksoperation
\[
\begin{array}{cccc}
\rho_L: & \End_{\mathcal{A}}(p)\times\Hom_R(Tp,p) & \longrightarrow & \Hom_R(Tp,p)\\
		&  (\alpha,x) & \longmapsto & \alpha\circ x
\end{array}
\]
auf $ \Hom_R(Tp,p) $. Analog induzieren die Endomorphismen von $ \End_R(Tp) $ nach Bemerkung \ref{quark} eine $ R $-lineare Rechtsoperation
\[
\begin{array}{cccc}
\rho_R: & \Hom_R(Tp,p)\times\End_R(Tp) & \longrightarrow & \Hom_R(Tp,p) \\
		&  (\varphi,x) & \longmapsto & x\circ\varphi
\end{array}
\]
auf $ \Hom_R(Tp,p) $.
\end{Bem}

\begin{Lemma}\label{Asso}

Die Operationen $ \rho_L $ und $ \rho_R $ aus Bemerkung \ref{gdg} sind miteinander verträglich, das heißt für $ x\in\Hom_R(Tp,p)$ , $\varphi\in\End_R(Tp)$ und $\alpha\in\End_{\mathcal{A}}(p) $ gilt
\[ \alpha\circ(x\circ\varphi)=(\alpha\circ x)\circ\varphi.\]

\end{Lemma}

\begin{proof}

Wir müssen zeigen, dass für alle $ \alpha\in\End_{\mathcal{A}}(p) $ und $ \varphi\in\End_R(Tp,Tp) $ die induzierten Endomorphismen $ \alpha\circ $ und $ \circ\varphi $ auf $ \Hom_R(Tp,p) $ kommutieren. 
Wenden wir $ T $ auf die Endomorphismen $ \alpha\circ $ und $ \circ\varphi $ an, so erhalten wir die übliche Rechts- bzw. Linksoperation $ (T \alpha)\circ $ bzw. $ \circ\varphi $ auf $ \Hom_R(Tp,Tp) $. Für alle $ A\in\Hom_R(Tp,Tp) $ gilt
\[ T\alpha\circ (A\circ\varphi) - (T\alpha\circ A)\circ\varphi=0.\]
Da $ T $ treu ist, folgt daraus für alle $ x\in\Hom_R(Tp,p) $
\[\alpha\circ(x\circ\varphi)-(\alpha\circ x)\circ\varphi=0,\]
also
\[\alpha\circ(x\circ\varphi)=(\alpha\circ x)\circ\varphi.\]
\end{proof}
\newpage
\section{Der Funktor $ \tilde{T} $ ist wesentlich surjektiv}

\begin{Def}
Sei $ \mathcal{A} $ eine beliebige abelsche Kategorie und $ F\subset\mathcal{A} $ eine volle Unterkategorie. Wir bezeichnen mit $ \langle F \rangle $ die von $ F $ erzeugte abelsche Kategorie, d. h. die kleinste abelsche Kategorie, die alle Objekte und Morphismen von $ F $ enthält. Die Kategorie $ \langle F \rangle $ enthält also insbesondere alle endlichen direkten Summen, Kokerne und Kerne, sowie Subquotienten von Objekten in $ F $. Außerdem gilt offensichtlich $ \langle F \rangle\subseteq \mathcal{A}$.\\
Für die von einem Objekt $ p\in\mathcal{A} $ erzeugte abelsche Kategorie schreiben wir $ \langle p \rangle $ anstatt  $ \langle \{p\} \rangle $.
\end{Def}

Wir kehren nun zu unserem Ziel zurück, zu zeigen, dass eine abelsche Kategorie $ \mathcal{A} $ mit treuem, exakten Funktor $ T $ in die $ R $-Moduln äquivalent zu ihrer Diagrammkategorie ist. Wir bedienen uns hierbei einer ähnlichen Strategie, wie sie Deligne in \cite[S.128 ff]{MR654325} für den Körperfall verwendet hat. Für ein $ p\in\mathcal{A} $ werden wir, mit Hilfe der Konstruktion des Objekts $ \Hom_R(Tp,p) $ aus Abschnitt 3.1, sowie den Operationen $ \alpha\circ $ und $ \circ\varphi $ auf $ \Hom_R(Tp,p) $, einen Funktor 
\[ \End(T_{|\{p\}})\Mod\longrightarrow  \langle p \rangle \]
konstruieren, sodass die Verkettung
\[\End(T_{|\{p\}})\Mod\longrightarrow  \langle p \rangle\overset{\tilde{T}_ {|\langle p \rangle}}{\longrightarrow} \End(T_{|\{p\}})\Mod \]
dem Identitätsfunktor auf $ \End(T_{|\{p\}})\Mod $ entspricht. Der direkte Limes über Unterkategorien von $ \mathcal{A} $ induziert dann eine Faktorisierung des Identitätsfunktors
\[\C(T_{\mathcal{A}})\longrightarrow  \mathcal{A} \overset{\tilde{T}}{\longrightarrow} \C(T_{\mathcal{A}}), \]
was zeigt, dass $ \tilde{T} $ wesentlich surjektiv ist. \\
Wir konstruieren zunächst ein Urbild von $ \End(T_{|\{p\}})\Mod $ unter $ T $:

\begin{Lemma}\label{ll}
Sei $ \mathcal{A} $ eine $ R $-lineare abelsche Kategorie und
\[ \mathcal{A}\overset{T}{\longrightarrow}R\Mod \]
ein treuer, exakter, $ R $-linearer Funktor. Sei $ p\in\mathcal{A} $ ein beliebiges Objekt. Dann existiert in $ \langle p \rangle $ ein Objekt $ X(p)\subset\Hom_R(Tp,p) $, zusammen mit einer Rechtsoperation
\[
\begin{array}{cccc}
 X(p)\times \End(T_{|\{p\}})\Mod & \rightarrow & X(p) \\
 (x,\varphi) & \mapsto & x\circ\varphi,
\end{array}
\]
sodass $ T(X(p))\cong\End(T_{|\{p\}}) $ und die Rechtsoperation auf $ X(p) $ nach Anwenden von $ T $ zur gewöhnlichen Rechtsoperation auf $ \End(T_{|\{p\}}) $ wird. 
\end{Lemma}

\begin{proof}
$ \End_{\mathcal{A}}(p) $ ist ein $ R $-Modul. Da der Funktor $ T $ $ R $-linear ist, ist $ T(\End_{\mathcal{A}}(p))\subset\End_R(Tp) $ ein Untermodul. Da $ R $ noethersch ist, ist $ \End_R(Tp) $ nach Lemma \ref{hom_noethersch} endlich erzeugt. Also ist nach Lemma \ref{noether} auch $ T(\End_{\mathcal{A}}(p)) $ ein endlich erzeugter Modul. Da $ T $ treu ist, ist dann auch $ \End_{\mathcal{A}}(p) $ als $ R $-Modul endlich erzeugt.\\
Sei $ (\alpha_1,...\alpha_n) $ ein Erzeugendensystem von $ \End_{\mathcal{A}}(p) $ als $ R $-Modul. Für alle $ \alpha_i $ operiert nun nach Lemma \ref{links} $ \alpha_i\circ $ von links und nach Lemma \ref{rechts} $ \circ T\alpha_i $ von rechts auf dem Objekt $ \Hom_R(Tp,p) $. Wir definieren $ X(p)\in\mathcal{A} $ als Kern der Abbildung
\[ ((\alpha_i\circ)-(\circ T\alpha_i))_{i\in\{1..n\}}:\Hom_R(Tp,p)\rightarrow\bigoplus_{i=n}^n\Hom_R(Tp,p). \]
Da $ \mathcal{A} $ abelsch ist, existiert dieses Objekt in $ \mathcal{A} $. Wenden wir den Funktor $ T $ auf diese Abbildung an, so erhalten wir nach Lemma \ref{rechts} bzw. \ref{links} die Abbildung 
\[ ((T\alpha_i\circ)-(\circ T\alpha_i))_{i\in\{1..n\}}:\Hom_R(Tp,Tp)\rightarrow\bigoplus_{i=n}^n\Hom_R(Tp,Tp). \]
Da $ (\alpha_1...\alpha_n) $ ein Erzeugendensystem von $ \End_{\mathcal{A}}(p) $ ist, gilt also:
\[
\begin{split} 
ker (((T\alpha_i\circ)-(\circ T\alpha_i))_{i\in\{1..n\}})& =\{\varphi\in\End_R(Tp)|\forall i\in \{1\dots n\}:T\alpha_i\circ\varphi=\varphi\circ T\alpha_i\} \\
	 & =\{\varphi\in\End_R(Tp)|\forall \alpha\in \End_{\mathcal{A}}(p):T\alpha\circ\varphi=\varphi\circ T\alpha\} \\
	 & =\End(T_{|\{p\}}).
\end{split}
\]
Weil $ T $ exakt ist, gilt
\[ T(X(p))=\End(T_{|\{p\}}).\]
Da $  \langle p \rangle $ stabil unter endlichen direkten Summen und Kernen ist, ist $ \Hom_R(Tp,p) $ in der von $ p $ erzeugten abelschen Kategorie und mit dem gleichen Argument auch $ X(p) $.
Als nächstes möchten wir sehen, dass sich die Rechtsoperation

\[
\begin{array}{cccc}
\rho_R: & \Hom_R(Tp,p)\times\End_R(Tp) & \rightarrow & \Hom_R(Tp,p) \\
		&  (x,\varphi) & \mapsto & x\circ\varphi
\end{array}
\]
zu einer Operation 

\[
\begin{array}{cccc}
 X(p)\times \End(T_{|\{p\}}) & \rightarrow & X(p) \\
 (x,\varphi) & \mapsto & x\circ\varphi
\end{array}
\]
einschränken lässt. Hierfür rechnen wir für $ x\in X(p),\varphi\in \End(T_{|\{p\}}) $ und $ \alpha_i $ aus dem Erzeugendensystem von $ \End_{\mathcal{A}}(p) $:
\[
\begin{split} 
					       \alpha_i\circ(x\circ\varphi)  \overset{\ref{Asso}}{=} (\alpha_i\circ x)\circ\varphi) 
  \overset{x\in X(p)}{=}  (x\circ T\alpha_i)\circ\varphi)  \overset{\varphi\in \End(T_{|\{p\}})}{=} (x\circ\varphi)\circ T\alpha_i.  \\
\end{split}
\]
Also liegt $ x\circ\varphi $ in $ X(p) $. \\
Die Rechtsoperation
\[
\begin{array}{cccc}
\rho_R: & \Hom_R(Tp,p)\times\End_R(Tp) & \rightarrow & \Hom_R(Tp,p) \\
		&  (x,\varphi) & \mapsto & x\circ\varphi
\end{array}
\]
induziert nach Lemma \ref{links} durch Anwenden des Funktors $ T $ den Morphismus $ \circ T\varphi $ auf $ \Hom_R(Tp,Tp) $. Die Einschränkung auf 
\[ X(p)\times \End(T_{|\{p\}}) \rightarrow  X(p)  \] liefert demnach nach Anwenden von $ T $ wie gefordert die Operation $ \circ T\varphi $ auf $ \End(T_{|\{p\}}) $.\\

\end{proof}

Wir bezeichnen ab sofort die $ R $-Algebra $ \End(T_{|\{p\}}) $ vereinfachend mit $ E(p) $.

\begin{Bem}\label{17}
Die Rechtswirkung 
\[ \rho_R:  X(p)\times E(p) \rightarrow  X(p) \]
definiert einen $ R $-Algebrenhomomorphismus
\[ 
\begin{array}{cccc}
f: & \big ( E(p) \big ) ^{op} & \longrightarrow & \End_{\mathcal{A}}(X(p)) \\
 &  \varphi & \longmapsto & \circ \varphi 
\end{array}
\]

aus dem opponierten Ring. Dies induziert nach Lemma \ref{tensor} einen treuen, exakten, $ R $-linearen Funktor

\[ 
\begin{array}{cccc}
\_\otimes_{E(p)^{op}} X(p): & \Mor ( E(p) \big ) ^{op} & \longrightarrow & \mathcal{A} \\
 &  M & \longmapsto & M\otimes_{E(p)^{op}} X(p)
\end{array}
\]
Das Objekt $ M\otimes_{E(p)^{op}} X(p) $ liegt nach Konstruktion für alle $ M $ in der von $ X(p) $ erzeugten abelschen Kategorie $ \langle X(p) \rangle $. Da $ X(p) $ ein Kern von direkten Summen von $ p $ ist, nimmt der Funktor also nur Werte in $ \langle p \rangle $ an. Nach Bemerkung \ref{op} erhalten wir analog zu $ \_\otimes_{E(p)^{op}} X(p) $ einen treuen, exakten, $ R $-linearen Funktor

\[ 
\begin{array}{cccc}
X(p) \otimes_{E(p)} \_: & E(p)\Mod & \longrightarrow & \langle p \rangle \\
 &  M & \longmapsto & p \otimes_{E(p)} M
\end{array}
\]

\end{Bem}

\begin{Korollar}\label{id}
Sei $ \mathcal{A} $ eine abelsche, $ R $-lineare Kategorie und 
\[ T:\mathcal{A}\longrightarrow R\Mod \]
ein treuer, exakter, $ R $-linearer Funktor, also insbesondere eine Darstellung. Sei 
\[  \mathcal{A}\xrightarrow{\tilde{T}}\mathcal{C}(T_{\mathcal{A}})\xrightarrow{ff_T}R\Mod   \]
die Faktorisierung von $ T $ über seine Diagrammkategorie.\\
Dann entspricht die Verkettung
\[
\begin{array}{ccccc}
E(p)\Mod & \overset{X(p)\otimes_{E(p)}\_} {\longrightarrow} & \langle p \rangle & \overset{\tilde{T}_{|\langle p \rangle}}{\longrightarrow} & E(p)\Mod \\
				 M & \longmapsto & X(p)\otimes_{E(p)}M & \mapsto & \tilde{T}( X(p)\otimes_{E(p)}M ) \\
\end{array}
\]
bis auf Isomorphie dem Identitätsfunktor auf $ E(p)\Mod $.
\end{Korollar}

\begin{proof}
Für $ E(p) $ als $ E(p) $-Linksmodul gilt nach Konstruktion 
\[ \tilde{T} \big (  X(p)\otimes_{E(p)}E(p)\big )=\tilde{T}\big ( X(p)\big ) = E(p). \]
Sei $ \varphi\in E(p) $. Betrachten wir das Bild des Edomorphismus $ \circ\varphi\in \End_{E(p)}\big ( E(p) \big ) $ unter folgender Kette von Funktoren:
\[ E(p)\Mod \overset{\sim}{\longrightarrow} \Mor  E(p)^{op}  \xrightarrow{\_\otimes_{E(p)^{op}} X(p)}  \langle p \rangle \overset{\tilde{T}_{|\langle p \rangle}}{\longrightarrow} E(p)\Mod \]
Der Endomorphismus $ \circ\varphi $ entspricht in $ \Mor  E(p)^{op} $ der gewöhnlichen Multiplikation von rechts mit dem Ringelement $ \varphi\in E(p)^{op} $. Nach Definition ist das Bild dieses Morphismus unter dem Funktor $ X(p)\otimes_{E(p)}\_ $ gerade $ f(\varphi) $, also der Endomorphismus $ \circ\varphi\in\End_{\mathcal{A}}(p) $, welcher nach Bemerkung \ref{quark} unter $ \tilde{T} $ auf $ \circ\varphi\in \End_{E(p)}(E(p)) $ abgebildet wird. Die natürliche Rechtsoperation auf $ E(p) $ wird also unter der Kette von Funktoren auf sich selbst abgebildet. Dies impliziert, dass die Kette von Funktoren schon für alle Endomorphismen von $ E(p) $ dem Identitätsfunktor entspricht, da die Rechtswirkung von $ \varphi $ auf einem Element $ e\in E(p) $ der Linkswirkung von $ e $ auf dem Element $ \varphi $ entspricht. \\

Für beliebige $ E(p) $-Moduln rechnet man dies über Auflösungen wieder einfach nach:\\
Sei $ g:M\rightarrow N $ ein Morphismus zwischen beliebigen $ E(p) $-Moduln. Nach der Wahl von Auflösungen für $ M $ und $ N $, sowie nach der Wahl von Lifts für $ g $ erhalten wir ein kommutatives Diagramm mit exakten Zeilen
\begin{center}
\begin{tikzpicture}[description/.style={fill=white,inner sep=2pt}]
    \matrix (m) [matrix of math nodes, row sep=4.5em,
    column sep=3.5em, text height=1.5ex, text depth=0.25ex]
    { E(p)^m & E(p)^n& M & 0 \\
      E(p)^a & E(p)^b & N & 0 \\
	};
    \path[->,font=\scriptsize]
    (m-1-1) edge node[auto] {$  A $} (m-1-2)
    (m-1-2) edge node[auto] {$ $} (m-1-3)
    (m-1-3) edge node[auto] {} 			  (m-1-4)
    
    (m-2-1) edge node[auto] {$ B $} (m-2-2)
    (m-2-2) edge node[auto] {$ $} (m-2-3)
    (m-2-3) edge node[auto] {} 			  (m-2-4)

    (m-1-1) edge node[auto] {$ g^1 $} (m-2-1)
    (m-1-2) edge node[auto] {$ g^0 $} (m-2-2)
    (m-1-3) edge node[description] {$ g $} 	  (m-2-3);

\end{tikzpicture}
\end{center}
Für alle $ k\in \mathbb{N} $ gilt nun
\[ \tilde{T} \big (  X(p)^k\otimes_{E(p)}E(p)\big )=\tilde{T}\big ( X(p)^k\big ) =E(p)^k. \]
Die Matrizen $ A,B,g^0 $ und $ g^1 $ entsprechen eintragsweise Endomorphismen $ \varphi $ von $ E(p) $ und für diese gilt, wie wir geprüft haben
\[ \tilde{T} \big ( X(p)\otimes_{E(p)}\_\big )(\varphi) = \varphi. \]
Also gilt wegen Additivität der Funktoren auch
\[ \tilde{T} \big (  X(p)\otimes_{E(p)}\_\big )(A)=A, \] 
sowie
\[ \tilde{T} \big (  X(p)\otimes_{E(p)}\_\big )(B)=B. \] 
\[ \tilde{T} \big (  X(p)\otimes_{E(p)}\_\big )(g^0)=g^0. \] 
\[ \tilde{T} \big (  X(p)\otimes_{E(p)}\_\big )(g^1)=g^1. \] 

Wenden wir $ \tilde{T}\circ \big ( X(p)\otimes_{E(p)}\_\big ) $ auf das Diagramm an, so erhalten wir also
\begin{center}
\begin{tikzpicture}[description/.style={fill=white,inner sep=2pt}]
    \matrix (m) [matrix of math nodes, row sep=4.5em,
    column sep=3.5em, text height=1.5ex, text depth=0.25ex]
    { E(p)^m & E(p)^n& \tilde{T} \big ( X(p)\otimes_{E(p)}\_\big )(M) & 0 \\
      E(p)^a & E(p)^b &\tilde{T}(\big ( X(p)\otimes_{E(p)}\_)\big )(N) & 0 \\
	};
    \path[->,font=\scriptsize]
    (m-1-1) edge node[auto] {$  A $} (m-1-2)
    (m-1-2) edge node[auto] {$ $} (m-1-3)
    (m-1-3) edge node[auto] {} 			  (m-1-4)
    
    (m-2-1) edge node[auto] {$ B $} (m-2-2)
    (m-2-2) edge node[auto] {$ $} (m-2-3)
    (m-2-3) edge node[auto] {} 			  (m-2-4)

    (m-1-1) edge node[auto] {$ g^1 $} (m-2-1)
    (m-1-2) edge node[auto] {$ g^0 $} (m-2-2)
    (m-1-3) edge node[description] {$ \tilde{T} \big ( X(p)\otimes_{E(p)}\_\big )(g)$} 	  (m-2-3);

\end{tikzpicture}
\end{center}
Es gilt also 
\[ \tilde{T} \big ( X(p)\otimes_{E(p)}\_\big )(M)\cong coker(A) \cong M \]
und
\[ \tilde{T} \big ( X(p)\otimes_{E(p)}\_\big )(N) \cong coker(B) \cong N. \]
Nach universeller Eigenschaft des Kokerns gibt es wieder eine eindeutige Abbildung zwischen $ M $ und $ N $, welche das Diagramm kommutativ macht, also entsprechen sich also auch $ \tilde{T} \big ( X(p)\otimes_{E(p)}\_\big )(g) $ und $ g $. Der Funktor $ \tilde{T}\circ \big ( X(p)\otimes_{E(p)}\_\big ) $ induziert also bis auf Isomorphie den Identitätsfunktor auf $ E(p)\Mod $. 
\end{proof}

\begin{Satz}\label{wesentlich_surjektiv}
Sei $ R $ ein kommutativer, noetherscher Ring, $ \mathcal{A} $ eine abelsche, $ R $-lineare Kategorie und 
\[ T:\mathcal{A}\longrightarrow R\Mod \]
ein treuer, exakter, $ R $-linearer Funktor, also insbesondere eine Darstellung. Sei $ \C(T_\mathcal{A}) $ die Diagrammkategorie der Darstellung $ T $ und 
\[  \mathcal{A}\xrightarrow{\tilde{T}}\mathcal{C}(T_{\mathcal{A}})\xrightarrow{ff_T}R\Mod   \]
die zugehörige Faktorisierung von $ T $. Dann ist $ \tilde{T} $ wesentlich surjektiv. 
\end{Satz}

\begin{proof}
Wir konstruieren einen Funktor
\[ S:\mathcal{C}(T_{\mathcal{A}})\longrightarrow \mathcal{A}, \]
so dass die Verkettung
\[\mathcal{C}(T_{\mathcal{A}})\overset{S}{\longrightarrow} \mathcal{A}\overset{\tilde{T}}{\longrightarrow }{\C(T_{\mathcal{A}})} \]
natürlich äquivalent zum Identitätsfunktor auf $ \C(T_{\mathcal{A}}) $ ist.
Erinnern wir uns an die Definition von $ \C(T_{\mathcal{A}}) $:
\[  \C(T_{\mathcal{A}}):=\varinjlim_{F\subset \mathcal{A}\text{ endlich}}\End(T_{|F})\Mod \]
Sei $ F $ eine endliche Teilmenge von $ Ob(\mathcal{A}) $. Sei $ p=\bigoplus_{p_i\in F}p_i $. Wir überlegen uns, dass die Kategorie $ \End(T_{|F})\Mod $ natürlich isomorph zur Kategorie $ \End(T_{\{p\}})\Mod $ ist:\\
Ein Modulendomorphismus
\[ f:\bigoplus_{i=1}^n p_i\longrightarrow\bigoplus_{i=1}^n p_i \]
ist gegeben durch eine Familie von Morphismen $ (f_{k,j})_{k,j\in\{1..n\}} $, wobei wir $ f_{k,j} $ durch die Verkettung
\[p_i\stackrel{i_i}{\longrightarrow}\bigoplus_{k=1}^n p_k\stackrel{f}{\longrightarrow}\bigoplus_{k=1}^n p_k\stackrel{\pi_j}{\longrightarrow}p_j \]
erhalten. Ein Element aus $ \End(T_{\{p\}}) $ ist demnach ein $ R $-linearer Modulhomomorphismus
\[ \bigoplus_{i=1}^n Tp_i\stackrel{(e_{kj})_{kj}}{\longrightarrow}\bigoplus_{i=1}^n Tp_i,\]
so dass für alle $ (a_{kj})_{kj}\in\End(\bigoplus_{i=1}^n p_i) $ die Diagramme

\begin{center}
\begin{tikzpicture}[description/.style={fill=white,inner sep=2pt}]
    \matrix (m) [matrix of math nodes, row sep=3em,
    column sep=2.5em, text height=1.5ex, text depth=0.25ex]
    {  \bigoplus_{i=1}^n Tp_i & & \bigoplus_{i=1}^n Tp_i \\
       \bigoplus_{i=1}^n Tp_i & & \bigoplus_{i=1}^n Tp_i \\
	};
    \path[->,font=\scriptsize]
    (m-1-1) edge node[auto] {$ (e_{kj})_{kj} $} (m-2-1)
    (m-1-3) edge node[auto] {$ (e_{kj})_{kj} $} (m-2-3)
    (m-1-1) edge node[auto] {$ (Ta_{kj})_{kj} $} (m-1-3)
    (m-2-1) edge node[auto] {$ (Ta_{kj})_{kj} $} (m-2-3);
\end{tikzpicture}
\end{center}

kommutieren. Hieraus folgt direkt, dass $ e_{kj}=0_{Abb} $ für alle $ k\neq j $, da sonst das Diagramm für
\[ \operatorname{a_{l,m}}=\begin{cases} id_{p_k}\  & k=l=m\ ,\\ 0_{Abb}\ & sonst \end{cases}\]
nicht kommutieren würde. Ein Element aus $ \End(T_{\{p\}}) $ ist also gegeben durch eine Familie $ (e_{ii})\in\prod_{p_i\in F}\End_R(Tp_i) $, so dass alle Diagramm der Form

\begin{center}
\begin{tikzpicture}[description/.style={fill=white,inner sep=2pt}]
    \matrix (m) [matrix of math nodes, row sep=3em,
    column sep=2.5em, text height=1.5ex, text depth=0.25ex]
    {   Tp_i & &  Tp_j \\
		Tp_i & &  Tp_j \\
	};
    \path[->,font=\scriptsize]
    (m-1-1) edge node[auto] {$ e_{ii} $} (m-2-1)
    (m-1-3) edge node[auto] {$ e_{jj} $} (m-2-3)
    (m-1-1) edge node[auto] {$ Ta_{ij} $} (m-1-3)
    (m-2-1) edge node[auto] {$ Ta_{ij} $} (m-2-3);
\end{tikzpicture}
\end{center}

kommutieren. Dies ist aber gerade die Definition eines Elements aus $ \End(T_{|F}) $. Also ist auch die Kategorie der $ \End(T_{|F}) $-Moduln natürlich isomorph zur Kategorie der $ \End(T_{\{p\}}) $-Moduln. \\
Da offensichtlich auch $ \langle p \rangle=\langle F \rangle $ gilt, induziert die Faktorisierung des Identitätsfunktors aus Korollar \ref{id}
\[
\begin{array}{ccccc}
\End(T_{|\{p\}})\Mod & \xrightarrow{X(p)\otimes_{\End(T_{|\{p\}})}\_}  & \langle p \rangle & \overset{\tilde{T}_{|\langle p \rangle}}{\longrightarrow} & \End(T_{|\{p\}})\Mod \\
\end{array}
\]
eine Faktorisierung des Identitätsfunktors 
\[ \End(T_F)\Mod \xrightarrow{X(p)\otimes_{\End(T_{|\{p\}})} \_} \langle F \rangle  \overset{\tilde{T}_{|F}} {\longrightarrow} \End(T_F)\Mod \]
auf $ \End(T_F)\Mod $.
%
Da $ \mathcal{A}=\varinjlim_{F\subset \mathcal{A}}\langle F \rangle$ für das gerichtete System der endlichen Teilmengen $ F\subset D $ ist, erhalten wir nach Übergang zum direkten Limes durch dessen universelle Eigenschaft eindeutige Funktoren
\[\mathcal{C}(T_{\mathcal{A}})\overset{S}{\longrightarrow} \mathcal{A}\overset{\tilde{T}}{\longrightarrow }{\C(T_{\mathcal{A}})}, \]
welche nach Restriktion die Verkettung
 \[ \End(T_F)\Mod  \overset{X(p)\otimes_{\End(T_{|\{p\}})}\_} {\longrightarrow}  \langle F \rangle  \overset{\tilde{T}_{|F}} {\longrightarrow}  \End(T_F)\Mod \]
liefern. Da jedes Objekt aus $ \mathcal{C}(T_{\mathcal{A}}) $ ein $ \End(T_F)$-Modul für geeignetes endliches $ F $ ist, entspricht auch
\[\mathcal{C}(T_{\mathcal{A}})\overset{S}{\longrightarrow} \mathcal{A}\overset{\tilde{T}}{\longrightarrow }{\C(T_{\mathcal{A}})}, \]
bis auf Isomorphie dem Identitätsfunktor auf $\mathcal{C}(T_{\mathcal{A}})  $. Für ein beliebiges Objekt \\$ M\in \C(T_{\mathcal{A}}) $ ist $ S(M) $ also ein Urbild von $ M $ unter $ \tilde{T} $. Damit ist $ \tilde{T} $ wesentlich surjektiv.
\end{proof}

\section{Die Adjunktion der Funktoren $ \Hom_R(M,\_) $ und $ \_\otimes_R M $ in $ \mathcal{A} $}

Wir müssen nun noch zeigen, dass der Funktor $ \tilde{T} $ voll ist. Im letzten Beweis haben wir für alle $ p\in\mathcal{A} $ einen Funktor $ \C(T)\longrightarrow\mathcal{A} $ gefunden, so dass die Verkettung
\[
\begin{array}{ccccc}
\C(T) & \overset{S}{\longrightarrow} & \mathcal{A} & \overset{\tilde{T}}{\longrightarrow} & \C(T)\\
Tp & \mapsto & X(p)\otimes_{E(p)}Tp & \mapsto & Tp
\end{array}
\]
für $ Tp $ die Identität in $ \C(T) $ ist. Also ist der Funktor $ \tilde{T} $ zumindest auf dem Objekt $ X(p)\otimes_{E(p)}Tp $ voll. Unsere Strategie wird sein, einen natürlichen Isomorphismus
\[  X(p)\otimes_{E(p)}Tp \cong p \]
zu konstruieren. Dieser wird von einer natürlichen Abbildung 
\[ \Hom_R(Tp,p)\otimes_R Tp \longrightarrow p \]
induziert. Ein wesentlicher Schritt ist es, diese natürliche Abbildung zu finden. In $ R $-Mod haben wir in ähnlichem Fall die natürliche Evaluationsabbildung
\[ \Hom_R(Tp,Tp)\otimes_R Tp \longrightarrow Tp. \]
Eine ähnliche Konstruktion werden wir in $ \mathcal{A} $ durchführen.

\begin{Def}
Seien $ \mathcal{C}, \mathcal{D} $ Kategorien und
\[ \mathcal{F}:\C\rightarrow\mathcal{D} \hspace{2 cm}
\mathcal{G}:\mathcal{D}\rightarrow\mathcal{C} \]
Funktoren. Wir sagen $ \mathcal{F} $ ist \emph{linksadjungiert} zu $ \mathcal{G} $ bzw. $ \mathcal{G} $ ist \emph{rechtsadjungiert} zu $ \mathcal{F} $, wenn für alle $ X\in\C, Y\in\mathcal{D} $ eine natürliche Bijektion 
\[ \phi_{X,Y}:\Hom_{\C}(X,\mathcal{G}(Y))\cong\Hom_{\mathcal{D}}(\mathcal{F}(X),Y) \]
existiert.
Die Bijektion heißt natürlich, wenn für alle $ f\in\Hom_{\C}(X',X) $, $ g\in\Hom_{\mathcal{D}}(Y,Y') $ das Diagramm

\begin{center}
\begin{tikzpicture}[description/.style={fill=white,inner sep=2pt}]
    \matrix (m) [matrix of math nodes, row sep=3em,
    column sep=3.8em, text height=1.9ex, text depth=0.25ex]
    { \Hom_{\C}(X,\mathcal{G}(Y)) & \Hom_{\mathcal{D}}(\mathcal{F}(X),Y)  \\
      \Hom_{\C}(X',\mathcal{G}(Y'))  & \Hom_{\mathcal{D}}(\mathcal{F}(X'),Y') \\
	};
    \path[->,font=\scriptsize]
    (m-1-1) edge node[auto] {$ \mathcal{G}(g)\circ \_ \circ f $} (m-2-1)
    (m-1-2) edge node[auto] {$ g\circ \_ \circ \mathcal{F}(f) $} (m-2-2)
    (m-1-1) edge node[auto] {$ \phi_{X,Y} $}  (m-1-2)
    (m-2-1) edge node[auto] {$ \phi_{X',Y} $}  (m-2-2)

;

\end{tikzpicture}
\end{center}
kommutiert.
\end{Def}

\begin{Bsp}
Sei $ R $ ein kommutativer Ring, $ M $ ein beliebiger $ R $-Modul. Dann ist der Funktor
\[
\begin{array}{cccc}
\_\otimes_R M: & R\Mod & \rightarrow & R\Mod \\
		&  N & \mapsto & N\otimes_R M
\end{array}
\]
linksadjungiert zu
\[
\begin{array}{cccc}
\Hom_R(M,\_): & R\Mod & \rightarrow & R\Mod \\
		   &  N & \mapsto & \Hom_R(M,N)
\end{array}
\]
\end{Bsp}

\begin{Bem}
Ist $ \mathcal{F} $ linksadjungiert zu $ \mathcal{G} $, so impliziert dies, dass es natürliche Transformationen \\
\[ \epsilon:\mathcal{G}\mathcal{F}\rightarrow 1_{\mathcal{C}} \]
\[ \eta: 1_{\mathcal{D}} \rightarrow\mathcal{F}\mathcal{G} \]
gibt \cite[S.84 ff]{MR0354799}. Im obigen Beispiel ist die Transformation $ \epsilon $ durch den Auswertungsmorphismus
\[
\begin{array}{cccc}
ev: &  \Hom_R(M,N)\otimes_R M & \rightarrow & N \\
		   &  \varphi\otimes m & \mapsto & \varphi(m)
\end{array}
\]
gegeben. 
\end{Bem}

\begin{Lemma}\label{adjunktion}
Sei $ \mathcal{A} $ eine $ R $-lineare, abelsche Kategorie, $ M $ ein endlich erzeugter $ R $-Modul. Seien aus $ \Hom_R(M,\_)  $ und $ \_\otimes_R M $ die Funktoren auf $ \mathcal{A} $ aus Konstruktion \ref{konstr}.
Dann ist der Funktor $ \Hom_R(M,\_)  $ rechtsadjungiert zu $ \_\otimes_R M $.
\end{Lemma}

\begin{proof}
Sei zunächst $ M=R^n $. Seien $ f\in\Hom_{\mathcal{A}}(a,b) $, $ g\in\Hom_{\mathcal{A}}(p,q) $ Morphismen zwischen Objekten aus $ \mathcal{A} $. Wir haben einen kanonischen Isomorphismus von $ R $-Moduln
\[ \phi^n_{p,b}:\Hom_{\mathcal{A}}(p,\bigoplus_{i=1}^n b)\cong\Hom_{\mathcal{A}}(\bigoplus_{i=1}^n p,b), \] 
also nach Definition der beiden Funktoren einen Isomorphismus
\[ \Hom_{\mathcal{A}}(p,\Hom_R(R^n,b))=\Hom_{\mathcal{A}} (p,\bigoplus_{i=1}^n b) \cong \Hom_{\mathcal{A}}(\bigoplus_{i=1}^n p,b)=\Hom_{\mathcal{A}}(p \otimes_R R^n,b). \] 

$ (\_\otimes_R R^n)(f)$ liefert die Abbildung 
\[ \bigoplus_{i=1}^n p\xrightarrow{(f,..f)}\bigoplus_{i=1}^n q, \]
welche auf jeder Komponente $ f $ ausführt. Ebenso liefert $ \Hom_R(R^n,\_ )(g) $ die Abbildung 
\[ \bigoplus_{i=1}^n a\xrightarrow{(g,..g)}\bigoplus_{i=1}^n b. \]
Das Diagramm

\begin{center}
\begin{tikzpicture}[description/.style={fill=white,inner sep=2pt}]
    \matrix (m) [matrix of math nodes, row sep=3em,
    column sep=3.8em, text height=1.9ex, text depth=0.25ex]
    { \Hom_{\mathcal{A}}(p,\bigoplus_{i=1}^n b) & \Hom_{\mathcal{A}}(\bigoplus_{i=1}^n p,b)  \\
	  \Hom_{\mathcal{A}}(q,\bigoplus_{i=1}^n a) & \Hom_{\mathcal{A}}(\bigoplus_{i=1}^n q,a)  \\
	};
    \path[->,font=\scriptsize]
    (m-1-1) edge node[auto] {$ (g,..,g)\circ \_ \circ f $} (m-2-1)
    (m-1-2) edge node[auto] {$ g\circ \_ \circ (f,..f) $} (m-2-2)
    (m-1-1) edge node[auto] {$ \phi^n_{p,b} $}  (m-1-2)
    (m-2-1) edge node[auto] {$ \phi^n_{p,b} $}  (m-2-2)
;

\end{tikzpicture}
\end{center}

kommutiert und demzufolge auch das Diagramm

\begin{center}
\begin{tikzpicture}[description/.style={fill=white,inner sep=2pt}]
    \matrix (m) [matrix of math nodes, row sep=3em,
    column sep=3.8em, text height=1.9ex, text depth=0.25ex]
    {& \Hom_{\mathcal{A}}(p,\Hom_R(R^n,b)) & \Hom_{\mathcal{A}}(p\otimes_R R^n,b) & \\
	 & \Hom_{\mathcal{A}}(q,\Hom_R(R^n,a)) & \Hom_{\mathcal{A}}( q\otimes_R R^n,a) & (1) \\
	};
    \path[->,font=\scriptsize]
    (m-1-2) edge node[auto] {$ \Hom_R(R^n,\_)(g)\circ \_ \circ f $} (m-2-2)
    (m-1-3) edge node[auto] {$ g\circ \_ \circ (\_\otimes_R R^n)(f) $} (m-2-3)
    (m-1-2) edge node[auto] {$ \phi^n_{p,b} $}  (m-1-3)
    (m-2-2) edge node[auto] {$ \phi^n_{p,b} $}  (m-2-3)
;

\end{tikzpicture}
\end{center}

Sei nun $ M $ ein beliebiger endlich erzeugter $ R $-Modul. Sei 
\[ R^m\xrightarrow{A}R^n\twoheadrightarrow M\rightarrow 0 \]
eine Präsentation von $ M $ und $ p\in\mathcal{A} $ ein Objekt. Die Sequenz
\[ p^m\xrightarrow{\tilde{A}}p^n\twoheadrightarrow p \otimes_R M\rightarrow 0, \]
welche $ p\otimes_R M $ definiert, ist exakt. Da der Funktor $ \Hom_{\mathcal{A}}(\_,b) $ für $ b\in\mathcal{A} $ kontravariant und linksexakt ist, erhalten wir eine exakte Sequenz
\[ \Hom_{\mathcal{A}}(p^m,b)\xleftarrow{\circ\tilde{A}}\Hom_{\mathcal{A}}(p^n,b)\hookleftarrow \Hom_{\mathcal{A}}(p\otimes_R M,b)\leftarrow 0. \]
Analog erhalten wir durch Anwenden des kovarianten, linksexakten Funktors $ \Hom_{\mathcal{A}}(p,\_) $ auf die exakte Sequenz
\[ b^m\xleftarrow{\tilde{A}^T} b^n\hookleftarrow \Hom_R(M,b)\leftarrow 0 \]
eine exakte Sequenz
\[ \Hom_{\mathcal{A}}(p,b^m)\xleftarrow{\tilde{A}^T\circ}\Hom_{\mathcal{A}}(p,b^n)\hookleftarrow \Hom_{\mathcal{A}}(p,\Hom_R(M,b))\leftarrow 0. \]
Das Diagramm 
\begin{center}
\begin{tikzpicture}[description/.style={fill=white,inner sep=2pt}]
    \matrix (m) [matrix of math nodes, row sep=3em,
    column sep=3.8em, text height=1.9ex, text depth=0.25ex]
    { \Hom_{\mathcal{A}}(p,b^n) & \Hom_{\mathcal{A}}(p^n,b)  \\
	  \Hom_{\mathcal{A}}(p,b^m) & \Hom_{\mathcal{A}}(p^m,b)  \\
 	  \Hom_{\mathcal{A}}(p,\Hom_R(M,b)) & \Hom_{\mathcal{A}}(p\otimes_R M,b)  \\
 	  0 & 0 \\
	};
    \path[->,font=\scriptsize]
    (m-2-1) edge node[auto] {$ \tilde{A}^T\circ $} (m-1-1)
    (m-2-2) edge node[auto] {$ \circ \tilde{A} $} (m-1-2)
    (m-1-1) edge node[auto] {$ \phi^n_{p,b} $}  (m-1-2)
    (m-2-1) edge node[auto] {$ \phi^m_{p,b} $}  (m-2-2)
    (m-3-1) edge node[auto] {$ ker(\tilde{A}^T\circ) $}  (m-2-1)
    (m-3-2) edge node[auto] {$ ker(\circ\tilde{A}) $}  (m-2-2)
    (m-4-1) edge node[auto] {$  $}  (m-3-1)
    (m-4-2) edge node[auto] {$  $}  (m-3-2)
;
\end{tikzpicture}
\end{center}

kommutiert. Dies kann man direkt nachrechnen. Alternativ überlegt man sich, dass die Abbildung
\[
\begin{array}{cccc}
\phi_{p,b}^n: &  \Hom_{\mathcal{A}}(p,b^n) & \rightarrow & \Hom_{\mathcal{A}}(p^n,b)\\
 	& (f_1,..,f_n) & \mapsto &  \begin{pmatrix} f_1 \\..\\ f_n \end{pmatrix} 
\end{array}
\]
dem Transponieren von Matrizen mit Einträgen in $ \Hom_{\mathcal{A}}(p,b) $ entspricht. \\Für $ B\in Mat_{\Hom_{\mathcal{A}}(p,b)}(n\times 1) $ ist also 
\[\phi_{p,b}^n(\tilde{A}^T\circ B)=(\tilde{A}^T\circ B)^T=B^T\circ \tilde{A}= \phi_{p,b}^n(B)\circ\tilde{A}.\]
Dass wir hier rechnen können wie mit Matrizen, liegt daran, dass alle Einträge in $ B $ $ R $-lineare Abbildungen sind und $ \tilde{A} $ sowie $ \tilde{A}^T $ nur Einträge in $ R $ haben (vgl. Argumentation in Lemma \ref{vertauschen}).

Dies definiert nach universeller Eigenschaft des Kerns (vgl. Beweis von Lemma \ref{Hom}) eine eindeutige Abbildung
 \[  \Hom_{\mathcal{A}}(p,\Hom_R(M,b))\xrightarrow{\tilde{\phi}_{p,b}} \Hom_{\mathcal{A}}(p\otimes_R M,b), \]
welche das Diagramm

\begin{center}
\begin{tikzpicture}[description/.style={fill=white,inner sep=2pt}]
    \matrix (m) [matrix of math nodes, row sep=3em,
    column sep=3.8em, text height=1.9ex, text depth=0.25ex]
    { \Hom_{\mathcal{A}}(p,b^n) & \Hom_{\mathcal{A}}(p^n,b)  \\
	  \Hom_{\mathcal{A}}(p,b^m) & \Hom_{\mathcal{A}}(p^m,b)  \\
 	  \Hom_{\mathcal{A}}(p,\Hom_R(M,b)) & \Hom_{\mathcal{A}}(p\otimes_R M,b)  \\
   	  0 & 0 \\
	};
    \path[->,font=\scriptsize]
    (m-2-1) edge node[auto] {$ \tilde{A}^T\circ $} (m-1-1)
    (m-2-2) edge node[auto] {$ \circ \tilde{A} $} (m-1-2)
    (m-1-1) edge node[auto] {$ \phi^n_{p,b} $}  (m-1-2)
    (m-2-1) edge node[auto] {$ \phi^m_{p,b} $}  (m-2-2)
    (m-3-1) edge node[auto] {$ \tilde{\phi}_{p,b} $}  (m-3-2)
    (m-3-1) edge node[auto] {$ ker(\tilde{A}^T\circ) $}  (m-2-1)
    (m-3-2) edge node[auto] {$ ker(\circ\tilde{A}) $}  (m-2-2)
    (m-4-1) edge node[auto] {$  $}  (m-3-1)
    (m-4-2) edge node[auto] {$  $}  (m-3-2)
;
 
\end{tikzpicture}
\end{center}

kommutativ macht. Wir können durch die Inklusionen 
\[ \Hom_{\mathcal{A}}(p,\Hom_R(M,b))\xrightarrow{ker(\tilde{A}^T\circ)}\Hom_{\mathcal{A}}(p,b^m) \]
und
\[ \Hom_{\mathcal{A}}(p\otimes_R M,b)\xrightarrow{ker(\circ\tilde{A})}\Hom_{\mathcal{A}}(p^m,b) \]
die Mengen $ \Hom_{\mathcal{A}}(p,\Hom_R(M,b)) $ bzw. $ \Hom_{\mathcal{A}}(p\otimes_R M,b) $ als Teilmengen von $  \Hom_{\mathcal{A}}(p,b^m) $ bzw. $ \Hom_{\mathcal{A}}(p^m,b) $ auffassen. Der Morphismus $ \tilde{\phi}_{p,b} $ wird dann zum Einschränkungshomomorphimus von $ \phi^m_{p,b} $. Für diese Teilmengen gilt dann offensichtlich, wie in Diagramm (1), dass

\begin{center}
\begin{tikzpicture}[description/.style={fill=white,inner sep=2pt}]
    \matrix (m) [matrix of math nodes, row sep=3em,
    column sep=3.8em, text height=1.9ex, text depth=0.25ex]
    { \Hom_{\mathcal{A}}(p,\Hom_R(M,b)) & \Hom_{\mathcal{A}}(p\otimes_R M,b)\\
	  \Hom_{\mathcal{A}}(p,\Hom_R(M,a)) & \Hom_{\mathcal{A}}(p\otimes_R M,a)\\
	};
    \path[->,font=\scriptsize]
    (m-1-1) edge node[auto] {$ ((\Hom_R(M,\_)(g))\circ \_ \circ f $} (m-2-1)
    (m-1-2) edge node[auto] {$ g\circ \_ \circ (\_\otimes_R M)(f)) $} (m-2-2)
    (m-1-1) edge node[auto] {$ \tilde{\phi}_{p,b} $}  (m-1-2)
    (m-2-1) edge node[auto] {$ \tilde{\phi}_{p,b} $}  (m-2-2)
;
\end{tikzpicture}
\end{center}

kommutiert, was zu zeigen war.

\end{proof}

\begin{Bem}\label{Korrespondenz}
Wir haben diese Adjunktion analog zu der Adjunktion in $ R $-Mod konstruiert. Wenden wir den Funktor $ T $ auf das letzte Diagramm an, erhalten wir daher die gewöhnliche Adjunktion von $ \Hom_R(M,\_) $ und $ \_\otimes_R M  $ in $ R $-Mod.
\end{Bem}

\section{Der Funktor $\tilde{T}$ ist voll}

\begin{Satz}\label{voll}
Sei $ \mathcal{A} $ eine abelsche, $ R $-lineare Kategorie, also insbesondere ein Diagramm. Sei 
\[ T:\mathcal{A}\rightarrow R\Mod \]
ein treuer, exakter, $ R $-linearer Funktor. Sei 
\[  \mathcal{A}\xrightarrow{\tilde{T}}\mathcal{C}(T_{\mathcal{A}})\xrightarrow{ff_T}R\Mod   \]
die Faktorisierung über die Diagrammkategorie von $ T $. Dann ist der Funktor $ \tilde{T} $ voll. 
\end{Satz}
\begin{proof}
Die Idee des Beweises stammt aus \cite[S.130 ff]{MR654325}. Wir konstruieren zunächst für alle $ p\in\mathcal{A} $ einen Isomorphismus zu $ X(p)\otimes_{E(p)}Tp $. \\
Die Adjunktion zwischen Tensorprodukt und Hom aus Lemma \ref{adjunktion} liefert uns eine natürliche Bijektion 
\[ \Hom_{\mathcal{A}}(\Hom_R(Tp,p)\otimes_R Tp,p)\cong\Hom_{\mathcal{A}}(\Hom_R(Tp,p),\Hom_R(Tp,p)), \]
welche unter $ T $ gemäß Bemerkung \ref{Korrespondenz} zu der natürlichen Bijektion 
\[ \Hom_{\mathcal{A}}(\Hom_R(Tp,Tp)\otimes_R Tp,Tp)\cong\Hom_{\mathcal{A}}(\Hom_R(Tp,Tp),\Hom_R(Tp,Tp)), \]
in $ R $-Mod wird. 
Korrespondierend zur Identität in $ \Hom_R(\Hom_R(Tp,Tp),\Hom_R(Tp,Tp)) $ haben wir den Auswertungshomomorphismus 
\[
\begin{array}{cccc}
ev: &  \Hom_R(Tp,Tp)\otimes_R Tp & \rightarrow & Tp\\
 	& (\varphi,x) & \mapsto & \varphi(x).
\end{array}
\]
Wir bezeichnen den zur Identität in $ \Hom_{\mathcal{A}}(\Hom_R(Tp,p),\Hom_R(Tp,p)) $ korrespondierenden Morphismus in $  \Hom_{\mathcal{A}}(\Hom_R(Tp,p)\otimes_R Tp,p) $ mit $ \tilde{ev} $. Nach Bemerkung \ref{Korrespondenz} gilt $ T(\tilde{ev})=ev $.
Das Diagramm 

\begin{center}
\begin{tikzpicture}[description/.style={fill=white,inner sep=2pt}]
    \matrix (m) [matrix of math nodes, row sep=3.5em,
    column sep=4.0em, text height=1.5ex, text depth=0.25ex]
    {& \Hom_R(Tp,p)\otimes_R Tp & \\
    	X(p)\otimes_R Tp & & p   \\
	};
    \path[->,font=\scriptsize]
    (m-2-1) edge node[auto] {$ incl. $} (m-1-2)
    (m-1-2) edge node[auto] {$ \tilde{ev} $} (m-2-3)
    (m-2-1) edge node[auto] {$ \tilde{ev}_{|X(p)} $} (m-2-3)
    ;   
\end{tikzpicture}
\end{center}

wird von $ T $ auf folgendes Diagramm in der Kategorie $ R $-Mod abgebildet:

\begin{center}
\begin{tikzpicture}[description/.style={fill=white,inner sep=2pt}]
    \matrix (m) [matrix of math nodes, row sep=3.5em,
    column sep=4.0em, text height=1.5ex, text depth=0.25ex]
    {& \Hom_R(Tp,Tp)\otimes_R Tp& \\
    	E(p)\otimes_R Tp & & Tp   \\
	};
    \path[->,font=\scriptsize]
    (m-2-1) edge node[auto] {$ incl. $} (m-1-2)
    (m-1-2) edge node[auto] {$ ev $} (m-2-3)
    (m-2-1) edge node[auto] {$ ev_{|E(p)} $} (m-2-3)
    ;   
\end{tikzpicture}
\end{center}

Betrachten wir nun in $ \mathcal{A} $ die kanonische Abbildung
\[ X(p)\otimes_R Tp \longrightarrow  X(p)\otimes_{E(p)} Tp. \]
Sie entspricht der Projektionsabbildung
\[ X(p) \otimes_R Tp  \longrightarrow  \dfrac{X(p)\otimes_R Tp}{\sim},\]
wobei die Äquivalenzrelation erzeugt wird von $ \varphi\cdot e\otimes_R x \sim \varphi \otimes_R e\cdot x $ für $ x\in Tp,$ \\ $\varphi \in X(p) $ und $ e\in E(p)\subset\End_R(Tp) $. Das Objekt $ Tp\otimes_{E(p)}X(p) $ ist also der Kokern von 
\[ 
\begin{array}{cccc}
X(p)\otimes_R E(p)\otimes_R Tp & \xrightarrow{id\otimes_R ev- \rho_R\otimes_R id} & X(p)\otimes_R Tp, \\
\end{array}
\]
wobei $ \rho_R $ die Rechtswirkung von $ E(p) $ auf $ X(p) $ beschreibt. 
Unter $ T $ bleiben Kokerne erhalten, daher wird aus dem Diagramm in $ \mathcal{A} $

\begin{center}
\begin{tikzpicture}[description/.style={fill=white,inner sep=2pt}]
    \matrix (m) [matrix of math nodes, row sep=3.5em,
    column sep=4.0em, text height=1.5ex, text depth=0.25ex]
    { X(p)\otimes_R E(p)\otimes_R Tp & & \Hom_R(Tp,p)\otimes_R Tp & \\
    	& X(p)\otimes_R Tp & & p  \\
     & & X(p)\otimes_{E(p)} Tp & \\
	};
    \path[->,font=\scriptsize]
    (m-1-1) edge node[description] {$ id\otimes_R \tilde{ev}- \rho_R\otimes_R id $} (m-2-2)
    (m-2-2) edge node[description] {$ incl. $} (m-1-3)
    (m-1-3) edge node[description] {$ \tilde{ev} $} (m-2-4)
    (m-2-2) edge node[description] {$ \tilde{ev}_{|X(p)} $} (m-2-4)
    (m-2-2) edge node[description] {$ coker(id \otimes_R \tilde{ev} - \rho_R \otimes_R id) $} (m-3-3)
    ;   
\end{tikzpicture}
\end{center}

durch Anwenden von $ T $ folgendes Diagramm in $ R\Mod $:

\begin{center}
\begin{tikzpicture}[description/.style={fill=white,inner sep=2pt}]
    \matrix (m) [matrix of math nodes, row sep=3.5em,
    column sep=4.0em, text height=1.5ex, text depth=0.25ex]
    { E(p)\otimes_R E(p)\otimes_R Tp & & \Hom_R(Tp,Tp)\otimes_R Tp & \\
    	& E(p)\otimes_R Tp & & Tp   \\
     & & E(p)\otimes_{E(p)} Tp & \\
	};
    \path[->,font=\scriptsize]
    (m-1-1) edge node[description] {$ id \otimes_R ev- (\_\circ\_)\otimes_R id $} (m-2-2)
    (m-2-2) edge node[description] {$ incl. $} (m-1-3)
    (m-1-3) edge node[description] {$ ev $} (m-2-4)
    (m-2-2) edge node[description] {$ ev_{|E(p)} $} (m-2-4)
    (m-2-2) edge node[description] {$ coker(id \otimes_R ev- (\_\circ\_)\otimes_R id) $} (m-3-3)
    (m-3-3) edge node[description] {$ \simeq $} (m-2-4)
    ;   
\end{tikzpicture}
\end{center}

Da für $ e,e'\in E(p), x\in Tp $ gilt, dass $ (e\circ e')(x)=e(e'(x)) $ ist, ist die Verkettung

\[E(p)\otimes_R E(p)\otimes_R Tp \xrightarrow{id\otimes_R ev- (\_\circ\_) \otimes_R id} E(p)\otimes_R Tp\overset{ev_{|E(p)}}{\longrightarrow} Tp \]

die Nullabbildung. 
Da $ T $ treu ist, ist dann auch die Verkettung

\[X(p)\otimes_R E(p)\otimes_R Tp\xrightarrow{id\otimes_R \tilde{ev} - \rho_R\otimes_R id} X(p)\otimes_R Tp \overset{\tilde{ev}_{|X(p)}}{\longrightarrow} p \]
die Nullabbildung. Nun gibt es nach universeller Eigenschaft des Kokerns einen eindeutigen Morphismus 
\[X(p)\otimes_{E(p)}Tp\longrightarrow p,\] 
welcher das Diagramm

\begin{center}
\begin{tikzpicture}[description/.style={fill=white,inner sep=2pt}]
    \matrix (m) [matrix of math nodes, row sep=3.5em,
    column sep=4.0em, text height=1.5ex, text depth=0.25ex]
    { X(p)\otimes_R E(p)\otimes_R Tp & & \Hom_R(Tp,p) \otimes_R Tp & \\
    	& X(p)\otimes_R Tp & & p   \\
     & & X(p)\otimes_{E(p)} Tp & \\
	};
    \path[->,font=\scriptsize]
    (m-1-1) edge node[description] {$ id\otimes_R \tilde{ev}- \rho_R \otimes_R id $} (m-2-2)
    (m-2-2) edge node[description] {$ incl. $} (m-1-3)
    (m-1-3) edge node[description] {$ \tilde{ev} $} (m-2-4)
    (m-2-2) edge node[description] {$ \tilde{ev}_{|X(p)} $} (m-2-4)
    (m-2-2) edge node[description] {$ coker(id \otimes_R \tilde{ev} -\rho_R \otimes_R id) $} (m-3-3)
    (m-3-3) edge node[description] {$ \exists ! $} (m-2-4)
    ;   
\end{tikzpicture}
\end{center}

kommutativ macht. Dieser wird nach Anwenden von T zum kanonischen Isomorphismus 
\[E(p)\otimes_{E(p)}Tp\longrightarrow p\]
und ist daher nach Lemma \ref{iso} auch ein Isomorphismus. Den Isomorphismus 
\[X(p)\otimes_{E(p)}Tp\longrightarrow p\] 
bezeichnen wir mit $ can $.\\

Wir wollen zeigen, dass 
\[\tilde{T}:\mathcal{A}\longrightarrow \C(T):=\varinjlim_{F\subset D\text{ endlich}}\End(T_F)\Mod \]
voll ist.\\
Seien nun $ p_1, p_2 $ Objekte aus $ \mathcal{A} $. Jeder Morphismus aus $ \Hom_{\C(T)}(\tilde{T}(p_1),\tilde{T}(p_2)) $ hat für eine geeignete endliche Teilmenge $ F $ von $ \mathcal{A} $ schon einen Repräsentanten 
\[ Tp_1\xrightarrow{\varphi}Tp_2 \] 
in $ \End(T_F)\Mod $. Es genügt also zu zeigen, dass für alle endlichen Teilmengen $ F $ von $ \mathcal{A} $ mit $ p_1,p_2\in F $ die Abbildung
\[ \tilde{T}_{|F}:\Hom_{\mathcal{A}}(p_1,p_2)\longrightarrow\Hom_{E(p)}(Tp_1,Tp_2) \]
surjektiv ist. \\

Sei $ p=\bigoplus_{q_i\in F}q_i $ und $ f\in\End_{\mathcal{A}}(p) $ ein beliebiger Endomorphismus.
Betrachten wir folgendes Diagramm:

\begin{center}
\begin{tikzpicture}[description/.style={fill=white,inner sep=2pt}]
    \matrix (m) [matrix of math nodes, row sep=3.5em,
    column sep=3.0em, text height=1.5ex, text depth=0.25ex]
    { p & X(p)\otimes_{E(p)}Tp \\
      p & X(p)\otimes_{E(p)}Tp \\
	};
    \path[->,font=\scriptsize]
    (m-1-2) edge node[description] {$ can $} (m-1-1)
    (m-2-2) edge node[description] {$ can $} (m-2-1)
    (m-1-1) edge node[description] {$ f $} (m-2-1)
    (m-1-2) edge node[description] {$ id\otimes Tf $} (m-2-2)
    ;   
\end{tikzpicture}
\end{center}

Wenden wir hierauf $ T $ an, so erhalten wir das Diagramm

\begin{center}
\begin{tikzpicture}[description/.style={fill=white,inner sep=2pt}]
    \matrix (m) [matrix of math nodes, row sep=3.5em,
    column sep=3.0em, text height=1.5ex, text depth=0.25ex]
    { Tp & Tp \\
      Tp & Tp, \\
	};
    \path[->,font=\scriptsize]
    (m-1-2) edge node[description] {$ id $} (m-1-1)
    (m-2-2) edge node[description] {$ id $} (m-2-1)
    (m-1-1) edge node[description] {$ Tf $} (m-2-1)
    (m-1-2) edge node[description] {$ Tf $} (m-2-2)
    ;   
\end{tikzpicture}
\end{center}

das offensichtlich kommutiert. Es gilt also $ id\circ Tf-Tf\circ id=0 $. Da $ T $ treu ist, gilt auch $ can\circ f- id\otimes  Tf\circ can=0 $. Demnach kommutiert auch das obere Diagramm in $ \mathcal{A} $.
Der Isomorphismus $ can $ induziert also eine 1:1-Beziehung
\[
\begin{array}{ccc}
\End_{\mathcal{A}}(p) & \overset{1:1}{\longleftrightarrow} & \End_{\mathcal{A}}(X(p)\otimes_{E(p)}Tp) \\
f & \mapsto & id \otimes Tf , \\
\end{array} 
\]
welche unter dem Funktor $ T $ auf die gleichen Morphismenmengen abgebildet werden. 
In Satz \ref{wesentlich_surjektiv} haben wir gesehen, dass die Verkettung
\[
\begin{array}{ccccc}
\End(T_{|F})\Mod & \longrightarrow & \mathcal{A} & \overset{\tilde{T}}{\longrightarrow} & \End(T_{|F})\Mod\\
Tp & \mapsto & X(p)\otimes_{E(p)}Tp & \mapsto & Tp
\end{array}
\]
die Identität auf $ \End(T_{|F})\Mod $ induziert.
Also ist die Abbildung 
\[ \tilde{T}:\End_{\mathcal{A}}(X(p)\otimes_{E(p)}Tp)\longrightarrow\End_{E(p)}(Tp) \]
surjektiv. Dann sind auch, unter Verwendung von $ can $,
\[ \tilde{T}:\Hom_{\mathcal{A}}(p,p)\longrightarrow\Hom_{E(p)}(Tp,Tp), \]
bzw.
\[ \tilde{T}:\Hom_{\mathcal{A}}(\bigoplus_{q_i\in F}q_i,\bigoplus_{q_i\in F}q_i)\longrightarrow\Hom_{E(p)}(\bigoplus_{q_i\in F}Tq_i , \bigoplus_{q_i\in F}Tq_i) \]
surjektiv.\\
Da additive Funktoren Abbildungen zwischen direkten Summanden respektieren, können wir diese Abbildung zur surjektiven Abbildung
\[ \tilde{T}:\Hom_{\mathcal{A}}(p_1,p_2)\longrightarrow\Hom_{E(p)}(Tp_1,Tp_2) \]
einschränken. 
Also ist für eine endliche Teilmenge $ F $ von $ \mathcal{A} $ und $ p=\bigoplus_{q_i\in F}q_i $ der eingeschränkte Funktor
\[ \tilde{T}_{|F}:F\longrightarrow E(p)\Mod=\End(T_{|F})\Mod \]
voll und damit auch $ \tilde{T} $.

\end{proof}

\begin{Satz}\label{Aequivalenz}
Sei $ R $ ein noetherscher Ring und $ \mathcal{A} $ eine abelsche, $ R $-lineare Kategorie, also insbesondere ein Diagramm. Sei 
\[ T:\mathcal{A}\rightarrow R\Mod \]
ein treuer, exakter, $ R $-linearer Funktor und 
\[  \mathcal{A}\xrightarrow{\tilde{T}}\mathcal{C}(T_{\mathcal{A}})\xrightarrow{ff_T}R\Mod   \]
die Faktorisierung über die Diagrammkategorie $ \mathcal{C}(T_{\mathcal{A}}) $. Dann ist $ \tilde{T} $ ein Äquivalenz von Kategorien
\end{Satz}

\begin{proof}
Der Funktor $ T $ und der Vergissfunktor $ ff_T $ sind treu. Da $ T=ff_t \circ \tilde{T} $ gilt, ist auch $ \tilde{T} $ treu.
Außerdem ist $ \tilde{T} $ nach Satz \ref{wesentlich_surjektiv} wesentlich surjektiv und nach Satz \ref{voll} voll, also eine Äquivalenz von Kategorien.
\end{proof}

\chapter{Universelle Eigenschaft}
\label{ch:kap_4}
Sei in diesem Kapitel wieder $ R $ stets ein kommutativer, noetherscher Ring und $ R\Mod $ die Kategorie der endlich erzeugten $ R $-Moduln.\\
Ziel des Kapitels ist, die universelle Eigenschaft der Diagrammkategorie $ \C(T) $ zu beweisen. Wir beginnen mit einem kleinen kategorientheoretischen Lemma:

\begin{Lemma}\label{zui}
Seien 
\[ \mathcal{F},\mathcal{G}:\C\longrightarrow\mathcal{D} \]
zwei treue exakte Funktoren zwischen abelschen Kategorien, die auf einer zu $ \C $ äquivalenten Unterkategorie $ \C' $ von $ \C $ übereinstimmen. Dann ist $ \mathcal{F} $ natürlich isomorph zu $ \mathcal{G} $.
\end{Lemma}

\begin{proof}
Sei $ f:X\rightarrow Y $ ein Morphismus zwischen Objekten aus $ \C $. Da $ \C $ äquivalent zu $ \C' $ ist, finden wir $ f':X'\rightarrow Y' $ in $ \C' $ sowie zwei Isomorphismen $ \tau_X:X\rightarrow X' $, $ \tau_Y:Y\rightarrow Y' $, so dass folgendes Diagramm in $ \C $ kommutiert:
\begin{center}
\begin{tikzpicture}[description/.style={fill=white,inner sep=2pt}]
    \matrix (m) [matrix of math nodes, row sep=3em,
    column sep=6.5em, text height=1.5ex, text depth=0.25ex]
    {  X  & Y \\
      X' & Y' \\
	};
    \path[->,font=\scriptsize]
    (m-1-1) edge node[description] {$ f $} (m-1-2)
    (m-2-1) edge node[description] {$ f' $} (m-2-2)
    (m-1-1) edge node[description] {$ \tau_X $} (m-2-1)
    (m-1-2) edge node[description] {$ \tau_Y $} (m-2-2)
;    
    
\end{tikzpicture}
\end{center}

Da die Funktoren $ \mathcal{F} $ und $ \mathcal{G} $ auf $ \C' $ übereinstimmen, erhalten wir folgendes kommutative Diagramm in $ \mathcal{D} $:

\begin{center}
\begin{tikzpicture}[description/.style={fill=white,inner sep=2pt}]
    \matrix (m) [matrix of math nodes, row sep=3.5em,
    column sep=8.5em, text height=1.5ex, text depth=0.25ex]
    {  \mathcal{F}(X)  & \mathcal{F}(Y) \\
      \mathcal{F}(X')=\mathcal{G}(X') & \mathcal{F}(Y')=\mathcal{G}(Y') \\
      \mathcal{G}(X)  & \mathcal{G}(Y) \\
	};
    \path[->,font=\scriptsize]
    (m-1-1) edge node[auto] {$ \mathcal{F}(f) $} (m-1-2)
    (m-2-1) edge node[auto] {$ \mathcal{F}(f')=\mathcal{G}(f') $} (m-2-2)
    (m-1-1) edge node[auto] {$ \mathcal{F}(\tau_X) $} (m-2-1)
    (m-1-2) edge node[auto] {$ \mathcal{F}(\tau_Y) $} (m-2-2)
    (m-3-1) edge node[auto] {$ \mathcal{G}(\tau_X) $} (m-2-1)
    (m-3-2) edge node[auto] {$ \mathcal{G}(\tau_Y) $} (m-2-2)
    (m-3-1) edge node[auto] {$ \mathcal{G}(f) $} (m-3-2)

;    
    
\end{tikzpicture}
\end{center}

Da nach Lemma \ref{iso} treue, exakte Funktoren zwischen abelschen Kategorien Isomorphismen auf Isomorphismen abbilden, sind die Abbildungen $ \mathcal{G}^{-1}(\tau_X)\circ\mathcal{F}(\tau_X) $ für alle $ X\in \C $ Isomorphismen und bilden gemeinsam einen natürlichen Isomorphismus zwischen $ \mathcal{F} $ und $ \mathcal{G} $.

\end{proof}

\begin{Proposition}\cite{Nori}\label{hilfsprop}
Seien $ D_1 $, $ D_2 $ Diagramme und $ F:D_1\rightarrow D_2 $ eine Abbildung von Diagrammen. Sei weiter $ T:D_2\rightarrow R\Mod $ eine Darstellung und
\[ D_2\xrightarrow{\tilde{T}}\C(T)\xrightarrow{ff_T} R\Mod \]
die Faktorisierung von $ T $ über die Diagrammkategorie $ \C(T) $, sowie
\[ D_1\xrightarrow{\widetilde{T\circ F}}\C(T\circ F)\xrightarrow{ff_{T\circ F}} R\Mod \]
die Faktorisierung von $ T\circ F $.\\
Dann existiert ein bis auf Isomorphie eindeutiger treuer $ R $-linearer, exakter Funktor $ \mathcal{F} $, so dass folgendes Diagramm kommutiert:

\begin{center}
\begin{tikzpicture}[description/.style={fill=white,inner sep=2pt}]
    \matrix (m) [matrix of math nodes, row sep=3em,
    column sep=2.5em, text height=1.5ex, text depth=0.25ex]
    {  D_1 & & D_2 \\
      \C(T\circ F) &  & \C(T) \\
      & R\Mod & \\
	};
    \path[->,font=\scriptsize]
    (m-1-1) edge node[description] {$ F $} (m-1-3)
    (m-1-1) edge node[description] {$ \widetilde{T\circ F} $} (m-2-1)
    (m-1-3) edge node[description] {$ \tilde{T} $} (m-2-3)
    (m-2-1) edge node[description] {$ ff_{T\circ F} $}  (m-3-2) 
    (m-2-3) edge node[description] {$ ff_{T} $}  (m-3-2) ;
    \path[->,densely dotted]
    (m-2-1) edge node [description] {$\mathcal{F}$}  (m-2-3);    
    
\end{tikzpicture}
\end{center}

\end{Proposition}\label{df}

\begin{proof}
Seien zunächst $ D_1 $, $ D_2 $ wieder endliche Diagramme.Wir suchen also einen eindeutigen Funktor $ \mathcal{F} $ in folgendem Diagramm.

\begin{center}

\begin{tikzpicture}[description/.style={fill=white,inner sep=2pt}]
    \matrix (m) [matrix of math nodes, row sep=3em,
    column sep=2.5em, text height=1.5ex, text depth=0.25ex]
    { & & D_1 & & D_2 & &  \\
      & & \End(T\circ F)\Mod &  & \End(T)\Mod &  & (1)\\
      & & & R\Mod & & &  \\
	};
    \path[->,font=\scriptsize]
    (m-1-3) edge node[description] {$ F $} (m-1-5)
    (m-1-3) edge node[description] {$ \widetilde{T\circ F} $} (m-2-3)
    (m-1-5) edge node[description] {$ \tilde{T} $} (m-2-5)
    (m-2-3) edge node[description] {$ ff_{T\circ F} $}  (m-3-4) 
    (m-2-5) edge node[description] {$ ff_{T} $}  (m-3-4) ;
    \path[->,densely dotted]
    (m-2-3) edge node [description] {$\mathcal{F}$}  (m-2-5);    
    
\end{tikzpicture}

\end{center}

Ein Element $ (e_{q_i})_{q_i\in D_2}\in\End(T) $  hat die Eigenschaft, dass für alle $ q_i,q_j\in D_2 $, $m\in M(q_1,q_2) $ Diagramme der Form 

\begin{center}
\begin{tikzpicture}[description/.style={fill=white,inner sep=2pt}]
    \matrix (m) [matrix of math nodes, row sep=3em,
    column sep=2.5em, text height=1.5ex, text depth=0.25ex]
    {  Tq_i & & Tq_j \\
      Tq_i & & Tq_j \\
	};
    \path[->,font=\scriptsize]
    (m-1-1) edge node[auto] {$ e_{q_i} $} (m-2-1)
    (m-1-3) edge node[auto] {$ e_{q_j} $} (m-2-3)
    (m-1-1) edge node[auto] {$T(m)$} (m-1-3)
    (m-2-1) edge node[auto] {$T(m)$} (m-2-3);
\end{tikzpicture}
\end{center}

kommutieren. Dann kommutieren aber auch für alle $ p_i,p_j\in D_1 $ und $ m'\in M(p_i,p_j) $ die Diagramme der Form

\begin{center}
\begin{tikzpicture}[description/.style={fill=white,inner sep=2pt}]
    \matrix (m) [matrix of math nodes, row sep=3em,
    column sep=2.5em, text height=1.5ex, text depth=0.25ex]
    {  (T\circ F)(p_i) & & (T\circ F)(p_j) \\
      (T\circ F)(p_i) & & (T\circ F)(p_j) \\
	};
    \path[->,font=\scriptsize]
    (m-1-1) edge node[auto] {$ e_{p_i} $} (m-2-1)
    (m-1-3) edge node[auto] {$ e_{p_j} $} (m-2-3)
    (m-1-1) edge node[auto] {$TF(m')$} (m-1-3)
    (m-2-1) edge node[auto] {$TF(m')$} (m-2-3);
\end{tikzpicture}
\end{center}

Das heißt, wir können die Abbildung
\[ \prod_{q\in D_2}\End_R(Tq)\twoheadrightarrow\prod_{q\in F(D_1)}\End_R(Tq)\rightarrow\prod_{p\in D_1}\End_R(T\circ F)(p)) \]
zu einer Abbildung
\[ \circ F:\End(T)\rightarrow\End(T\circ F) \]
einschränken. Um das Diagramm (1) für ein $ p\in D_1 $ kommutativ zu machen, muss auf $ (T\circ F)(p) $ die $ \End(T) $-Modulstruktur gerade mittels $ \circ F $ zu der Modulstruktur in \\
$ \End(T\circ F)\Mod $ zurückgezogen werden. Dies entspricht gerade der Restriktion der Skalare des Morphismus $ \circ F $. Restriktion der Skalare ist nach Lemma \ref{Lemma_Skalare} treu, $ R $-linear und exakt. \\
Ein Funktor, der das Diagramm (1) zum Kommutieren bringt, wirkt auf der Familie $ \widetilde{T\circ F}(D_1)\subset \End(T\circ F)\Mod $ also durch Restriktion der Skalare mittels $ \circ F $. Wir müssen zeigen, dass diese Bedingung den Funktor auf beliebigen $ \End(T\circ F) $-Moduln schon bis auf Isomorphie eindeutig festlegt. Sei $ \mathcal{F} $ also der Funktor Restriktion der Skalare und 
\[ \mathcal{G}:\End(T\circ F)\Mod\longrightarrow\End(T)\Mod \]
ein weiterer treuer, exakter Funktor, der auf der Familie 
\[ \widetilde{T\circ F}(D_1)\subset \End(T\circ F)\Mod \]
der Restriktion der Skalare entspricht. Die Tatsache, dass $ \mathcal{G} $ nach Voraussetzung $ R $-linear und exakt ist, legt $ \mathcal{G} $ auch auf endlichen direkten Summen, Kernen und Kokern, bzw. der von $ (T\circ F)(D_1) $ erzeugten abelschen Kategorie eindeutig fest. \\
Wir bezeichnen wieder mit $ \langle \widetilde{T\circ F}(D_1) \rangle $ die von $ \widetilde{T\circ F}(D_1) $ erzeugte abelsche Kategorie. \\
Die Einschränkung des Vergissfunktors $ ff_{T\circ F} $ auf
\[ \langle \widetilde{T\circ F}(D_1) \rangle \overset{ff_{T\circ F}}{\longrightarrow} R\Mod \]
ist ebenfalls treu, exakt und $ R $-linear, also ist $ \langle \widetilde{T\circ F}(D_1) \rangle $ nach Satz \ref{Aequivalenz} äquivalent zu seiner Diagrammkategorie $ \End(T\circ F)\Mod $. Der Funktor $ \mathcal{G} $ muss also auf der zu $ \End(T\circ F)\Mod $ äquivalenten Unterkategorie $ \langle \widetilde{T\circ F}(D_1) \rangle $ mit $ \mathcal{F} $ übereinstimmen, ist also nach Lemma \ref{zui} natürlich isomorph zu $ \mathcal{F} $. Dies beweist die Aussage für endliche Diagramme $ D_1,D_2 $.\\
Sind $ D_1 $, $ D_2 $ nicht endlich, so bilden die endlichen Teilmengen $ E\subset D_1 $ bzw. $ F\subset D_2 $ wieder bezüglich Inklusion filtrierende Systeme. Für endliche Teilmengen $ E,E'\subset D_1 $ mit  $ E\subset E' $ erhalten wir ein kommutatives Diagramm

\begin{center}
\begin{tikzpicture}[description/.style={fill=white,inner sep=2pt}]
    \matrix (m) [matrix of math nodes, row sep=3em,
    column sep=2.5em, text height=1.5ex, text depth=0.25ex]
    {  \prod_{p\in E}\End(T\circ F(p)) & \prod_{p\in E'}\End(T\circ F(p)) \\
       \prod_{q\in F(E)}\End(Tq) & \prod_{q\in F(E')}\End(Tq) \\
	};
    \path[->>,font=\scriptsize]
    (m-1-2) edge node[auto] {} (m-1-1)
    (m-2-2) edge node[auto] {} (m-2-1);
     \path[->,font=\scriptsize]
    (m-1-1) edge node[auto] {} (m-2-1)
    (m-1-2) edge node[auto] {} (m-2-2);
\end{tikzpicture}
\end{center}

welches wir wieder einschränken können zu 

\begin{center}
\begin{tikzpicture}[description/.style={fill=white,inner sep=2pt}]
    \matrix (m) [matrix of math nodes, row sep=3em,
    column sep=2.5em, text height=1.5ex, text depth=0.25ex]
    {  \End(T\circ F_{|E}) & \End(T\circ F_{|E'}) \\
	   \End(T_{|F(E)}) & \End(T_{|F(E')}) \\
	};
    \path[->,font=\scriptsize]
    (m-1-2) edge node[auto] {} (m-1-1)
    (m-2-2) edge node[auto] {} (m-2-1);
     \path[->,font=\scriptsize]
    (m-1-1) edge node[auto] {} (m-2-1)
    (m-1-2) edge node[auto] {} (m-2-2);
\end{tikzpicture}
\end{center}

Durch Restriktion der Skalare erhalten wir also ähnlich wie in Proposition  \ref{Kategorienproposition} zwei filtrierende Systeme von abelschen Kategorien mit treuen, exakten, $ R $-linearen Übergangsfunktoren, sowie folgenden treuen, exakten, $ R $-linearen Funktoren von einem System in das andere.

\begin{center}
\begin{tikzpicture}[description/.style={fill=white,inner sep=2pt}]
    \matrix (m) [matrix of math nodes, row sep=3em,
    column sep=4.5em, text height=1.5ex, text depth=0.25ex]
    {  ...\End(T\circ F_{|E})\Mod & \End(T\circ F_{|E'})\Mod... \\
	   ...\End(T_{|F(E)})\Mod & \End(T_{|F(E')})\Mod... \\
	};
    \path[->,font=\scriptsize]
    (m-1-1) edge node[auto] {$ \phi_{E',E} $} (m-1-2)
    (m-2-1) edge node[auto] {$ \phi_{F(E'),F(E)} $} (m-2-2);
     \path[->,font=\scriptsize]
    (m-1-2) edge node[auto] {$\mathcal{F}_{|E'}$} (m-2-2)
    (m-1-1) edge node[auto] {$\mathcal{F}_{|E}$} (m-2-1);
\end{tikzpicture}
\end{center}

Das Diagramm kommutiert nach Konstruktion für alle $ E,E' $. \\
Die Diagrammkategorien sind definiert durch 
\[  \C(T)=\varinjlim_{E\subset D_2\text{ endlich}}\End(T_{|E})\Mod \]
und
\[  \C(T\circ F)=\varinjlim_{E\subset D_1\text{ endlich}}\End(T\circ F_{|E})\Mod. \]
Die Limiten existieren nach Proposition \ref{Kategorienproposition} und sind $ R $-linear und abelsch.\\
Wir suchen nun einen bis auf Isomorphie eindeutigen Funktor
\[ \mathcal{F}:\C(T\circ F)\rightarrow \C(T), \] so dass folgendes Diagramm kommutiert:

\begin{center}
\begin{tikzpicture}[description/.style={fill=white,inner sep=2pt}]
    \matrix (m) [matrix of math nodes, row sep=3em,
    column sep=4.5em, text height=1.5ex, text depth=0.25ex]
    {  ...\End(T\circ F_{|E})\Mod & & \End(T\circ F_{|E'})\Mod...  \\
& \C(T\circ F) & \\
  ...\End(T_{F(E)})\Mod & & \End(T\circ F_{F(E')})\Mod...  \\
   & & \\
    & \C(T) & \\
	};
    \path[->,font=\scriptsize]
    (m-1-1) edge node[description] {$ \phi_{E'E} $} (m-1-3)
    (m-1-1) edge node[description] {$ \psi_E $} (m-2-2)
    (m-1-3) edge node[description] {$ \psi_{E'} $} (m-2-2)
    (m-3-1) edge node[description] {$ \phi_{F(E')F(E)} $} (m-3-3)
    (m-3-1) edge node[description] {$ \chi_{F(E)} $} (m-5-2)
    (m-3-3) edge node[description] {$ \chi_{F(E')} $} (m-5-2)
    (m-2-2) edge node[description] {$ \exists\mathcal{F}? $} (m-5-2)
    (m-1-1) edge node[description] {$ \mathcal{F}_{|E} $} (m-3-1)
    (m-1-3) edge node[description] {$ \mathcal{F}_{|E'} $} (m-3-3)
    ;
\end{tikzpicture}
\end{center}

Die Existenz eines solchen Funktors entspricht genau der universellen Eigenschaft der Limeskategorie $ \C(T\circ F) $.

\end{proof}

Wir zeigen nun die universelle Eigenschaft der Diagrammkategorie:

\begin{Satz}
Sei $ D\xrightarrow{T} R$-Mod ein Diagramm und 
\[ D\xrightarrow{\tilde{T}}\C(T)\xrightarrow{ff_T}R\Mod   \]
die Faktorisierung von $ T $ über die Diagrammkategorie. Sei $\mathcal A$ eine weitere $R$-lineare abelsche Kategorie, $F:D\rightarrow \mathcal{A}$ eine Darstellung, sowie  $f:\mathcal{A} \rightarrow R\Mod$ ein treuer, exakter, $R$-linearer Funktor in die Kategorie der R-Moduln, so dass $ T=f \circ F $ faktorisiert. Dann existiert ein bis auf Isomorphismus eindeutiger $R$-linearer, exakter, treuer Funktor 
\[L(F):\mathcal{C}(T) \rightarrow \mathcal{A}, \]
so dass folgendes Diagramm kommutiert:

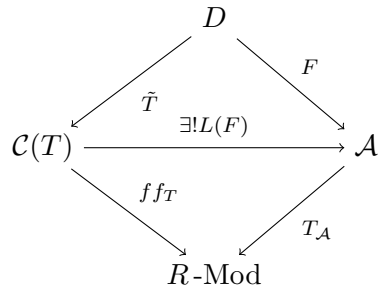
\begin{center}
\begin{tikzpicture}[description/.style={fill=white,inner sep=2pt}]
    \matrix (m) [matrix of math nodes, row sep=3em,
    column sep=2.5em, text height=1.5ex, text depth=0.25ex]
    {  & D &  \\
      \mathcal{C}(T) &  & \mathcal{A} \\
      & R\Mod & \\
	};
    \path[->,font=\scriptsize]
    (m-1-2) edge node[auto] {$ \tilde{T} $} (m-2-1)
    (m-1-2) edge node[auto] {$ F $} (m-2-3)
    (m-2-1) edge node[auto] {$ \exists !L(F) $} (m-2-3)
    (m-2-1) edge node[auto] {$ ff_T $} (m-3-2)
    (m-2-3) edge node[auto] {$ T_{\mathcal{A}} $}  (m-3-2); 

\end{tikzpicture}
\end{center}

\end{Satz}

\begin{proof}

Wir können $\mathcal{A}$ als Diagramm auffassen und erhalten eine Darstellung 
\[ \mathcal{A}\xrightarrow{T_{\mathcal{A}}}R\Mod, \]
welche über seine Diagrammkategorie
\[ \mathcal{A}\xrightarrow{\tilde{T_{\mathcal{A}}}}\C(T_{\mathcal{A}})\xrightarrow{ff_{T_{\mathcal{A}}}} R\Mod \]
faktorisiert. Wir erhalten folgendes kommutatives Diagramm:

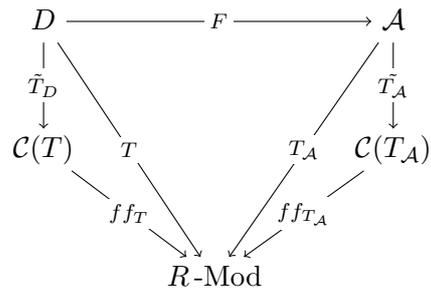
\begin{center}
\begin{tikzpicture}[description/.style={fill=white,inner sep=2pt}]
    \matrix (m) [matrix of math nodes, row sep=3em,
    column sep=2.5em, text height=1.5ex, text depth=0.25ex]
    {  D &  & \mathcal{A} \\
      \mathcal{C}(T) &  & \mathcal{C}(T_{\mathcal{A}}) \\
      & R\Mod & \\
	};
    \path[->,font=\scriptsize]
    (m-1-1) edge node[description] {$ \tilde{T}_D $} (m-2-1)
    (m-1-1) edge node[description] {$ F $} (m-1-3)
    (m-2-1) edge node[description] {$ ff_T $} (m-3-2)
    (m-1-3) edge node[description] {$ \tilde{T_{\mathcal{A}}}$} (m-2-3)
    (m-2-3) edge node[description] {$ ff_{T_{\mathcal{A}}} $}  (m-3-2)
    (m-1-1) edge node[description] {$ T $}  (m-3-2)
    (m-1-3) edge node[description] {$ T_{\mathcal{A}} $}  (m-3-2);   

\end{tikzpicture}
\end{center}

Nach Proposition \ref{hilfsprop} existiert dann ein bis auf Isomorphie eindeutiger $ R $-linearer, treuer, exakter Funktor $ \mathcal{F} $, so dass folgendes Diagramm kommutiert:

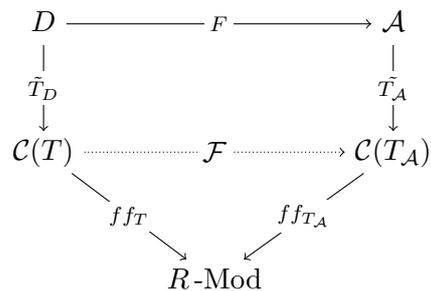
\begin{center}
\begin{tikzpicture}[description/.style={fill=white,inner sep=2pt}]
    \matrix (m) [matrix of math nodes, row sep=3em,
    column sep=2.5em, text height=1.5ex, text depth=0.25ex]
    {  D &  & \mathcal{A} \\
      \mathcal{C}(T) &  & \mathcal{C}(T_{\mathcal{A}}) \\
      & R\Mod & \\
	};
    \path[->,font=\scriptsize]
    (m-1-1) edge node[description] {$ \tilde{T}_D $} (m-2-1)
    (m-1-1) edge node[description] {$ F $} (m-1-3)
    (m-2-1) edge node[description] {$ ff_T $} (m-3-2)
    (m-1-3) edge node[description] {$ \tilde{T_{\mathcal{A}}}$} (m-2-3)
    (m-2-3) edge node[description] {$ ff_{T_{\mathcal{A}}} $}  (m-3-2);
    \path[->,densely dotted]
    (m-2-1) edge node[description] {$ \mathcal{F} $} (m-2-3);

\end{tikzpicture}
\end{center}

Da $ \mathcal{A} $ eine $ R $-lineare abelsche Kategorie und $ T $ ein treuer, exakter, $ R $-linearer Funktor ist, wissen wir aus Satz \ref{Aequivalenz}, dass $ \tilde{T}_{\mathcal{A}} $ eine Äquivalenz von Kategorien ist. Der bis auf Isomorphie eindeutige Funktor
\[L(F):\mathcal{C}(T) \rightarrow \mathcal{A} \]
ergibt sich dann aus Verkettung von $ \mathcal{F} $ mit dem Inversen von $ \tilde{T}_{\mathcal{A}} $. Ein Äquivalenz von $ R $-linearen Kategorien ist exakt, treu und $ R $-linear, also ist auch $ L(F) $ exakt, treu und $ R $-linear als Verkettung solcher Funktoren. 

\end{proof}

\chapter{Dualität zwischen Moduln und Komoduln}
\label{ch:kap_5}

Wir haben nun zu einem Diagramm $ D $ mit Darstellung $ T:D\rightarrow R\Mod $ Noris Diagrammkategorie $ \C(T) $ konstruiert. Ist $ R $ nicht nur ein kommutativer Ring, sondern sogar ein Körper, so ist die Kategorie $ \C(T) $ äquivalent zur Kategorie der endlich erzeugten Komoduln über einer gewissen Koalgebra. Der Schlüssel hierbei ist, dass es für endliche Algebren $ A $ über einem Körper $ k $ und einem endlich-dimensionalen Vektorraum $ V $ eine natürliche 1:1-Beziehung zwischen $ A $-Modulstrukturen und $ A^\vee $-Komodulstrukturen auf $ V $ gibt. $ A^\vee $ bezeichnet hierbei die zu $ A $ duale Koalgebra $ \Hom_k(A,k) $. Eine einfache Rechnung zeigt dann, dass 
\[ \C(T)=\varinjlim_{F\subset D}(\End(T_F)\Mod)=\varinjlim_{F\subset D}(\End(T_F)^\vee\Komod)=(\varinjlim_{F\subset D}\End(T_F)^\vee)\Komod \]
die Kategorie der endlich erzeugten $ (\varinjlim_{F\subset D}\End(T_F)^\vee) $-Komoduln ist, wobei $ F\subset D $ das gerichtete System der endlichen Unterdiagramme beschreibt. Diese Erkenntnis stellt sich als wesentlich für die Tannaka-Dualität heraus, welche wir in Kapitel 6 formulieren möchten. \\
Im allgemeinen Fall, dass $ R $ ein noetherscher Ring ist, gibt es diese Dualität leider nicht. Das Problem ist, dass es für $ R $-Moduln $ A $ und $ N $ im Allgemeinen keinen Isomorphismus
\[ A^\vee\otimes_R N\longrightarrow \Hom_k(A,N) \]
gibt. Ein ähnliches Dualitätsresultat erhalten wir jedoch, wenn wir uns auf endlich erzeugte, freie Moduln über Hauptidealringen einschränken. Ziel dieses Kapitels ist es, genaue Definitionen der Begriffe zu geben und die genannten Dualitäten zu erläutern.

\section{Definitionen}

Wir möchten zunächst eine weitere, äquivalente Definiton für $ R $-Algebren geben, welche die anschließenden Definitionen von Koalgebren und Komoduln motivieren sollen. 

\begin{Def}\label{algebra2}
Sei $ R $ ein kommutativer Ring. Eine \emph{unitäre Algebra} über $ R $ ist ein $ R $-Modul $ A $ zusammen mit zwei $ R $- Modulhomomorphismen
\[ \mu:A\otimes_R A\rightarrow A \hspace{1 cm} \text{und} \hspace{1 cm}  \eta:R\rightarrow A,  \]

sodass folgende Diagramme kommutieren

\begin{itemize}
\item Assoziativität:

\begin{center}
\begin{tikzpicture}[description/.style={fill=white,inner sep=2pt}]
    \matrix (m) [matrix of math nodes, row sep=3.5em,
    column sep=5.5em, text height=1.5ex, text depth=0.25ex]
    { A\otimes_R A\otimes_R A & A\otimes_R A \\
    	      A\otimes_R A & A   \\
	};
    \path[->,font=\scriptsize]
    (m-1-1) edge node[auto] {$ \mu\otimes id $} (m-1-2)
    (m-1-1) edge node[auto] {$ id\otimes\mu $} (m-2-1)
    (m-2-1) edge node[auto] {$ \mu $} (m-2-2)
    (m-1-2) edge node[auto] {$ \mu $} (m-2-2)
    ;   
\end{tikzpicture}
\end{center}

\item Verträglichkeit mit Skalarmultiplikation:

\begin{center}
\begin{tikzpicture}[description/.style={fill=white,inner sep=2pt}]
    \matrix (m) [matrix of math nodes, row sep=3.5em,
    column sep=5.5em, text height=1.5ex, text depth=0.25ex]
    { R\otimes_R A  &  A\otimes_R A & A\otimes_R R \\
    	    & A &    \\
	};
    \path[->,font=\scriptsize]
    (m-1-1) edge node[description] {$ \eta\otimes id  $} (m-1-2)
    (m-1-3) edge node[description] {$ id\otimes\eta $} (m-1-2)
    (m-1-2) edge node[auto] {$ \mu $} (m-2-2)
    (m-1-1) edge node[description] {$ Skalarmult. $} (m-2-2)
    (m-1-3) edge node[description] {$Skalarmult.$} (m-2-2)
    ;   
\end{tikzpicture}
\end{center}
\end{itemize}
Den Morphismus $ \eta $ nennt man \emph{Einheit}.
\end{Def}

\begin{Bem}
In der Literatur sind Algebren zum Teil nicht notwendig assoziativ, man würde eine solche Algebra als \emph{assoziative, unitäre Algebra} beschreiben. Für uns sind Algebren aber stets assoziativ. 
\end{Bem}

\begin{Lemma}
Die in Definition \ref{algebra1} und in Definition \ref{algebra2} gegebenen Definitionen einer unitären Algebra sind äquivalent.
\end{Lemma}

\begin{proof}
\ref{algebra1} $\Rightarrow$ \ref{algebra2} : \\
Ein Ringhomomorphismus $ \eta:R\rightarrow A $ definiert nach Bemerkung \ref{algebra_als_modul} eine $ R $-Modulstruktur auf $ A $. Die innere Multiplikation auf $ A $ liefert eine $ R $-bilineare, assoziative Abbildung $ A\times A\rightarrow A $, also eine assoziative Abbildung 
\[ \mu:A\otimes A\rightarrow A. \]
$ \eta $ ist ein Ringhomomorphismus, also insbesondere einen Morphismus von $ R $-Moduln. Bezeichnen wir die von $ \mu $ induzierte Multiplikation auf $ A $ mit $ \odot_A $. Gemäß Bemerkung \ref{algebra_als_modul} ist die $ R $-Skalarmultiplikation auf $ A $ als $ R $-Modul definiert durch \\
\[ r \cdot a:=\eta(r)\odot_A a \hspace{1 cm} \text{also} \hspace{1 cm} r\cdot a=\mu(\eta(r)\otimes a).  \]
Demnach kommutiert die Skalarmultiplikation mit der Abbildung $ \mu\circ (\eta\otimes id) $ und, da $ R $ kommutativ ist, auch mit $ \mu\circ (id\otimes \eta) $.\\
Für $ r\in R, a\in A $ gilt wegen der Kommutativität des zweiten Diagramms 
\[ \mu\circ (id\otimes\eta)(a\otimes r)= \mu\circ(\eta\otimes id)(r\otimes a), \]
also
\[ \mu(a\otimes\eta(r))=\mu(\eta(r)\otimes a). \]
Das heißt $ \eta(R) $ kommutiert unter $ A $-Multiplikation mit allen Elementen aus $ A $. $ \eta(R) $ liegt also im Zentrum von $ A $.\\
\vspace{1cm}
\\
\ref{algebra2} $\Rightarrow$ \ref{algebra1} : \\
Ein $ R $-Modul $ A $ wird durch eine assoziative Verknüpfung zu einem Ring. 
Das Element $ \eta(1) $ ist ein neutrales Element der Verknüpfung $ \mu $, da wegen der Kommutativität des zweiten Diagramms
\[ \eta(1)\odot_A a= \mu(\eta(1)\otimes a)=\mu\circ(\eta\otimes id)(1\otimes a)=1\cdot a=a \]
gilt.\\

Wir müssen überprüfen, ob $ \eta $ einen Ringhomomorphismus induziert. Dies rechnen wir nach, wobei wir das multiplikativ neutrale Element des Rings $ A $, $ \eta(1) $ mit $ 1_A $ und die $ R $-Skalarmultiplikation des Moduls $ A $ mit $ \odot_R $ notieren:\\
$ \eta(rs)=\mu(\eta(rs)\otimes 1_A)=(rs)\odot_R 1_A=r\odot_R (s\odot_R 1_A)=r\odot_R \mu(\eta(s)\otimes 1_A)\\ 
=\mu(\eta(r)\otimes (\mu(\eta(s)\otimes1_A)))=\mu(\mu(\eta(r)\otimes\eta(s))\otimes 1_A)=\mu(\eta(r)\otimes\eta(s)) $
\end{proof}

Wir definieren nun dual den Begriff der Koalgebra, wie man ihn zum Beispiel in \cite{MR2294803} findet:

\begin{Def}
Eine \emph{Koalgebra} über einem kommutativen Ring ist ein $ R $-Modul $ C $ zusammen mit zwei Abbildungen von $ R $- Moduln
\[ \Delta:C\rightarrow C\otimes_R C \hspace{1 cm} \text{und} \hspace{1 cm}  \epsilon:C\rightarrow R,  \]

sodass folgende Diagramme kommutieren

\begin{itemize}
\item Koassoziativität:

\begin{center}
\begin{tikzpicture}[description/.style={fill=white,inner sep=2pt}]
    \matrix (m) [matrix of math nodes, row sep=3.5em,
    column sep=5.5em, text height=1.5ex, text depth=0.25ex]
    { C & C\otimes_R C \\
    	      C\otimes_R C & C\otimes_R C\otimes_R C   \\
	};
    \path[->,font=\scriptsize]
    (m-1-1) edge node[auto] {$ \Delta $} (m-1-2)
    (m-1-1) edge node[auto] {$ \Delta $} (m-2-1)
    (m-2-1) edge node[auto] {$ \Delta\otimes id $} (m-2-2)
    (m-1-2) edge node[auto] {$ id\otimes\Delta $} (m-2-2)
    ;   
\end{tikzpicture}
\end{center}

\item Koeinheit:

\begin{center}
\begin{tikzpicture}[description/.style={fill=white,inner sep=2pt}]
    \matrix (m) [matrix of math nodes, row sep=3.5em,
    column sep=7.5em, text height=2.5ex, text depth=0.25ex]
    { R\otimes_R C  &  C\otimes_R C & C\otimes_R R \\
    	    & C &    \\
	};
    \path[->,font=\scriptsize]
    (m-1-2) edge node[description] {$ \Delta\otimes id  $} (m-1-1)
    (m-1-2) edge node[description] {$ id\otimes\Delta $} (m-1-3)
    (m-2-2) edge node[auto] {$ \epsilon $} (m-1-2)
    (m-1-1) edge node[description] {$ Skalarmult. $} (m-2-2)
    (m-1-3) edge node[description] {$Skalarmult.$} (m-2-2)
    ;   
\end{tikzpicture}
\end{center}

\end{itemize}
\end{Def}

\begin{Bem}
Die Abbildung $ \epsilon $ nennt man auch Koeinheit. Konsequenterweise müsste man eine derartige Koalgebra eine kounitäre Koalgebra nennen, das ist jedoch nicht üblich.
\end{Bem}

\begin{Def}
Seien $ C $, $ C' $ zwei $ R $-Koalgebren. Eine Abbildung $ f:C\rightarrow C' $ heißt \emph{Koalgebrenhomomorphismus}, wenn folgende Diagramme kommutativ sind:
\begin{center}
\begin{tikzpicture}[description/.style={fill=white,inner sep=2pt}]
    \matrix (m) [matrix of math nodes, row sep=3.5em,
    column sep=5.5em, text height=1.5ex, text depth=0.25ex]
    { C & C' \\
     C\otimes_R C & C'\otimes_R C'   \\
	};
    \path[->,font=\scriptsize]
    (m-1-1) edge node[auto] {$ f $} (m-1-2)
    (m-1-1) edge node[auto] {$ \Delta $} (m-2-1)
    (m-2-1) edge node[auto] {$ f\otimes f $} (m-2-2)
    (m-1-2) edge node[auto] {$ \Delta' $} (m-2-2)
    ;   
\end{tikzpicture}
\end{center}

\begin{center}
\begin{tikzpicture}[description/.style={fill=white,inner sep=2pt}]
    \matrix (m) [matrix of math nodes, row sep=3.5em,
    column sep=5.5em, text height=1.5ex, text depth=0.25ex]
    { C &  C' \\
    & R    \\
	};
    \path[->,font=\scriptsize]
    (m-1-1) edge node[auto] {$ f $} (m-1-2)
    (m-1-1) edge node[auto] {$ \epsilon $} (m-2-2)
    (m-1-2) edge node[auto] {$ \epsilon' $} (m-2-2)
    ;   
\end{tikzpicture}
\end{center}

\end{Def}

Dual zu Moduln über einer Algebra, definieren wir wie in \cite{MR2294803} Komoduln über einer Koalgebra.

\begin{Def}
Sei $ (C,\Delta,\epsilon) $ eine $ R $-Koalgebra. Ein \emph{$ C $-Rechtskomodul}  ist ein $ R $-Modul $ M $ zusammen mit einer $ R $-linearen Abbildung
\[ \rho:M\rightarrow M\otimes_R C, \] so dass folgende Diagramme kommutieren
\begin{itemize}
\item Verträglichkeit mit $ \Delta $:
\begin{center}
\begin{tikzpicture}[description/.style={fill=white,inner sep=2pt}]
    \matrix (m) [matrix of math nodes, row sep=3.5em,
    column sep=5.5em, text height=1.5ex, text depth=0.25ex]
    { M &  M\otimes_R C \\
    M\otimes_R C & M\otimes_R C\otimes_R C    \\
	};
    \path[->,font=\scriptsize]
    (m-1-1) edge node[auto] {$ \rho $} (m-1-2)
    (m-1-1) edge node[auto] {$ \rho $} (m-2-1)
    (m-1-2) edge node[auto] {$ id\otimes\Delta $} (m-2-2)
    (m-2-1) edge node[auto] {$ \rho\otimes id$} (m-2-2)
    ;   
\end{tikzpicture}
\end{center}
\item Verträglichkeit mit $ R $-Multiplikation:
\begin{center}
\begin{tikzpicture}[description/.style={fill=white,inner sep=2pt}]
    \matrix (m) [matrix of math nodes, row sep=3.5em,
    column sep=5.5em, text height=1.5ex, text depth=0.25ex]
    { M &  M\otimes_R C \\
     & M\otimes_R R    \\
	};
    \path[->,font=\scriptsize]
    (m-1-1) edge node[auto] {$ \rho $} (m-1-2)
    (m-2-2) edge node[auto] {$ Skalarmult. $} (m-1-1)
    (m-1-2) edge node[auto] {$ id\otimes\epsilon $} (m-2-2)
    ;   
\end{tikzpicture}
\end{center}
\end{itemize}
Analog kann man auch Links-Komoduln definieren. 
\end{Def}

\section{Dualität}

Sei $ A $ ein $ R $-Modul. Bezeichne $ A^\vee:=\Hom_R(A,R) $ den Dualraum des Moduls $ A $. Wir interessieren uns zunächst für die Frage, unter welchen Bedingungen eine Algebra $ A $ eine Koalgebren-Struktur auf $ A^\vee $ induziert und umgekehrt. Seien zunächst
\[ C\overset{\rho}{\longrightarrow} C\otimes_R C  \hspace{1 cm} \text{und} \hspace{1 cm}  R\overset{\epsilon}{\longrightarrow}C \]
die Komultiplikation und Koeinheit einer Koalgebra. Dann hat $ C^\vee=\Hom_R(C,R) $ mit Hilfe der dualen Abbildungen $ \epsilon^*,\rho^* $ die Struktur einer $ R $-Algebra:
\[\Hom(C,R)\otimes_R\Hom(C,R)\overset{\_\otimes\_}{\longrightarrow}\Hom(C\otimes_R C,R)\overset{\rho^*}{\longrightarrow}\Hom(C,R) \]
definiert eine Multiplikation und 
\[R\overset{can}{\longrightarrow}\Hom_R(R,R)\overset{\epsilon^*}{\longrightarrow}\Hom_R(C,R) \]
eine Einheit.
Umgekehrt definiert leider nicht jede Algebra auf kanonische Weise eine Koalgebra: \\
Sei $ A $ eine $ R $-Algebra und
\[ A\otimes_R A\overset{\mu}{\longrightarrow} A \]
die Algebrenmultiplikation. Dies induziert die duale Abbildung
\[A^\vee\overset{\mu^*}{\longrightarrow} (A\otimes_R A)^\vee. \]
Da diese Abbildung eine Komultiplikation auf $ A^\vee $ induzieren soll, stellt sich die Frage, unter welchen Voraussetzung es einen kanonischen Isomorphismus 
\[ (A\otimes_R A)^\vee\cong A^\vee\otimes_R A^\vee \]
gibt. Im Allgemeinen gilt das nicht. Für projektive Moduln ist dies relativ einfach zu sehen:

\begin{Def}
Ein Morphismus von $ R $-Moduln $ r:M\rightarrow N$ heißt \emph{Retraktion}, wenn es einen Morphismus $ i:N\rightarrow M $ gibt, sodass $ r\circ i=id_N $. Wenn zwischen $ M $ und $ N $ eine Retraktion existiert, nennen wir $ N $ ein \emph{Retrakt} von $ M $.
\end{Def}

\begin{Proposition}\label{xcv}
Ein Modul ist genau dann projektiv, wenn er Retrakt eines freien Moduls ist.
\end{Proposition}

\begin{proof}
\cite[Prop.5.1]{MR2294803}
\end{proof}

\begin{Satz}\label{projektiv}
Sei $ R $ ein kommutativer, unitärer Ring, sei $ A $ ein $ R $-Modul. Dann sind folgende Aussagen äquivalent:
\begin{enumerate}
\item $ A $ ist projektiv und endlich erzeugt
\item Die Abbildung 
\[
\begin{array}{cccc}
\rho^A_N: & A^\vee\otimes_R N & \rightarrow & \Hom(A,N) \\
 & \varphi\otimes n & \mapsto & (m\mapsto\varphi(m)\cdot n) \\
\end{array}
\]
ist für alle $ R $-Moduln $ N $ ein Isomorphismus.
\end{enumerate}
\end{Satz}

\begin{proof}
Der Beweis folgt in weiten Teilen der Ausführung \cite[Prop5.2]{MR2294803} und \\ \cite[Übung5.1]{MR2294803}.\\
$ 1\Rightarrow 2$:  Sei zunächst $ A= R^n $ frei und endlich erzeugt. Sei $ \varphi_1...\varphi_n $ die duale Basis von ${(R^n)}^\vee$, also $ \varphi_i(e_j)=\delta_{ij} $. Die Abbildung $ \rho^{R^n}_N$ ist dann gegeben durch
\[
\begin{array}{cccc}
\rho^{R^n}_N: & {R^n}^\vee\otimes_R N & \rightarrow & \Hom(R^n,N) \\
 & \varphi_k\otimes n & \mapsto & (e_i\mapsto\varphi_k(e_i)\cdot n=\delta_{ki}\cdot n). \\
\end{array}
\]
Mit Hilfe dieser Basiswahlen, können wir, wie im Fall endlich-dimensionaler Vektorräume, die Umkehrabbildung konkret angeben:
\[
\begin{array}{cccc}
\Hom(R^n,N) & \rightarrow & {R^n}^\vee\otimes_R N \\
f & \mapsto & \sum_{i=1}^n\varphi_i\otimes f(e_i). \\
\end{array}
\]
Die Abbildung $ \rho^A_N $ ist also ein Isomorphismus.\\
Ein beliebiger endlich erzeugter projektiver Modul $ A $ ist nun nach Proposition \ref{xcv} Retrakt eines freien Moduls $ F=\bigoplus_{i\in I}R $. Sei $ r:F\rightarrow A $ die Retraktion und $ i:A\rightarrow F $ ihr Rechtsinverses. Unter einem Homomorphismus $ i:A\rightarrow F $ trifft jedes Element eines Erzeugendensystems von $ A $ nur endlich viele Summanden in $ F $. Da $ A $ endlich erzeugt ist, landet $ A $ unter $ i $ also in einem endlich erzeugten Untermodul von $ F $. Wir können also OBdA annehmen, dass $ A $ ein Retrakt eines endlich erzeugten freien Moduls $ R^n $ ist. 
Wir betrachten das Diagramm

\begin{center}
\begin{tikzpicture}[description/.style={fill=white,inner sep=2pt}]
    \matrix (m) [matrix of math nodes, row sep=3.5em,
    column sep=5.0em, text height=2.0ex, text depth=0.25ex]
    { {R^n}^\vee\otimes_R N &  \Hom_R(R^n,N) \\
	{A}^\vee\otimes_R N &  \Hom_R(A,N)    \\
	};
    \path[->,font=\scriptsize]
    (m-1-1) edge node[description] {$ \rho_N^{R^n} $} (m-1-2)
    (m-2-1) edge node[description] {$ \rho_N^A $} (m-2-2)
    (m-1-1) edge node[description] {$ i^*\otimes id $} (m-2-1)
    (m-1-2) edge node[description] {$ ( \_ \circ i) $} (m-2-2)
    ;   
\end{tikzpicture}
\end{center}
und prüfen dessen Kommutativität: Sei $ \varphi\otimes n\in {R^n}^\vee\otimes N $. Dann gilt
\[
\rho^A_N\circ(i^*\otimes id)(\varphi\otimes n)=\rho^A_N((\varphi\circ i)\otimes n)=(Abb: a\mapsto \varphi(i(a))n)=(\_\circ i)\circ \rho_N^{R^n}(\varphi\otimes n).
\]
Da alle Abbildungen linear sind, gilt dies auch für beliebige Elemente $ \sum_{i=1}^n\varphi_i\otimes n_i $.\\
Analog überprüft man die Kommutativität von

\begin{center}
\begin{tikzpicture}[description/.style={fill=white,inner sep=2pt}]
    \matrix (m) [matrix of math nodes, row sep=3.5em,
    column sep=5.0em, text height=2.0ex, text depth=0.25ex]
    { {R^n}^\vee\otimes_R N &  \Hom_R(R^n,N) \\
	{A}^\vee\otimes_R N &  \Hom_R(A,N)    \\
	};
    \path[->,font=\scriptsize]
    (m-1-1) edge node[description] {$ \rho_N^{R^n} $} (m-1-2)
    (m-2-1) edge node[description] {$ \rho_N^A $} (m-2-2)
    (m-2-1) edge node[description] {$ r^*\otimes id $} (m-1-1)
    (m-2-2) edge node[description] {$ ( \_ \circ r) $} (m-1-2)
    ;   
\end{tikzpicture}
\end{center}
Da $ \rho^{R^n}_N $ ein Isomorphismus ist, können wir die Umkehrabbildung von $ \rho^A_N $ durch
\[ (\rho^A_N)^{-1}:=(i^*\otimes id)\circ(\rho^{R^n}_N)^{-1}\circ(\_\circ r) \]
definieren und rechnen nach, dass dies auch wirklich eine Umkehrabbildung ist:
\[
\begin{split}
& \rho^A_N\circ  (\rho^A_N)^{-1}  \\
=  & \rho^A_N\circ (i^*\otimes id)\circ(\rho^{R^n}_N)^{-1}\circ(\_\circ r) \\
= & (\_\circ i)\circ\rho^{R^n}_N\circ(\rho^{R^n}_N)^{-1}\circ(\_\circ r) \\
= & (\_\circ i)\circ (\_\circ r)=(\_\circ id) \\
= & id_{\Hom(A,N).}
\end{split}
\]
Ebenso gilt:
\[
\begin{split}
& (\rho^A_N)^{-1} \circ  (\rho^A_N) \\
=  & (i^*\otimes id)\circ(\rho^{R^n}_N)^{-1}\circ(\_\circ r)\circ \rho^A_N \\
= & (i^*\otimes id)\circ\rho^{R^n}_N\circ(\rho^{R^n}_N)^{-1}\circ (r^*\otimes id) \\
= & (i^*\otimes id)\circ (r^*\otimes id) \\
= & (r\circ i)^*\otimes id \\
= & id_{A^\vee\otimes_R N.}
\end{split}
\]
\\
$ 2\Rightarrow 1$: Das Urbild von $id_A$ unter dem Isomorphismus 
\[ \rho^A_A:A^\vee\otimes_R A\xrightarrow{\sim}\Hom_R(A,A) \]
liefert uns ein ausgezeichnetes Element $ \sum_{i=1}^n u_i\otimes a_i\in A^\vee\otimes A $. \\
Für alle $ a\in A $ gilt 
\[ a=id_A(a)=(\rho^A_A(\sum_{i=1}^n u_i\otimes a_i))(a)=\sum_{i=1}^n u_i(a)\cdot a_i. \]
Also lässt sich jedes Element des Moduls als Linearkombination der $(a_1...a_n)$ darstellen und der Modul $ A $ ist endlich erzeugt. \\
Wir zeigen noch Projektivität:
Seien $ f:A\rightarrow L'$ und $ g:L\twoheadrightarrow L'$ zwei Morphismen. Wir erhalten ein Diagramm 

\begin{center}
\begin{tikzpicture}[description/.style={fill=white,inner sep=2pt}]
    \matrix (m) [matrix of math nodes, row sep=3.5em,
    column sep=5.0em, text height=2.0ex, text depth=0.25ex]
    { A^\vee\otimes_R L &  \Hom_R(A,L) \\
	 A^\vee\otimes_R L' &  \Hom_R(A,L')    \\
	};
    \path[->,font=\scriptsize]
    (m-1-1) edge node[description] {$ \rho_L^A $} (m-1-2)
    (m-2-1) edge node[description] {$ \rho_{L'}^A $} (m-2-2)
    (m-1-1) edge node[description] {$ id\otimes g $} (m-2-1)
    (m-1-2) edge node[description] {$ (g\circ\_) $} (m-2-2)
    ;   
\end{tikzpicture}
\end{center}

wobei die horizontalen Abbildungen nach Voraussetzung Isomorphismen sind.\\
Wir prüfen nach, dass das Diagramm kommutiert:

\begin{align*}
 \rho_{L'}^A \circ (id\otimes g)(u\otimes l) &=\rho^A_{L'}(u\otimes g(l)) \\
 &=  (Abb:a\mapsto u(a)g(l))\\
& =(Abb:a\mapsto g(u(a)l))  \\ 
& = (g\circ \_)\circ (Abb:a\mapsto u(a)l)=(g\circ\_)\circ\rho^A_L(a\otimes l). 
 \end{align*}

Wir müssen nun zeigen, dass die Abbildung $ (g\circ\_) $ surjektiv ist. Dies ist klar, da $ (id\otimes g) $ surjektiv ist und die horizontalen Abbildungen Isomorphismen sind. 
\end{proof}

\begin{Konvention}
Seien für den Rest dieses Kapitels $ R $ ein kommutativer Ring und alle Algebren über $ R $ endlich sowie alle Moduln über $ R $ endlich erzeugt. 
\end{Konvention}

\begin{Lemma}\label{natuerlich}
Seien $ A $ und $ B $ zwei $ R $-Moduln. Die Abbildungen
\[
\begin{array}{ccccccc}
A^\vee\otimes_R B^\vee & \xrightarrow{\rho^A_{B^\vee}}& \Hom(A,B^\vee) & \xrightarrow{\sim} & \Hom(A\otimes B,R) & =(A\otimes_R B)^\vee\\
\alpha\otimes\beta     &   \mapsto    & (a\mapsto\alpha(a)\cdot\beta) &     \mapsto        & (a\otimes b\mapsto \alpha(a)\cdot\beta(b))\\
\end{array}
\]
sind im folgenden Sinne natürlich: Gegeben zwei Modulmorphismen
\[ C \xrightarrow{f} A, \]
\[ D \xrightarrow{g} B, \]
so kommutiert folgendes Diagramm
\begin{center}
\begin{tikzpicture}[description/.style={fill=white,inner sep=2pt}]
    \matrix (m) [matrix of math nodes, row sep=3.5em,
    column sep=5.0em, text height=2.0ex, text depth=0.25ex]
    { A^\vee\otimes_R B^\vee &  \Hom_R(A,B^\vee) & (A\otimes_R B)^\vee \\
	  C^\vee\otimes_R D^\vee &  \Hom_R(C,D^\vee) & (C\otimes_R D)^\vee \\
	};
    \path[->,font=\scriptsize]
    (m-1-1) edge node[auto] {$ \rho_{B^\vee}^A $} (m-1-2)
    (m-1-2) edge node[auto] {$ \sim $} (m-1-3)
    (m-2-1) edge node[auto] {$ \rho_{D^\vee}^C $} (m-2-2)
    (m-2-2) edge node[auto] {$ \sim $} (m-2-3)
    (m-1-1) edge node[auto] {$ f^*\otimes g^* $} (m-2-1)
    (m-1-2) edge node[auto] {$ (g^*\circ\_\circ f) $} (m-2-2)
    (m-1-3) edge node[auto] {$ (f\otimes g)^* $} (m-2-3)
    ;   
\end{tikzpicture}
\end{center}

\end{Lemma}

\begin{proof}
Für das äußere Rechteck rechnen wir:\\

\begin{align*}
(f\otimes g)^* \circ \rho_{B^\vee}^A(\alpha\otimes\beta) &=(f\otimes g)^* \circ(Abb:a\otimes b\mapsto\alpha(a)\beta(b)) \\
&= (Abb:c\otimes d\mapsto (\alpha\circ f)(c)(\beta\circ g)(d)\\
&=\rho_{D^\vee}^C(\alpha\circ f\otimes \beta\circ g) \\
& = \rho_{D^\vee}^C\circ (f^* \otimes g^*). \\
\end{align*}

Für das rechte Quadrat sei $ \varphi\in\Hom_R(A,B^\vee) $.
Führen wir zunächst den kanonischen Isomorphismus und dann die Abbildung $ (f\otimes g)^* $ aus, erhalten wir:
\[ \varphi \longmapsto \Big{(}Abb:a\otimes b\mapsto \big{(}\varphi(a)\big{)}(b)\Big{)}\longmapsto \Big(Abb:c\otimes d\mapsto\varphi\big(f(c)\big)\big(g(d)\big)\Big). \]
Führen wir die Abbildung $ (g^*\circ\_\circ f) $ von dem kanonischen Isomorphismus aus erhalten wir gleichermaßen:
\[ \varphi \longmapsto \Big(Abb:c\otimes d\mapsto \big(g^*\circ\varphi\circ f(c)\big)(d)\Big)=\Big(Abb:c\otimes d\mapsto\varphi\big(f(c)\big)\big(g(d)\big)\Big). \]
\end{proof}

\begin{Korollar}\label{Koalgebra}
Sei $ A $ eine unitäre $ R $-Algebra und $ A $ als Modul projektiv. Dann trägt $ A^{\vee} $ eine natürliche Struktur als $ R $-Koalgebra.
\end{Korollar}

\begin{proof}
Wenn $ A $ eine $ R $-Algebra ist, dann haben wir eine assoziative, $ R $-bilineare Abbildung 
\[ \mu:A\otimes_R A\rightarrow A.\]
Dies definiert durch Pullback eine Abbildung
\[
\begin{array}{cccc}
\mu^*:A^\vee & \rightarrow & (A\otimes_R A)^\vee \\
 \varphi & \mapsto & \varphi\circ\mu. \\
\end{array}
\]

Dualisieren wir das Diagramm der Assoziativität von $ \mu $, so erhalten wir folgendes kommutative Diagramm:

\begin{center}
\begin{tikzpicture}[description/.style={fill=white,inner sep=2pt}]
    \matrix (m) [matrix of math nodes, row sep=3.5em,
    column sep=5.5em, text height=1.5ex, text depth=0.25ex]
    { (A\otimes_R A\otimes_R A)^\vee & (A\otimes_R A)^\vee \\
    	      (A\otimes_R A)^\vee & A^\vee   \\
	};
    \path[->,font=\scriptsize]
    (m-1-2) edge node[auto] {$ (\mu\otimes id)^* $} (m-1-1)
    (m-2-1) edge node[auto] {$ (id\otimes\mu)^* $} (m-1-1)
    (m-2-2) edge node[auto] {$ \mu^* $} (m-2-1)
    (m-2-2) edge node[auto] {$ \mu^* $} (m-1-2)
    ;   
\end{tikzpicture}
\end{center}

Da $ A $ projektiv ist, ist die Abbildung $ \rho^A_{A^\vee} $ nach Satz \ref{projektiv} ein natürlicher Isomorphismus und wir haben wie in Lemma \ref{natuerlich} natürliche Isomorphismen:
\[ 
\begin{split}
(A\otimes_R A)^\vee \overset{\sim}{\rightarrow} A^\vee\otimes_R A^\vee \\
(A\otimes_R A\otimes A)^\vee\overset{\sim}{\rightarrow} (A\otimes_R A)^\vee \otimes_R A^\vee\\
(A\otimes_R A\otimes A)^\vee\overset{\sim}{\rightarrow}  A^\vee\otimes_R (A\otimes_R A)^\vee\\
\end{split}
\]
und mit Hilfe von Lemma \ref{natuerlich} erhalten wir folgendes kommutative Diagramm
\begin{center}
\begin{tikzpicture}[description/.style={fill=white,inner sep=2pt}]
    \matrix (m) [matrix of math nodes, row sep=3.5em,
    column sep=5.5em, text height=1.5ex, text depth=0.25ex]
    { A^\vee\otimes_R A^\vee\otimes_R A^\vee & A^\vee\otimes_R A^\vee \\
    	      A^\vee\otimes_R A^\vee & A^\vee   \\
	};
    \path[->,font=\scriptsize]
    (m-1-2) edge node[auto] {$ \mu^*\otimes id^* $} (m-1-1)
    (m-2-1) edge node[auto] {$ id^*\otimes\mu^* $} (m-1-1)
    (m-2-2) edge node[auto] {$ \mu^* $} (m-2-1)
    (m-2-2) edge node[auto] {$ \mu^* $} (m-1-2)
    ;   
\end{tikzpicture}
\end{center}

was die Koassoziativität zeigt. Analog erhalten wir durch Dualisieren aus der Abbildung
\[ \eta:R\rightarrow A \] 
die Abbildung
\[
\begin{array}{cccccc}
 \eta^*: & A^\vee & \rightarrow & R^\vee & \overset{\sim}{\rightarrow} & R \\
 & \varphi & \mapsto & \varphi\circ\eta & \mapsto & \varphi(\eta(1)). \\
\end{array}
\]

Da die Skalarmultiplikationsabbildung ein Isomorphismus ist, erhalten wir wieder mit Hilfe von Lemma \ref{natuerlich} folgendes kommutative Diagramm

\begin{center}
\begin{tikzpicture}[description/.style={fill=white,inner sep=2pt}]
    \matrix (m) [matrix of math nodes, row sep=3.5em,
    column sep=7.5em, text height=2.5ex, text depth=0.25ex]
    { R\otimes_R C  &  C\otimes_R C & C\otimes_R R \\
    	    & C &    \\
	};
    \path[->,font=\scriptsize]
    (m-1-2) edge node[description] {$ \eta^*\otimes id  $} (m-1-1)
    (m-1-2) edge node[description] {$ id\otimes\eta^* $} (m-1-3)
    (m-2-2) edge node[auto] {$ \mu^* $} (m-1-2)
    (m-1-1) edge node[description] {$ Skalarmult. $} (m-2-2)
    (m-1-3) edge node[description] {$Skalarmult.$} (m-2-2)
    ;   
\end{tikzpicture}
\end{center}

was die Koeinheit zeigt.

\end{proof}

%

\begin{Bem}
Ein endlich erzeugter freier Modul $ M $ über einem kommutativen Ring $ R $ ist reflexiv, das heißt, $ M $ und $ (M^\vee)^\vee  $ sind kanonisch isomorph. Der Beweis funktioniert analog zu dem Standardbeweis für Vektorräume aus der Linearen Algebra. 
\end{Bem}

\begin{Bem+Def}\label{etr}
Wir erinnern noch einmal an die Adjunktion von $ \_\otimes_R M $ und $ \Hom(M,\_) $ in der Kategorie der $ R $-Moduln.
Für $ R $-Moduln $ A,M,P $ ist die natürliche Bijektion zwischen den Morphismenräumen definiert durch
\[
\begin{array}{cccc}
 \Hom_R(A\otimes_R M,P) &    \overset{1:1}{\longleftrightarrow} & \Hom_R(M,\Hom_R(A,P)) \\
f & \mapsto & (m\mapsto f(\_,m)) \\
(a\otimes m \mapsto \varphi_m(a)) & \leftarrowtail & (m\mapsto \varphi_m).
\end{array}
\]
Für einen Homomorphismus $ f\in \Hom(A\otimes_R M,P) $ bezeichnen wir mit $ \hat{f} $ den bezüglich dieser Bijektion korrespondierenden Homomorphismus in $ \Hom(M,\Hom_R(A,P)) $. Umgekehrt bezeichnen wir für einen Homomorphismus $ g\in\Hom(M,\Hom_R(A,P)) $ den korrespondierenden Morphismus in $ \Hom(A\otimes_R M,P) $ ebenfalls mit $ \hat{g} $. Es gilt also $ \hat{\hat{f}}=f $.
\end{Bem+Def}

\begin{Satz}\label{Komodul}
Sei $ R $ ein unitärer, kommutativer Ring und $ M $ ein $ R $-Modul. Sei $ A $ eine assoziative, unitäre, endliche $ R $-Algebra, welche als Modul projektiv ist. Dann gibt es eine kanonische 1:1-Beziehung zwischen $ A $-Links-Modulstrukturen und $ A^\vee $-Links-Komodulstrukturen auf $ M $.
\end{Satz}

\begin{proof}

Eine $ A $-Modulstruktur auf $ M $ ist gegeben durch eine assoziative, $ R $-bilineare Abbildung $ m:A\otimes_R M\rightarrow M $, also ein kommutatives Diagramm
\begin{center}
\begin{tikzpicture}[description/.style={fill=white,inner sep=2pt}]
    \matrix (m) [matrix of math nodes, row sep=3.5em,
    column sep=5.5em, text height=1.5ex, text depth=0.25ex]
    { M &  A\otimes_R M \\
    A\otimes_R M & A\otimes_R A\otimes_R M    \\
	};
    \path[->,font=\scriptsize]
    (m-1-2) edge node[auto] {$ m $} (m-1-1)
    (m-2-1) edge node[auto] {$ m $} (m-1-1)
    (m-2-2) edge node[auto] {$ id\otimes m $} (m-1-2)
    (m-2-2) edge node[auto] {$ \mu\otimes id$} (m-2-1)
    ;   
\end{tikzpicture}
\end{center}

Wir möchten nachrechnen, dass folgendes Diagramm kommutiert:

\begin{center}
\begin{tikzpicture}[description/.style={fill=white,inner sep=2pt}]
    \matrix (m) [matrix of math nodes, row sep=3.5em,
    column sep=5.5em, text height=1.5ex, text depth=0.25ex]
    { M &  \Hom_R(A,M) & \\
    \Hom_R(A,M) & \Hom_R(A\otimes_R A,M) & \Hom_R(A,\Hom_R(A,M))  \\
	};
    \path[->,font=\scriptsize]
    (m-1-1) edge node[auto] {$ \hat{m} $} (m-1-2)
    (m-1-1) edge node[auto] {$ \hat{m} $} (m-2-1)
    (m-2-1) edge node[auto] {$ (\_\circ\mu) $} (m-2-2)
    (m-1-2) edge node[auto] {$ (\hat{m}\circ\_) $} (m-2-3)
    (m-2-2) edge node[auto] {$ \simeq $} (m-2-3)
    ;   
\end{tikzpicture}
\end{center}

wobei der Isomorphismus $ \simeq $ die Adjunktion beschreibt. \\
Sei also $ x\in M $. Dann ist
\[ (\hat{m}\circ\_)\circ\hat{m}(x)=(\hat{m}\circ\_)(Abb:a\mapsto m(x\otimes a))=Abb:a\mapsto (Abb:b\mapsto m(m(x\otimes a)\otimes b). \]
Dies entspricht unter der Adjunktion der Abbildung 
\[ a\otimes b\mapsto m(m(x\otimes a)\otimes b) \]
in $ \Hom_R(A\otimes_R A,M) $. Da Algebrenmultiplikation $ \mu $ mit Skalarmultiplikation $ m $ vertauscht, entspricht dies der Abbildung  $ a\otimes b\mapsto m(\mu(a\otimes b)\otimes x) $ in $ \Hom_R(A\otimes_R A,M) $. Dies ist, wie gefordert, die Abbildung $ \hat{m}((\_\circ\mu)(x)) $.\\ 
Da $ A $ projektiv ist, erhalten wir unter Verwendung des natürlichen Isomorphismus \\ $ A^\vee\otimes_R M\cong\Hom_R(A,M) $, sowie der Identifikation aus Lemma \ref{natuerlich} ein kommutatives Diagramm

\begin{center}
\begin{tikzpicture}[description/.style={fill=white,inner sep=2pt}]
    \matrix (m) [matrix of math nodes, row sep=3.5em,
    column sep=5.5em, text height=1.5ex, text depth=0.25ex]
    { M &  A^\vee\otimes_R M & \\
    A^\vee\otimes_R M & A^\vee\otimes_R A^\vee \otimes_R M  \\
	};
    \path[->,font=\scriptsize]
    (m-1-1) edge node[auto] {$ \hat{m} $} (m-1-2)
    (m-1-1) edge node[auto] {$ \hat{m} $} (m-2-1)
    (m-2-1) edge node[auto] {$ \mu^*\otimes id $} (m-2-2)
    (m-1-2) edge node[auto] {$ id\otimes \hat{m} $} (m-2-2)
    ;   
\end{tikzpicture}
\end{center}

Für den $ A $-Modul $ M $ kommutiert außerdem 

\begin{center}
\begin{tikzpicture}[description/.style={fill=white,inner sep=2pt}]
    \matrix (m) [matrix of math nodes, row sep=3.5em,
    column sep=5.5em, text height=1.5ex, text depth=0.25ex]
    { M &  A\otimes_R M \\
     & R\otimes_R M    \\
	};
    \path[->,font=\scriptsize]
    (m-1-2) edge node[auto] {$ m $} (m-1-1)
    (m-2-2) edge node[description] {$ Skalarmult. $} (m-1-1)
    (m-2-2) edge node[description] {$ id\otimes\mu $} (m-1-2)
    ;   
\end{tikzpicture}
\end{center}

Daraus berechnen wir analog die Kommutativität von

\begin{center}
\begin{tikzpicture}[description/.style={fill=white,inner sep=2pt}]
    \matrix (m) [matrix of math nodes, row sep=3.5em,
    column sep=5.5em, text height=1.5ex, text depth=0.25ex]
    { M & \Hom(A,M) \\
     & \Hom(R,M)    \\
	};
    \path[->,font=\scriptsize]
    (m-1-1) edge node[auto] {$ \hat{m} $} (m-1-2)
    (m-2-2) edge node[auto] {$ \sim $} (m-1-1)
    (m-2-2) edge node[auto] {$ \circ f $} (m-1-2)
    ;   
\end{tikzpicture}
\end{center}
 
beziehungsweise

\begin{center}
\begin{tikzpicture}[description/.style={fill=white,inner sep=2pt}]
    \matrix (m) [matrix of math nodes, row sep=3.5em,
    column sep=5.5em, text height=1.5ex, text depth=0.25ex]
    { M &  M\otimes_R A^\vee \\
     & M\otimes_R R\cong M\otimes_R R^\vee    \\
	};
    \path[->,font=\scriptsize]
    (m-1-1) edge node[auto] {$ \hat{m} $} (m-1-2)
    (m-2-2) edge node[auto] {$ \sim $} (m-1-1)
    (m-2-2) edge node[auto] {$ id\otimes f $} (m-1-2)
    ;   
\end{tikzpicture}
\end{center}

Ein $ A $-Modul induziert also eine $ A^\vee $-Komodulstruktur. \\
Sei umgekehrt $ M $ ein $ A^\vee $-Komodul.
Dann wird durch gleiche Argumentation und unter Verwendung der Tatsache, dass $ A $ reflexiv ist, aus einem Diagramm

\begin{center}
\begin{tikzpicture}[description/.style={fill=white,inner sep=2pt}]
    \matrix (m) [matrix of math nodes, row sep=3.5em,
    column sep=5.5em, text height=1.5ex, text depth=0.25ex]
    { M &  A^\vee\otimes_R M & \\
    A^\vee\otimes_R M & A^\vee\otimes_R A^\vee \otimes_R M  \\
	};
    \path[->,font=\scriptsize]
    (m-1-1) edge node[auto] {$ \rho $} (m-1-2)
    (m-1-1) edge node[auto] {$ \rho $} (m-2-1)
    (m-2-1) edge node[auto] {$ \Delta\otimes id $} (m-2-2)
    (m-1-2) edge node[auto] {$ id\otimes\rho$} (m-2-2)
    ;   
\end{tikzpicture}
\end{center}

ein Diagramm der Form

\begin{center}
\begin{tikzpicture}[description/.style={fill=white,inner sep=2pt}]
    \matrix (m) [matrix of math nodes, row sep=3.5em,
    column sep=5.5em, text height=1.5ex, text depth=0.25ex]
    { M &  A\otimes_R M & \\
    A\otimes_R M & A\otimes_R A\otimes_R M  \\
	};
    \path[->,font=\scriptsize]
    (m-1-1) edge node[auto] {$ \hat{\rho} $} (m-1-2)
    (m-1-1) edge node[auto] {$ \hat{\rho} $} (m-2-1)
    (m-2-1) edge node[auto] {$ \Delta^*\otimes id $} (m-2-2)
    (m-1-2) edge node[auto] {$ id\otimes\hat{\rho}$} (m-2-2)
    ;   
\end{tikzpicture}
\end{center}

Die Konstruktionen sind invers zueinander, da $ \hat{\hat{m}}=m $ und $ (\mu^*)^*\cong\mu $ ist. Also gibt es eine kanonische 1:1-Beziehung zwischen $ A $-Modulstrukturen und $ A^\vee $-Komodulstrukturen auf $ M $.

\end{proof}

\begin{Bem}
Wir hätten gleichermaßen zeigen können, dass es eine 1:1 Beziehung zwischen Links-Modulstrukturen und Rechts-Komodulstrukturen gibt. 
Tatsächlich steckt in den Isomorphismen
\[ \Hom(A\otimes_R A,M)\xrightarrow{{\big (\rho^{(A\otimes A)}_M\big )}^{-1}}(A\otimes_R A)^\vee\otimes_R M\xrightarrow{{\big (\rho^A_{A^\vee}\big )}^{-1}\otimes id}A^\vee\otimes A^\vee\otimes M \]
eine gewisse Wahl, da wir die Abbildung $ \rho^A_N $ willkürlich durch
\[
\begin{array}{cccc}
\rho^A_N: & A^\vee\otimes_R N & \rightarrow & \Hom(A,N) \\
 & \varphi\otimes n & \mapsto & (m\mapsto\varphi(m)\cdot n) \\
\end{array}
\]
definiert haben. Definiert man $ \rho^A_N $ stattdessen als
\[
\begin{array}{cccc}
\rho^A_N: & N\otimes_R A^\vee & \rightarrow & \Hom(A,N) \\
 & n\otimes \varphi & \mapsto & (m\mapsto\varphi(m)\cdot n) \\
\end{array}
\]
so erhalten wir durch 

\[
\begin{array}{ccccccc}
B^\vee\otimes_R A^\vee & \xrightarrow{\rho^A_{B^\vee}}& \Hom(A,B^\vee) & \xrightarrow{\sim} & \Hom(A\otimes B,R) & =(A\otimes_R B)^\vee\\
\beta\otimes\alpha     &   \mapsto    & (a\mapsto\alpha(a)\cdot\beta) &     \mapsto        & (a\otimes b\mapsto \alpha(a)\cdot\beta(b))\\
\end{array}
\]
auf $ A^\vee $ genau die opponierte Komultiplikation zu der aus Lemma \ref{natuerlich}. Weiterhin erhalten wir mit dieser Definition von $ \rho^A_N $ die Isomorphismen

\[ \Hom(A\otimes_R A,M)\xrightarrow{{\big (\rho^{(A\otimes A)}_M\big )}^{-1}}M\otimes_R (A\otimes_R A)^\vee\xrightarrow{id\otimes{\big (\rho^A_{A^\vee}\big )}^{-1}}M\otimes A^\vee\otimes A^\vee. \]
Führt man den Beweis mit diesen Isomorphismen durch, so erhält man für jede $ A $-Linksmodulstruktur ein kommutatives Diagramm

\begin{center}
\begin{tikzpicture}[description/.style={fill=white,inner sep=2pt}]
    \matrix (m) [matrix of math nodes, row sep=3.5em,
    column sep=5.5em, text height=1.5ex, text depth=0.25ex]
    { M &  M\otimes_R A^\vee & \\
    M\otimes_R A^\vee & M \otimes_R A^\vee\otimes_R A^\vee  \\
	};
    \path[->,font=\scriptsize]
    (m-1-1) edge node[auto] {$ \hat{m} $} (m-1-2)
    (m-1-1) edge node[auto] {$ \hat{m} $} (m-2-1)
    (m-2-1) edge node[auto] {$ id\otimes \mu^* $} (m-2-2)
    (m-1-2) edge node[auto] {$ \hat{m} \otimes id $} (m-2-2)
    ;   
\end{tikzpicture}
\end{center}
Dies induziert auf $ M $ eine Struktur als Rechts-Komodul über der (nun opponierten) Koalgebra $ A^\vee $.
\end{Bem}

\begin{Lemma}\label{morphismen}
Sei $ A $ eine endliche $ R $-Algebra und als Modul projektiv. Seien $ M $ und $ N $ endlich-dimensionale $ A $-Moduln. Ein $ A $-Modulhomomorphismus $ f:M\rightarrow N $ induziert einen $ A^\vee $-Komodulhomomorphismus und umgekehrt. 
\end{Lemma}

\begin{proof}
Sei $ \mu:A\otimes_R M\rightarrow M $ die Multiplikationsabbildung als $ A $-Modul. 
Ein $ A $-Modulhomomorphismus erfüllt $ f\circ\mu=\mu\circ(id\otimes f) $, macht also folgendes Diagramm kommutativ.

\begin{center}
\begin{tikzpicture}[description/.style={fill=white,inner sep=2pt}]
    \matrix (m) [matrix of math nodes, row sep=3.5em,
    column sep=5.5em, text height=1.5ex, text depth=0.25ex]
    { M &  N \\
    A\otimes_R M & A\otimes_R N    \\
	};
    \path[->,font=\scriptsize]
    (m-1-1) edge node[auto] {$ f $} (m-1-2)
    (m-2-1) edge node[auto] {$ id\otimes f $} (m-2-2)
    (m-2-1) edge node[auto] {$ \mu $} (m-1-1)
    (m-2-2) edge node[auto] {$ \mu $} (m-1-2)
    ;   
\end{tikzpicture}
\end{center}

 Wir müssen überprüfen, ob dies genau dann der Fall ist, wenn auch
 
\begin{center}
\begin{tikzpicture}[description/.style={fill=white,inner sep=2pt}]
    \matrix (m) [matrix of math nodes, row sep=3.5em,
    column sep=5.5em, text height=1.5ex, text depth=0.25ex]
    {  A^\vee\otimes_R M &  A^\vee\otimes_R N \\
     M &  N    \\
	};
    \path[->,font=\scriptsize]
    (m-1-1) edge node[auto] {$ id\otimes f $} (m-1-2)
    (m-2-1) edge node[auto] {$  f $} (m-2-2)
    (m-2-1) edge node[auto] {$ \hat{\mu} $} (m-1-1)
    (m-2-2) edge node[auto] {$ \hat{\mu} $} (m-1-2)
    ;   
\end{tikzpicture}
\end{center}

kommutiert. Dafür schreiben wir das letzte Diagramm zunächst mit Hilfe des Isomorphismus aus Satz \ref{projektiv} und den korrespondierenden Morphismen aus Lemma \ref{natuerlich} um und zeigen äquivalent, dass dies genau dann der Fall ist, wenn

\begin{center}
\begin{tikzpicture}[description/.style={fill=white,inner sep=2pt}]
    \matrix (m) [matrix of math nodes, row sep=3.5em,
    column sep=5.5em, text height=1.5ex, text depth=0.25ex]
    { \Hom_R(A,M) &  \Hom_R(A,N) \\
     M &  N    \\
	};
    \path[->,font=\scriptsize]
    (m-1-1) edge node[auto] {$ f\circ $} (m-1-2)
    (m-2-1) edge node[auto] {$  f $} (m-2-2)
    (m-2-1) edge node[auto] {$ \hat{\mu} $} (m-1-1)
    (m-2-2) edge node[auto] {$ \hat{\mu} $} (m-1-2)
    ;   
\end{tikzpicture}
\end{center}

kommutiert.\\
Sei $ x\in M $. Dann gilt
\[(f\circ\_)\circ\hat{\mu}(x)=(f\circ\_)\big (Abb:a\mapsto \mu(x\otimes a) \big)= \big (Abb:a\mapsto f(\mu(x\otimes a)\big ) \]
\[\hat{\mu}\circ f (x)=\big( Abb:a\mapsto \mu(f(x)\otimes a) \big). \]
Das letzte Diagramm kommutiert also genau dann, wenn für alle $ x\in M $ gilt, dass 
\[ \mu(f(x)\otimes a)=f(\mu(x\otimes a) \]
ist. Dies entspricht genau der Bedingung, dass das erste Diagramm kommutiert. 

\end{proof}

Aus Satz \ref{Komodul} und Lemma \ref{morphismen} erhalten wir zusammen direkt:

\begin{Korollar}\label{Mod_Komod}
Sei $ A $ eine endliche $ R $-Algebra und als Modul projektiv. Dann gibt es eine kanonische Äquivalenz von Kategorien zwischen der Kategorie der endlich erzeugten $ A $-Moduln und der Kategorie der endlich erzeugten $ A^\vee $-Komoduln. 
\end{Korollar}

\section{Die Diagrammkategorie als Kategorie von Komoduln}

Wir kommen nun auf die Frage zurück, unter welchen Bedingungen die Diagrammkategorie zu einer Darstellung eine Kategorie von Komoduln über einer Koalgebra ist. In Korollar \ref{Koalgebra} haben wir gesehen, dass zumindest eine endliche $ R $-Algebra $ A $, die als Modul projektiv ist, auf kanonische Weise eine Koalgebra $ A^\vee $ induzieren. Außerdem gilt die Korrespondenz zwischen $ A $-Modul- und $ A^\vee $-Komodulstrukturen auf einem $ R $-Modul nach  Satz \ref{Komodul} ebenfalls nur, wenn $ A $ als Modul projektiv ist. Um einen $ \End(T) $-Modul aus der Diagrammkategorie einer Darstellung $ T $ als \\$ \End(T)^\vee $-Komodul auffassen zu können, müssen wir also eine Bedingung an den Ring $ R $ stellen, welche uns garantiert, dass  $ \End(T) $ projektiv ist. \\
Moduln über Hauptidealringen haben folgende schöne Eigenschaften:

\begin{Satz}\label{frei}
Jeder Untermodul eines freien Moduls über einem Hauptidealring ist wieder frei.
\end{Satz}

\begin{proof}
\cite[Thm7.1]{MR1878556}
\end{proof}

\begin{Satz}
Ein Modul über einem Hauptidealring ist genau dann frei, wenn er projektiv ist.
\end{Satz}

\begin{proof}
\cite[Ch. VII, § 3]{MR0026989}
\end{proof}

Bevor wir zeigen können, dass die Diagrammkategorie eine Kategorie von Komoduln ist, brauchen wir noch ein technisches Lemma. 

\begin{Lemma}\label{rtg}
\begin{enumerate}
\item Die Kategorie der Koalgebren über einem kommutativen Ring $ R $ ist kovollständig.
\item Seien $ A,B $ endliche Koalgebren über $ R $, die als Moduln projektiv sind. Ein Morphismus von Koalgebren $ A\xrightarrow{f}B $ induziert auf natürliche Weise einen Funktor $ A\Komod\xrightarrow{\mathcal{F}}B\Komod $.
\item Sei $ \{(A_i)_{i\in I},(f_{ij})_{i,j\in I})\} $ ein gerichtetes System von $ R $-Koalgebren. Dann gilt für das induzierte direkte System $ \{(A_i\Komod)_{i\in I},(\mathcal{F}_{ij})_{i,j\in I})\} $:
\[ (\varinjlim_I A_i)\Komod\cong\varinjlim_I (A_i\Komod)\]
\end{enumerate}
\end{Lemma}

\begin{proof}
zu (1):\\
Die Kategorie der $ R $-Moduln ist kovollständig. Wir müssen prüfen, dass sich die komultiplikative Struktur auf das Limesobjekt überträgt. Sei $ \{(A_i)_{i\in I},(f_{ij})_{i,j\in I})\} $ ein gerichtetes System von $ R $-Koalgebren. Die Abbildung $ A_k\rightarrow \varinjlim A_i $ und die Komultiplikation $ \Delta $ definieren für alle $ k $ eine Abbildung
\[ A_k\overset{\Delta}{\longrightarrow}A_k\otimes_R A_k\longrightarrow \varinjlim A_i\otimes \varinjlim A_i. \]
Für alle $ f_{jk}:A_k\longrightarrow A_j $ kommutiert 

\begin{center}
\begin{tikzpicture}[description/.style={fill=white,inner sep=2pt}]
    \matrix (m) [matrix of math nodes, row sep=4.8em,
    column sep=3.5em, text height=1.5ex, text depth=0.25ex]
    { 
     A_k & & A_j \\
    A_k\otimes_R A_k & &  A_j\otimes_R A_j \\
      &  \varinjlim A_i\otimes \varinjlim A_i  & \\
	};
    \path[->,font=\scriptsize]
    (m-2-1) edge node[auto] {$ f_{jk}\otimes f_{jk} $} (m-2-3)
    (m-2-1) edge node[auto] {$   $} (m-3-2)
    (m-2-3) edge node[auto] {$ $} (m-3-2)
    (m-1-1) edge node[auto] {$ f_{jk} $} (m-1-3)
    (m-1-1) edge node[auto] {$ \Delta_{A_k} $} (m-2-1)
    (m-1-3) edge node[auto] {$ \Delta_{A_j} $} (m-2-3)
    ;   
\end{tikzpicture}
\end{center}
da das Tensorprodukt rechtsexakt ist, also mit allen direkten Limiten vertauscht. 
Dies induziert nach universeller Eigenschaft des direkten Limes eine eindeutige, mit allen Komultiplikationen verträgliche Abbildung
\[ \varinjlim A_i \longrightarrow \varinjlim A_i\otimes \varinjlim A_i, \]
welche dem Limesobjekt die Struktur einer Koalgebra gibt. Wir wissen, dass die universelle Eigenschaft auf Vektorraumebene erfüllt ist und man prüft leicht, dass die entsprechenden induzierten Abbildungen aus dem Limesobjekt tatsächlich Morphismen von Koalgebren sind. \\
zu (2):\\ Der Koalgebrenhomomorphismus $ f:A\rightarrow B $ induziert einen Algebrenhomomorphismus \\
\[ f^*:B^\vee\rightarrow A^\vee. \] \\
Die Abbildung $ f^* $ induziert via Restriktion der Skalare einen Funktor
\[ A^\vee\Mod\rightarrow B^\vee \Mod. \] 
Da $ A $ und $ B $ reflexiv sind, entspricht dies nach Korollar \ref{Mod_Komod} auf natürliche Weise einem Funktor\\
\[ A\Komod\rightarrow B\Komod. \]
zu (3): \\ 
Nach (2) erhalten wir für $ \{(A_i)_{i\in I},(f_{ij})_{i,j\in I})\} $ ein gerichtetes System \\ $ \{((A_i)\Komod)_{i\in I},(\mathcal{F}_{ij})_{i,j\in I})\} $. Der direkte Limes existiert, da der entsprechende direkte Limes von Moduln über den dualen Algebren existiert. Da die Kategorie der $ R $-Komoduln kovollständig ist, haben wir eine Algebra $ \varinjlim_I (A_i) $ zusammen mit Koalgebrenmorphismen $ \varphi_j:A_j\rightarrow \varinjlim_I (A_i) $ für alle $ j\in I $. Dies induziert ein kommutatives Diagramm

\begin{center}
\begin{tikzpicture}[description/.style={fill=white,inner sep=2pt}]
    \matrix (m) [matrix of math nodes, row sep=3.5em,
    column sep=3.5em, text height=1.5ex, text depth=0.25ex]
    { \cdots A_j\Komod & & A_k\Komod \cdots \\
    	      & (\varinjlim_I A_i)\Komod &   \\
	};
    \path[->,font=\scriptsize]
    (m-1-1) edge node[description] {$ \mathcal{F}_{kj} $} (m-1-3)
    (m-1-1) edge node[description] {$ \phi_{j} $} (m-2-2)
    (m-1-3) edge node[description] {$ \phi_{k} $} (m-2-2)
    ;   
\end{tikzpicture}
\end{center}
 
und nach universeller Eigenschaft des direkten Limes gibt es einen eindeutigen Funktor
\[ u: \varinjlim_I (A_i\Komod)\rightarrow \ (\varinjlim_I A_i)\Komod,  \]
so dass folgendes Diagramm Kommutiert.

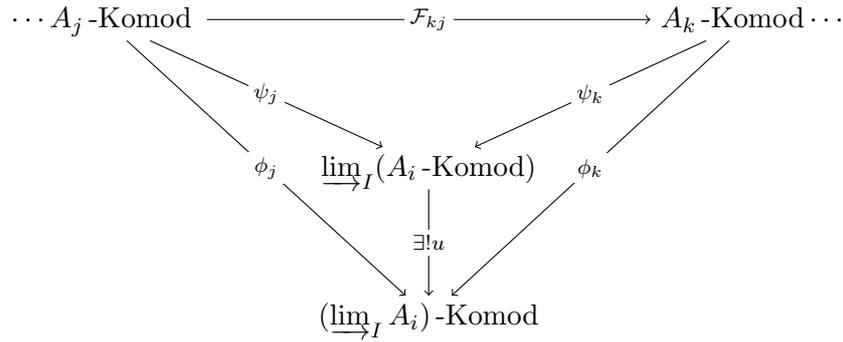
\begin{center}
\begin{tikzpicture}[description/.style={fill=white,inner sep=2pt}]
    \matrix (m) [matrix of math nodes, row sep=3.7em,
    column sep=3.5em, text height=1.5ex, text depth=0.25ex]
    { \cdots A_j\Komod & & A_k\Komod \cdots \\
    	      & \varinjlim_I (A_i\Komod) &   \\
	      & (\varinjlim_I A_i)\Komod &   \\
	};
    \path[->,font=\scriptsize]
    (m-1-1) edge node[description] {$ \mathcal{F}_{kj} $} (m-1-3)
    (m-1-1) edge node[description] {$ \psi_{j} $} (m-2-2)
    (m-1-1) edge node[description] {$ \phi_{j} $} (m-3-2)
    (m-1-3) edge node[description] {$ \psi_{k} $} (m-2-2)
    (m-1-3) edge node[description] {$ \phi_{k} $} (m-3-2)
    (m-2-2) edge node[description] {$\exists ! u $} (m-3-2)
    ;   
\end{tikzpicture}
\end{center}
 
Wir konstruieren die Umkehrabbildung zu $ u $:\\
Bezeichne ab sofort $A=\varinjlim_I A_i$ den direkten Limes der Koalgebren. Ein endlich erzeugter $ A $-Komodul $ M $ ist definiert durch eine Abbildung 
\[ m:\rightarrow M\otimes_R A. \] 
Sei $ (x_1,..,x_n) $ ein $ R $-Erzeugendensystem von $ M $. Dann ist $ m(x_i)=\sum_{k=1}^n a_{ki} \otimes x_{k} $ für geeignete Elemente $ a_{ki}\in A  $. Jedes Element $ a_{ki} $ liegt schon in einer $ R $-Koalgebra $ A_i $ des Systems und da das System von Koalgebren gerichtet ist, finden wir eine Koalgebra $ A_l $ so, dass für alle $ k,i $ die endlich vielen Elemente $ a_{ki} $ in $ A_l $ liegen. Da $ (x_1,..,x_n) $ den Komodul $ M $ erzeugen, induziert $ m $ eine Komultiplikation
\[\tilde{m}: M \otimes_R A_l. \]
$ M $ ist also auf natürliche Weise ein $ A_l $-Komodul, daher via $ \psi_l $ auch ein Objekt aus $ \varinjlim_I (A_i\Komod) $. 

\end{proof}

\begin{Satz}\label{Dualitaet}
Sei $ R $ ein Hauptidealring. Bezeichne $ R\Frei $ die Kategorie der endlich erzeugten, freien $ R $-Moduln. Sei $ D $ ein Diagramm und $ T:D\longrightarrow R\Frei $ eine Darstellung. Sei $ A(T):=\varinjlim_{F\subset D}\End(T_F)^\vee $ wobei $ F\subset D $ endliche Teilmengen seien. Dann ist $ A(T) $ auf natürliche Weise eine $ R $-Koalgebra und die Diagrammkategorie $ \C(T) $ ist natürlich äquivalent zur Kategorie der $ A(T) $-Komoduln. 
\end{Satz}

\begin{proof}
Sei zunächst wieder $ D $ endlich. Die Diagrammkategorie $ \C(T) $ ist die Kategorie der endlich erzeugten $ \End(T) $-Linksmoduln. Da $ T $ eine Darstellung in die freien Moduln ist, ist $ \End(T)\subset \prod_{p \in Ob(D)}\End(T_p) $ Untermodul eines endlich erzeugten freien Moduls, also nach Satz \ref{frei} selbst frei. Da Hauptidealringe noethersch sind, ist $ \End(T) $ außerdem als Modul endlich erzeugt. Freie Moduln sind projektiv, also ist $ \End(T)^\vee $ nach Korollar \ref{Koalgebra} eine Koalgebra und nach Satz \ref{Komodul} haben wir eine natürliche Äquivalenz zwischen $ \End(T) $-Linksmoduln und $ \End(T)^\vee $ Rechts-Komoduln.\\
Für beliebige Diagramme $ D $ gilt
\[
  \C(T):  =  \varinjlim (\End(T_F)\Mod) =\varinjlim(\End(T_F)^\vee\Komod)  \overset{\ref{rtg}}{=} (\varinjlim\End(T_F)^\vee)\Komod. 
\]
$ \C(T) $ ist also die Kategorie der endlich erzeugten $ \varinjlim_{F\subset D\text{ endlich}}\End(T_F)^\vee $-Komoduln.
\end{proof}

\chapter{Ausblick: Tannaka-Dualität}
\label{ch:kap_6}

Erinnern wir uns an die Konstruktion der Diagrammkategorie:
Gegeben ein Diagramm  $ D $ und eine Darstellung 
\[ D\overset{T}{\longrightarrow}R\Mod \]
in die endlich erzeugten $ R $-Moduln, so existiert eine universelle abelsche Kategorie $ \C(T) $, so dass $ T $ über 
\[ D\overset{\tilde{T}}{\longrightarrow}\C(T)\overset{ff_T}{\longrightarrow}R\Mod \]
faktorisiert. Der schwierige Teil des Beweises war Noris Satz, dass im Fall $ D=\mathcal{A} $ einer $ R $-linearen abelschen Kategorie und $ T $ einem treuen exakten Funktor die Abbildung 
\[ \tilde{T}:\mathcal{A}\longrightarrow\C(T) \]
schon eine Äquivalenz von Kategorien ist (Kapitel 3). \\
Für den Fall, dass $ R $ ein Körper ist, ist diese Aussage nicht neu. Sie findet sich schon bei Saavedra \cite{MR0338002} und in ähnlicher Form bei Deligne \cite{MR654325} und entspringt einem komplexeren Setting,  nämlich der Dualität zwischen neutralen Tannaka-Kategorien mit Faserfunktoren und den endlich-dimensionalen Darstellungen affiner Gruppenschemata. Diese Dualität nennt man Tannaka-Dualität.
Ziel dieses Kapitels ist es, diese Dualität für Körper präzise zu formulieren. In Kapitel 7 werden wir dann den Zusammenhang zur Kategorie der Nori-Motive klären.\\
Zentral für die Dualität ist die Tatsache, dass die Diagrammkategorie einer Darstellung in Vektorräume eine Kategorie von Komoduln über einer Koalgebra ist. Wir haben in Kapitel 5 gesehen, dass dies auch für Darstellungen gilt, welche Werte in freien endlich erzeugten Moduln über Hauptidealringen annimmt. Tannaka-Dualität sollte sich daher auf diese Fälle verallgemeinern lassen. Ein noch allgemeineres Resultat findet man bei Wedhorn \cite{MR2101076}, der Tannaka-Dualität für den Fall zeigt, dass $ R $ ein Dedekindring ist.

\section{Starre Tensorkategorien}
Wir beginnen nun mit einigen Definitionen über Tensorkategorien. Eine ausführlichere Beschreibung findet man in \cite{MR654325} .

\begin{Def}
Eine Tensorkategorie ist eine Kategorie $ \C $ zusammen mit einem Funktor
\[ 
\begin{array}{cccc}
\otimes: & \C \times \C & \longrightarrow & \C  \\
& (X,Y) & \mapsto & X\otimes Y \\
\end{array}
\]
mit folgenden Daten:
\begin{itemize}
\item \emph{Kommutativität}: Für alle Objekte $ X,Y \in \C $ einem Isomorphismus \\$ \Psi_{X,Y}:X\otimes Y\overset{\sim}{\longrightarrow} Y \otimes X $.
\item \emph{Assoziativität}: Für alle Objekte $ X,Y,Z \in \C $ einem Isomorphismus \\$ \Phi_{X,Y,Z}:(X\otimes Y)\otimes Z \overset{\sim}{\longrightarrow} X \otimes ( Y \otimes Z) $.
\item Einem Objekt $ U\in\C $, genannt \emph{Identitätsobjekt}, zusammen mit einem Isomorphismus 
\[ u:U\longrightarrow U\otimes U, \]
so dass die Abbildung $ X\mapsto U\otimes X $ für alle $ X\in \C  $ eine Äquivalenz von Kategorien ist. Das Identitätsobjekt wird auch oft mit $ \underline{1} $ bezeichnet.
\end{itemize}
\end{Def}

\begin{Def}\label{hoho}
Sei $ (\C,\otimes) $ eine Tensorkategorie. Ist der kontravariante Funktor
\[
\begin{array}{ccc}
\C & \longrightarrow & Set \\
T & \mapsto & \Hom_{\C}(T\otimes X,Y) \\
\end{array}
\]
darstellbar, so nennen wir sein darstellendes Objekt $ \Hom(X,Y) $. Wir haben dann also eine Bijektion 
\[ \Hom_{\C}(T,\Hom(X,Y))\overset{1:1}{\longleftrightarrow}\Hom_{\C}(T\otimes X,Y). \]
Den zu $ id_{\Hom(X,Y)} $ korrespondierenden Morphismus nennen wir
\[ \ev_{X,Y}:\Hom(X,Y)\otimes X\longrightarrow Y. \]
\end{Def}

\begin{Def}\label{hf}
Sei $ (\C,\otimes) $ eine Tensorkategorie. Dann definieren wir das zu $ X $ \emph{duale Objekt} durch $ X^\vee:=\Hom(X,\underline{1}) $, wobei $ \underline{1} $ das Idententitätsobjekt ist.\\
Die Bijektion
\[ \Hom_{\C}(X\otimes X^\vee,\underline{1})\overset{1:1}{\longleftrightarrow}\Hom_{\C}(X,{X^\vee}^\vee) \]
liefert eine zur Evaluationsabbildung
\[ \ev_{X}:X\otimes X^\vee\longrightarrow\underline{1} \]
korrespondierende Abbildung
\[ i_{X}:X\longrightarrow{X^\vee}^\vee. \]
Ist diese Abbildung ein Isomorphismus, so sagen wir das Objekt $ X $ ist \emph{reflexiv}.
\end{Def}

\begin{Bem}\label{vb}
Sei $ I $ eine endliche Indexmenge. Dann korrespondiert der natürliche Morphismus
\[ \big (\otimes_{i\in I}\Hom(X_i,Y_i)\big )\otimes \big (\otimes_{i\in I}X_i\big )\overset{\sim}{\longrightarrow}\otimes_{i\in I}\big (\Hom(X_i,Y_i)\otimes X_i\big )\overset{ev}{\longrightarrow}\otimes_{i\in I}Y_i \]
nach Definition \ref{hf} zu einem Morphismus
\[ \otimes_{i\in I}\Hom(X_i,Y_i)\longrightarrow\Hom(\otimes_{i\in I} Y_i). \]
\end{Bem}

\begin{Def}
Eine Tensorkategorie $ (\C,\otimes) $ heißt \emph{starr}, wenn
\begin{itemize}
\item $ \Hom(X,Y) $ für alle Objekte $ X,Y\in\C $ existiert. 
\item Alle Objekte reflexiv sind. 
\item Die natürliche Abbildung 
\[ \Hom(X_1,Y_1)\otimes\Hom(X_2,Y_2)\longrightarrow\Hom(X_1\otimes X_2,Y_1\otimes Y_2) \]
aus Bemerkung \ref{vb} für alle $ X_1 , X_2 , Y_1 , Y_2\in \C $ ein Isomorphismus ist.
\end{itemize}
\end{Def}

\begin{Def}
Unter einer \emph{abelschen Tensorkategorie} verstehen wir eine Tensorkategorie, die abelsch ist, so dass der Funktor $ \otimes $ biadditiv ist. 
\end{Def}

\begin{Bsp}
Die Kategorie der endlich-dimensionalen $ k $-Vektorräume mit gewöhnlichem Tensorprodukt ist eine starre Tensorkategorie. Die Kategorie der endlich erzeugten $ R $-Moduln über einem Ring $ R $ jedoch im Allgemeinen nicht, da zum Beispiel für $ R=\mathbb{Z} $ Moduln mit Torsion nicht reflexiv sind.
\end{Bsp}

\begin{Bem}
In einer starren Tensorkategorie, die abelsch ist, ist das Tensorprodukt immer biadditiv und exakt. Es vertauscht ganz allgemein mit Limiten und Kolimiten (siehe \cite[S. 118]{MR654325}).
\end{Bem}

\begin{Def}
Seien $ (\C,\otimes) $ und $ (\C',\otimes') $ Tensorkategorien. Ein \emph{Tensorfunktor} zwischen $ \C $ und $ \C' $ besteht aus einem Funktor \[ \mathcal{F}:\C\longrightarrow\C', \]
sowie einem funktoriellen Isomorphismus 
\[c_{X,Y}:\mathcal{F}(X)\otimes\mathcal{F}(Y)\longrightarrow\mathcal{F}(X\otimes Y), \] 
sodass $ c $ verträglich ist mit Assoziativität, Kommutativität und dem Identitätsobjekt, sowie den Identitätsmorphismus erhält. 
\end{Def}

\begin{Def}
Seien $ (\mathcal{F},c) $ und $ (\mathcal{G},d) $ Tensorfunktoren zwischen Tensorkategorien $ \C\rightarrow\C ' $. Eine \emph{natürliche Transformation von Tensorfunktoren} zwischen $ \mathcal{F} $ und $ \mathcal{G} $ ist eine natürliche Transformation von Funktoren $ \tau:\mathcal{F}\longrightarrow\mathcal{G} $, sodass das Diagramm

\begin{center}
\begin{tikzpicture}[description/.style={fill=white,inner sep=2pt}]
    \matrix (m) [matrix of math nodes, row sep=3.0em,
    column sep=5.0em, text height=1.5ex, text depth=0.25ex]
    { \otimes_{i\in I}\mathcal{F}(X_i)  & \mathcal{F}(\otimes_{i\in I} X_i)  \\
	  \otimes_{i\in I}\mathcal{G}(X_i)  & \mathcal{G}(\otimes_{i\in I} X_i) \\
	};
    \path[->,font=\scriptsize]
    (m-1-1) edge node[auto] {$ c $} (m-1-2)
    (m-2-1) edge node[auto] {$ d $} (m-2-2)
    (m-1-1) edge node[auto] {$ \otimes\tau_{X_i} $} (m-2-1)
    (m-1-2) edge node[auto] {$ \tau_{\otimes_{X_i}} $} (m-2-2)
    ;   
\end{tikzpicture}
\end{center}

für endliche Familien $ (X_i)_{i\in I} $ von Objekten aus $ \C $ kommutiert.\\
Wir bezeichnen mit $ \underline{\Hom}(F,G) $ die Menge aller Transformationen zwischen $ \mathcal{F} $ und $ \mathcal{G} $ und mit $ \underline{\Hom}^{\otimes}(F,G) $ die Menge aller Transformationen von Tensorfunktoren zwischen $ \mathcal{F} $ und $ \mathcal{G} $.
\end{Def}

\section{Affine Gruppenschemata und Tannaka-Dualität}

Seien in diesem Abschnitt $ k $ ein Körper und alle Algebren über $ k $ kommutativ.

\begin{Def}
Ein affines Gruppenschema über $ k $ ist ein affines Schema $ G=\Spec A $ über $ k $, zusammen mit einer Gruppenstruktur, genauer 
\begin{itemize}
\item Einem Schemamorphismus
\[ G\times_{\Spec k} G \overset{mult}{\longrightarrow} G \]
\item Einem neutralen Element
\[ \Spec k\overset{e}{\longrightarrow} G \]
\item Einem Schemamorphismus "`Inversenabbildung"' 
\[ G\overset{i}{\longrightarrow}G, \]
\end{itemize}
so dass die Abbildungen die gewöhnlichen Gruppenaxiome (Assoziativität, Inverse, neutrales Element) erfüllen. Seien $ G_1 $ und $ G_2 $ affine Gruppenschemata. Ein Morphismus von Gruppenschemata 
\[ f:G_1\longrightarrow G_2 \]
ist ein Schemamorphismus, der die Gruppenmultiplikation respektiert, für den also gilt
\[f\circ mult_{G_1} = mult_{G_2}\circ (f\times f). \]
\end{Def}

\begin{Bem}
Ein affines Gruppenschema über $ k $ ist genau ein Gruppenobjekt in der Kategorie der affinen Schemata über $ k $.
\end{Bem}

Wir erinnern uns an ein fundamentales Resultat aus der algebraischen Geometrie:

\begin{Satz}\label{Spec1}
Der kontravariante Funktor 
\[
\begin{array}{cccc}
\Spec: & \mathfrak{Alg}_k & \longrightarrow & \mathfrak{aff Sch}_k \\
		&  	A	& \longmapsto & \Spec (A)
\end{array}
\]
definiert eine Äquivalenz von Kategorien zwischen der Kategorie der $ k $-Algebren und der Kategorie der affinen Schemata über $ k $. Der inverse Funktor ist gegeben durch den globalen Schnittfunktor. 
\end{Satz}

\begin{proof}
\cite[S.4 ff]{MR563524}
\end{proof}

\begin{Bsp}
Die Gruppe $ Gl(V) $ für einen endlich-dimensionalen $ k $-Vektorraum $ V $ ist ein affines Gruppenschema. Die zugehörige $ k $-Algebra $ A $ ist gegeben durch 
\[ A:=\frac{k[t,x_{11}\dots x_{ij}\dots x_{nn}]}{\big (t\cdot det\big( (x_{ij})_{ij}\big )= 1\big )}\]
\end{Bsp}

\begin{Bem+Def}
Sei $ G=\Spec A $ ein affines Gruppenschema über $ k $. Dann induzieren die Gruppenabbildungen $ (mult,e,i) $ über den Funktor $ Spec $ Abbildungen zwischen $ k $-Algebren
\[ m:A\longrightarrow A\otimes_k A \]
\[\varepsilon:A\longrightarrow k\]
\[ S: A\longrightarrow A, \]
so dass $ \Delta $ und $ \varepsilon $ die Axiome der Koassoziativität und der Koeinheit einer Koalgebra erfüllen (vgl. Kapitel 5). Die Algebra $ A $ trägt also zusätzlich zu seiner Algebrenstruktur eine natürliche Struktur als Koalgebra und wird dadurch zur Bialgebra über $ k $. \\
Die Inversenabbildung des Gruppenschemas $ G $ erfüllt die Bedingung, dass folgendes Diagramm kommutiert:

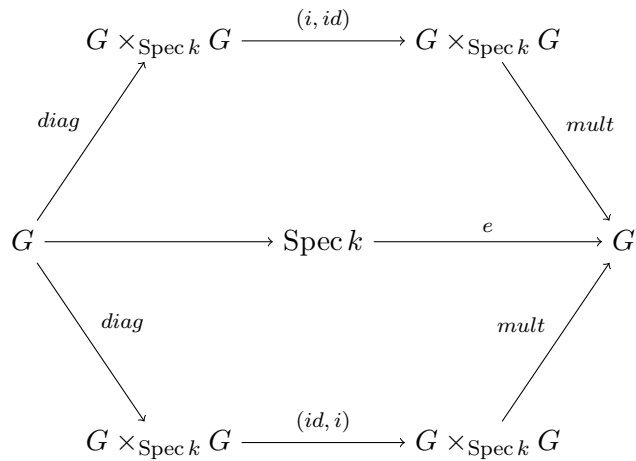
\begin{center}
\begin{tikzpicture}[description/.style={fill=white,inner sep=2pt}]
    \matrix (m) [matrix of math nodes, row sep=5.5em,
    column sep=1.1em, text height=1.5ex, text depth=0.25ex]
    { & G\times_{\Spec k} G & 		 & G\times_{\Spec k} G & \\
    G & 					& \Spec k &					& G\\
      & G\times_{\Spec k} G & 		 & G\times_{\Spec k} G & \\
	};
    \path[->,font=\scriptsize]
    (m-2-1) edge node[auto] {$ diag $} (m-1-2)
    (m-1-2) edge node[auto] {$ (i,id) $} (m-1-4)
    (m-1-4) edge node[auto] {$ mult $} (m-2-5)
    (m-2-1) edge node[auto] {$  $} (m-2-3)
    (m-2-3) edge node[auto] {$ e $} (m-2-5)
    (m-2-1) edge node[auto] {$ diag $} (m-3-2)
    (m-3-2) edge node[auto] {$ (id,i) $} (m-3-4)
    (m-3-4) edge node[auto] {$ mult $} (m-2-5)
    ;   
\end{tikzpicture}
\end{center}

Die Abbildung $ S $, genannt Antipodenabbildung, macht also folgendes Diagramm kommutativ:

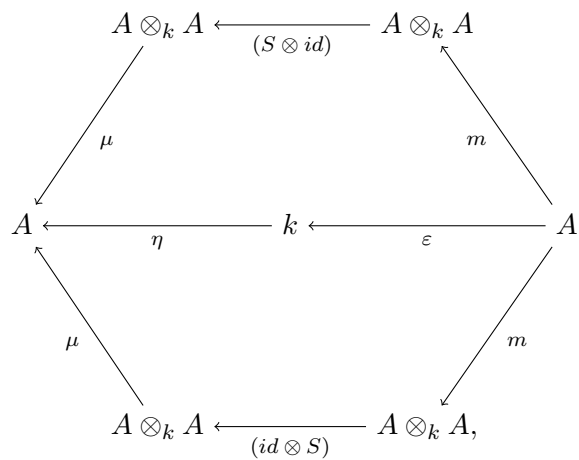
\begin{center}
\begin{tikzpicture}[description/.style={fill=white,inner sep=2pt}]
    \matrix (m) [matrix of math nodes, row sep=5.5em,
    column sep=2.0em, text height=1.5ex, text depth=0.25ex]
    { & A\otimes_k A & 		 & A\otimes_{k} A & \\
    A & 					&    k &			& A\\
      & A\otimes_{k} A & 		 & A\otimes_{k} A, & \\
	};
    \path[->,font=\scriptsize]
    (m-1-2) edge node[auto] {$ \mu $} (m-2-1)
    (m-1-4) edge node[auto] {$ (S\otimes id) $} (m-1-2)
    (m-2-5) edge node[auto] {$ m $} (m-1-4)
    (m-2-3) edge node[auto] {$ \eta $} (m-2-1)
    (m-2-5) edge node[auto] {$ \varepsilon $} (m-2-3)
    (m-3-2) edge node[auto] {$ \mu $} (m-2-1)
    (m-3-4) edge node[auto] {$ (id \otimes S) $} (m-3-2)
    (m-2-5) edge node[auto] {$ m $} (m-3-4)
    ;   
\end{tikzpicture}
\end{center}

wobei $ \mu $ die Algebrenmultiplikation und $ \eta $ die Einheit auf $ A $ beschreibt. Eine Bialgebra mit einer derartigen Antipodenabbildung  nennt man eine \emph{kommutative Hopf-Algebra}.
\end{Bem+Def}

Aus dieser Konstruktion, sowie Satz \ref{Spec1} folgt sofort:

\begin{Satz}\label{Schema}
Der Funktor
\[ A\longrightarrow \Spec A \] 
definiert eine Äquivalenz von Kategorien zwischen der Kategorie der kommutativen $ k $-Hopf-Algebren und der Kategorie der affinen Gruppenschemata über $ k $.
\end{Satz}

\begin{Def}
Sei $ G=\Spec A $ ein affines Gruppenschema über $ K $. Eine \emph{endlich-dimensionale $ G $-Darstellung auf $ V $}  ist ein endlich-dimensionaler $ k $-Vektorraum $ V $, zusammen mit einem Morphismus von Gruppenschemata
\[ \rho:G\longrightarrow Gl(V). \]
Ein \emph{Morphismus von $ G $-Darstellungen} zwischen $ (\rho_1,V) $ und $ (\rho_2,W) $ ist ein Morphismus von Vektorräumen 
\[f:V\longrightarrow W, \]
so dass für alle $ g\in G $ gilt:
\[f\circ\rho_1(g)=\rho_2(g)\circ f. \]
\end{Def}

\begin{Proposition}\label{Komod}
Sei $ G=\Spec A $ ein affines Gruppenschema über $ k $ und $ V $ ein endlich-dimensionaler $ k $-Vektorraum. Dann gibt es eine kanonische 1:1-Beziehung zwischen $ A $-Komodulstrukturen auf $ V $ und linearen Darstellungen von $ G $ auf $ V $.
\end{Proposition}
\begin{proof}
{\cite[2.2]{MR654325}}
\end{proof}

\begin{Def}
Eine \emph{neutrale Tannaka-Kategorie} über $ k $ ist eine starre, abelsche, $ k $-lineare Tensorkategorie $ \C $ zusammen mit einem exakten, treuen, $ k $-linearen Tensorfunktor 
\[ T:\C\longrightarrow k\Vect \]
in die Kategorie der endlich-dimensionalen $ k $-Vektorräume. Einen solchen Funktor nennt man \emph{Faserfunktor} von $ \C $.
\end{Def}

\begin{Theorem}[Tannaka-Dualität]\label{Tannaka}
Sei $ \mathcal{C} $ eine neutrale Tannaka-Kategorie mit Faserfunktor $ T $. Dann gibt es ein affines Gruppenschema $ G $, so dass $ T $ eine Äquivalenz von Tensor-Kategorien
\[ \C\longrightarrow \Rep_k(G) \]
induziert. 
\end{Theorem}

\begin{proof}
Der Beweis stammt von Deligne \cite[S.130ff]{MR654325}. Wir werden hier nur einige Schritte skizzieren und insbesondere den Zusammenhang zu der Diagrammkategorie von Nori klären.\\
Anstatt das Gruppenschema $ G $ zu konstruieren, können wir nach Proposition \ref{Schema} äquivalent eine Hopfalgebra $ A(T) $ über $ k $ angeben. Die Koalgebrenstruktur von $ A(T) $ erhalten wir direkt aus der Theorie über Noris Diagrammkategorie, ohne Verwendung der Tensorstruktur auf $ \C $:\\
Wir fassen die Abbildung
\[ T:\C\longrightarrow k\Vect \]
als Darstellung eines Diagramms $ \C $ auf. Dann faktorisiert $ T $ nach Theorem \ref{thm1} über Noris universelle Diagrammkategorie 
\[ \C\overset{\tilde{T}}{\longrightarrow}\C(T)\overset{ff_t}{\longrightarrow}k\Vect. \]
Da $ \C $ $ k $-linear, abelsch und $ T $ ein Faserfunktor ist, ist der Funktor $ \tilde{T} $ nach Satz \ref{Aequivalenz} eine Äquivalenz von Kategorien. Nach Proposition \ref{Dualitaet} ist $ \C(T) $ äquivalent zur Kategorie der $ A(T) $-Komoduln, wobei 
\[ A(T):=\varinjlim_{E\subset\C}\End(T_{|E})^\vee \]
ist. Die Kategorie $ \C $  ist also äquivalent zu der Kategorie der Komoduln über der Koalgebra $ A(T) $. \\
Als nächstes benötigen wir eine natürliche $ k $-Algebrenstruktur auf $ A(T) $. Eine Abbildung $ u:A\otimes_k A\rightarrow A $ definiert einen Funktor
\[ 
\begin{array}{ccc}
\phi^u: A \Komod\times A\Komod & \longrightarrow  & A\Komod \\
(X,Y) & \longmapsto & X \otimes_k Y, \\
\end{array}
\]
wobei die $ A $-Komodulstruktur auf $ X\otimes_k Y $ gegeben ist durch
\[ X\otimes Y\overset{m\otimes m}{\longrightarrow}X\otimes A\otimes Y\otimes A\cong X\otimes Y\otimes A\otimes A\overset{id\otimes id\otimes u}{\longrightarrow} X\otimes Y\otimes A. \]
Deligne zeigt in \cite[2.16]{MR654325}, dass Abbildung $ u\mapsto\phi^u $ eine Bijektion zwischen Homomorphismen 
\[ A\otimes_k A\longrightarrow A \]
und Funktoren 
\[ 
\phi:A\Komod\times A\Komod \longrightarrow A\Komod
\]
definiert, welche auf $ k $ Vektorräumen $ X,Y $ das gewöhnliche Tensorprodukt $ \phi(X,Y)=X\otimes_k Y $ induziert. Weiter zeigt er, dass die Abbildung $ u $ genau dann eine Multiplikation von unitären Algebren definiert, wenn der Funktor $ \phi^u $ ein Tensorfunktor ist. Da $ \C $ eine Tensorkategorie ist, induziert dies über die Äquivalenz $ \tilde{T} $ eine Tensorstruktur auf $ \C(T) $, also auf $ A(T)\Komod $. Der Tensorfunktor $ T $ definiert nun eine unitäre Algebrenstruktur auf $ A(T) $. 
Die Antipodenabbildungen kann man aus der Tatsache gewinnen, dass $ \C $ als Tensorkategorie starr ist (\cite[2.16 und 1.13]{MR654325}). $ A(T) $ ist also eine Hopfalgebra und $ \Spec A $ nach Proposition \ref{Schema} ein affines Gruppenschema über $ k $.\\
Nach Prop \ref{Komod} schließlich stehen die $ A(T) $-Komodulstrukturen auf einem Vektorraum $ V $ in 1:1-Beziehung zu den $ \Spec A(T) $-Darstellungen auf $ V $, also erhalten wir für $ G:=\Spec A(T) $ eine Äquivalenz von Kategorien
\[\C\cong A(T)\Komod\cong Rep_k(G). \]
\end{proof}

\begin{Satz}\label{6}
Ist die Kategorie $ \C $ nicht starr, so fehlt uns die Antipodenabbildung und $ A(T) $ wird keine Hopf-Algebra, sondern nur eine Bialgebra. $ \Spec A(T) $ ist dann ein affines Monoidschema anstatt eines affines Gruppenschemas.
\end{Satz}

\chapter{Ausblick: Nori-Motive}
\label{ch: Variet}

\section{Nori-Motive als Diagrammkategorie der guten Paare}

In diesem Kapitel werden wir die Kategorie der effektiven gemischten Nori-Motive definieren. Die Konstruktion geht zurück auf Nori \cite{Nori}. Eine veröffentliche Beweisskizze findet man bei Levine \cite{MR2182598}. Diese Kategorie ist eine abelsche, $ \mathbb{Q} $-lineare Tensorkategorie. Tannaka-Dualität aus Kapitel 6 besagt, dass diese Kategorie äquivalent zu den Darstellungen eines gewissen affinen Gruppenmonoids sind. \\

Sei in diesem Kapitel $ k $ ein Teilkörper von $ \mathbb{C} $ mit einer fest gewählten Einbettung.\\
\begin{Konvention}
Unter einer \emph{Varietät} über einem Körper $ k $ verstehen wir stets ein separiertes Schema endlichen Typs über $ k $. 
\end{Konvention}

\begin{Def}
Sei $ (X,\mathcal{O}_X) $ eine affine Varietät über $ \mathbb{C} $.  Dann ist $ X $ isomorph zum Spektrum einer endlich erzeugten $ \mathbb{C} $-Algebra $ \mathbb{C}[x_1,..,x_n]/(f_1,..,f_m) $. Wir definieren den zu $ (X,\mathcal{O}_X) $ \emph{assoziierten komplex-analytischen Raum} $ (X^{an},\mathcal{O}_{X^{an}}) $ wie folgt:
\[ X^{an}:=\{ z\in\mathbb{C}^n|f_1(z)=...=f_m(z)=0\} \]
mit der von $ \mathbb{C}^n $ induzierten analytischen Topologie. 
\[ \mathcal{O}_{X^{an}}:=\mathcal{O}_{\mathbb{C}^n}/(f_1,..,f_m),  \]
wobei $ \mathcal{O}_{\mathbb{C}^n} $ die Garbe der holomorphen Funktionen bezeichnet. 
Für beliebige Varietäten $ (X,\mathcal{O}_X) $ definieren wir den assoziierten komplex-analytischen Raum als die Verklebung der komplex-analytischen Räume zu einer gewählten affinen Überdeckung.
\end{Def}

Eine detailliertere Beschreibung des assoziierten komplex-analytischen Raums findet man in  \cite[B1]{Har}.

\begin{Def}
Sei $ X $ eine Varietät über $ k $ und $ Y $ eine abgeschlossene Untervarietät von $ X $. Wir bezeichnen mit
\[ H^i_B(X,Y):=H^i_{sing}(X_{\mathbb{C}}^{an},Y_{\mathbb{C}}^{an};\mathbb{Q}) \]
die $ i $-te relative singuläre Kohomologie mit Koeffizienten in $ \mathbb{Q} $ der assoziierten komplex-analytischen Räume, auch genannt \emph{Betti-Kohomologie}.
\end{Def}

\begin{Def}[Nori]
Wir definieren das Diagramm $ D^{\text{eff}} $ der \emph{effektiven Paare} wie folgt:
\begin{itemize}
\item Ecken sind Tripel $ (X,Y,i) $, wobei $ X $ eine Varietät über $ k $, $ Y\subset X  $ eine abgeschlossene Untervarietät und $ i\in\mathbb{N}_0 $  ist. 
\item Wir haben zwei Arten von Kanten zwischen Objekten: \begin{itemize}
\item Für einen Morphismus von Varietäten $ f:X \rightarrow X' $ mit $ f(Y)\subset Y' $ und $ i\in\mathbb{N}_0 $ einen Morphismus  ("`Funktorialität"')
\[ f^*:(X',Y',i)\longrightarrow (X,Y,i). \]
\item Für eine Kette $ X\subseteq Y\subseteq Z $ abgeschlossener Untervarietäten über $ k $ und $ i\in\mathbb{N}_0 $ den Morphismus ("`Korand"')
\[ \partial:(Y,Z,i)\longrightarrow (X,Y,i+1). \]
\end{itemize}
\end{itemize}
\end{Def}

\begin{Def}
Sei $ D^{\text{eff}} $ das Diagramm der effektiven Paare. Wir betrachten die Darstellung
\[
\begin{array}{cccc}
H^*_B: & D^{\text{eff}} & \longrightarrow & \mathbb{Q}\Vect \\
 & (X,Y,i) & \mapsto & H^i_B(X,Y),\\
\end{array}
\]
welche einem Tripel $ (X,Y,i) $ die $ i $-te Betti-Kohomologie zuordnet. \\
Dann faktorisiert diese Darstellung gemäß Theorem \ref{thm1} über die universelle abelsche Diagrammkategorie $ \C(H^*_B) $. Diese Kategorie nennen wir $ \MM $, die \emph{Kategorie der effektiven, gemischten Nori-Motive}.
\end{Def}

\begin{Bem}\label{g}
Diese Kategorie ist Noris Kandidat für eine Kategorie der gemischten Motive. Es ist nicht klar, ob und wie jede Kohomologietheorie, welche die Bloch-Ogus Axiome erfüllt, über diese Kategorie faktorisiert. Ist jedoch $ \tilde{H}^\bullet $ eine Kohomologietheorie mit Koeffizienten in $ \mathbb{Q} $, welche Werte in einer abelschen Kategorie annimmt und durch Koeffizientenerweiterung Vergleichsisomorphismen in die singuläre Kohomologie besitzt, dann liefert die universelle Eigenschaft der Diagrammkategorie $ \MM $ einen treuen, $ R $-linearen, exakten Funktor
\[ \MM\longrightarrow \tilde{H}^\bullet, \]
welcher auf $ \mathbb{Q} $-Vektorräumen den Vergleichsisomorphismus induziert. Dies gilt zum Beispiel für  $l$-adisch \'etale Kohomologie und algebraische de Rham-Kohomologie. 
\end{Bem}

Es ist ein elegantes Resultat, dass die Kategorie der effektiven gemischten Nori-Motive schon die Diagrammkategorie eines deutlich kleineren Diagramms ist:

\begin{Def}
Das Diagramm der effektiven sehr guten Paare $ \tilde{D}^{\text{eff}} $ besteht aus dem vollen Unterdiagramm von $ D^{\text{eff}} $ , so dass für Ecken $ (X,Y,i) $ 
\begin{itemize}
\item $ H^j_B(X,Y)=0\hspace{1cm}\text{für }i\neq j \hspace{1cm}\text{und }$
\item entweder gilt \begin{itemize}
\item $ X $ ist affin von Dimension $ i $, $ Y $ ist von Dimension $ i-1 $ und $ X\diagdown Y $ ist glatt \\
\end{itemize}
oder
\begin{itemize}
\item $ X=Y $ und $ dim(X)<i. $  
\end{itemize}
\end{itemize}
\end{Def}

\begin{Theorem}[Nori]
Die Diagrammkategorie $ \C(\tilde{D}^{\text{eff}}) $ der effektiven sehr guten Paare bezüglich der Darstellung in die relative singuläre Kohomologie ist äquivalent zur Diagrammkategorie der effektiven Paare $ \MM $.
\end{Theorem}

\begin{proof}
\cite[1.6]{paper}
\end{proof}

\begin{Bem}
Die Motivation für diese Konstruktion stammt aus der algebraischen Topologie. Für einen endlichen CW-Komplex $ X $ sei $ X^i $ das $ i $-Skelett von $ X $. Dann gilt für die relative singuläre Homologie des Skeletts:
\[ \operatorname{H_i(X^j,X^{j-1};\mathbb{Z})}=\begin{cases} 0\ , & i\neq j\ ,\\ \text{freie abelsche Gruppe über} \  \\ \text{der Menge der j-Simplizes} , & i=j.\ \end{cases}\]
Außerdem lässt sich die Homologie von $ X $ leicht aus den Gruppen $ H_i(X^j,X^{j-1};\mathbb{Z} $ berechnen. (Siehe zum Beispiel \cite[2.34,2.35]{MR1867354}.)
Diese Idee greift Nori mit der Idee der guten Paare auf. Wir brauchen aber für eine Varietät $ X $ eine entsprechende Filtrierung, welche ähnliche Charakteristika hat wie das Skelett für einen CW-Komplex. Fundamental hierfür ist Noris "`Basic Lemma"'.
\end{Bem}

\begin{Satz}[Basic Lemma von Nori]
Sei $ R $ ein noetherscher Ring. Sei $ X $ eine affine Varietät der Dimension $ n $ über $ k\subset\mathbb{C} $ und sei $ Z\subset X $ eine abgeschlossene Untervarietät der Dimension $ \le n-1 $. Dann existiert eine abgeschlossene Untervarietät $ Y\supset Z$, so dass $ (X,Y,n) $ ein gutes Paar ist und es gilt:
\begin{itemize}
\item $ dim(Y)\le n-1 $.
\item $ H^i_{sing}(X_{\mathbb{C}}^{an},Y_{\mathbb{C}}^{an};R)=0 $ für $ i\neq n $.
\item $ H^n_{sing}(X_{\mathbb{C}}^{an},Y_{\mathbb{C}}^{an};R) $ ist ein freier, endlich erzeugter $ R $-Modul.
\end{itemize}
Darüber hinaus kann man $ Y $ so wählen, dass $ X\diagdown Y $ glatt ist. 
\end{Satz}

\begin{proof}
\cite{MR923133}
\end{proof}

Mit Hilfe des Basic Lemmas lässt sich per Induktion für jede affine Varietät $ X  $ eine Filtrierung
\[ \emptyset=F_{-1}X\subset F_0 X \subset \cdots \subset F_{n-1}X \subset F_n X=X \]
konstruieren, so dass $ (F_jX,F_{j-1}X,j) $ ein sehr gutes Paar ist. 

\section{Nori-Motive als Darstellungen über einem affinen Schema}

\begin{Bem}\label{7}
Wir haben die Kategorie $ \MM $ der gemischten effektiven Motive als die Diagrammkategorie einer Darstellung in die endlich-dimensionalen \\$ \mathbb{Q} $-Vektorräume konstruiert. Da jeder Körper ein Hauptidealring und jeder Vektorraum frei ist, wissen wir aus Satz \ref{Dualitaet}, dass $ \MM $ äquivalent zu der Kategorie der endlich erzeugten Komoduln über der Koalgebra $ A(\tilde{H}^*):=\underset{F\subset D^{\text{eff}}}{\varinjlim}\End(\tilde{H}^*_F)^\vee $ ist.
\[
H^*_B: D^{\text{eff}}  \longrightarrow \mathbb{Q}\Vect
\]
faktorisiert also über
\[
D^{\text{eff}} \overset{\tilde{H}^*_B}{\longrightarrow} \MM=A(\tilde{H}^*)\Komod \overset{ff_{{\tilde{H}^*_B}}}{\longrightarrow}\mathbb{Q}\Vect,
\]
wobei $ ff_{\tilde{H}^*_B} $ den Vergissfunktor in die $ \mathbb{Q} $-Vektorräume beschreibt. 
\end{Bem}

\begin{Satz}
Die Kategorie $ \MM $ der effektiven gemischten Nori-Motive ist eine Tensorkategorie. Der Vergissfunktor $ ff_{\tilde{H}^*} $ wird zu einem Faserfunktor in die $ \mathbb{Q} $-Vektorräume. 
\end{Satz}

\begin{proof}
\cite[1.5]{paper}
\end{proof}

\begin{Korollar}
Die Kategorie der effektiven gemischten Nori-Motive $ \MM $ ist äquivalent zur Kategorie der endlich-dimensionalen Darstellungen des affinen Monoidschemas $\Spec A(\tilde{H}^*) $.
\end{Korollar}

\begin{proof}
Dies ist Tannaka-Dualität: Nach Bemerkung \ref{7} ist 
\[ \MM=A(\tilde{H}^*)\Komod. \]
Die Tensorstruktur auf $ \MM $ gibt $ A(\tilde{H}^*) $ nach Satz \ref{6} die Struktur einer Bialgebra. Ebenfalls nach Satz \ref{6} ist dann $ \Spec A(\tilde{H}^*) $ ein affines Monoidschema und $ \MM $ äquivalent zu den endlich-dimensionalen $ \Spec A(\tilde{H}^*) $-Darstellungen.
\end{proof}

\begin{Bem}
Die Kategorie $ \MM $ ist zwar eine abelsche Tensorkategorie, jedoch nicht starr, da für Objekte $ X,Y $ nicht notwendigerweise das in Definition \ref{hoho} definierte Objekt $ \Hom(X,Y) $ existieren muss. Man kann aber ein entsprechend größeres Diagramm $ D $ definieren, dessen Diagrammkategorie $ \mathcal{MM}_{\text{Nori}} $ auch diese Objekte besitzt \cite[B.19 und 1.15]{paper}. Tatsächlich erhält man $ \mathcal{MM}_{\text{Nori}} $ als die Lokalisierung von $ \MM $ an einem geeigneten Objekt \cite[B.19]{paper}. Die Kategorie  $ \mathcal{MM}_{\text{Nori}} $ ist starr, also eine neutrale Tannaka-Kategorie. Sein Tannaka-Duales ist dann nach Satz \ref{Tannaka} sogar ein affines Gruppenschema.
\end{Bem}

\begin{Bem}
Abschließend möchten wir noch bemerken, dass wir Nori-Motive auch als Diagrammkategorie der Betti-Kohomologie mit anderen Koeffizienten als $ \mathbb{Q} $ hätten definieren können. Tatsächlich haben wir uns in den früheren Kapiteln viel Mühe gegeben, Noris Diagrammkategorie für alle Diagramme zu konstruieren, welche Werte in endlich erzeugten $ R $-Moduln über noetherschen Ringen $ R $ annehmen. 
Nehmen wir also das Diagramm der effektiven, guten Paare und das Diagramm
\[
\begin{array}{cccc}
H^*_R: & D^{\text{eff}} & \longrightarrow & R\Mod \\
 & (X,Y,i) & \mapsto & H^i(X^{an},Y^{an},R)\\
\end{array}
\]
in die singulären Kohomologiegruppen mit Koeffizienten in $ R $, so erhalten wir für diese Darstellung ebenfalls eine abelsche universelle Diagrammkategorie. Damit diese Kategorie allerdings eine Kategorie von Komoduln über eine Koalgebra wird, benötigen wir als Koeffizienten Hauptidealringe und müssen fordern, dass die Darstellung nur Werte in freien $ R $-Moduln annimmt (Satz \ref{Dualitaet}). 
\end{Bem}

\nocite{*}

\bibliography{library}

\begin{thebibliography}{HMS11}


\providecommand{\url}[1]{\texttt{#1}}
\expandafter\ifx\csname urlstyle\endcsname\relax
  \providecommand{\doi}[1]{doi: #1}\else
  \providecommand{\doi}{doi: \begingroup \urlstyle{rm}\Url}\fi

\bibitem[AM69]{atiyahmacdonaldintroduction69}
\textsc{Atiyah}, M.F. ; \textsc{MacDonald}, I.G.:
\newblock \emph{Introduction to Commutative Algebra}.
\newblock Addison-Wesley, 1969
  \url{http://www.math.uchicago.edu/~allanaa/atiyah.html}

\bibitem[Be{\u\i}87]{MR923133}
\textsc{Be{\u\i}linson}, A.~A.:
\newblock On the derived category of perverse sheaves.
\newblock {In: }\emph{{$K$}-theory, arithmetic and geometry ({M}oscow,
  1984--1986)} Bd. 1289.
\newblock Berlin : Springer, 1987, S. 27--41

\bibitem[Bou48]{MR0026989}
\textsc{Bourbaki}, N.:
\newblock \emph{\'{E}l\'ements de math\'ematique. {VII}. {P}remi\`ere partie:
  {L}es structures fondamentales de l'analyse. {L}ivre {II}: {A}lg\`ebre.
  {C}hapitre {III}: {A}lg\`ebre multilin\'eaire}.
\newblock Hermann et Cie., Paris, 1948

\bibitem[Bru04]{MR2793022}
\textsc{Brugui{\`e}res}, Alain:
\newblock \emph{On a tannakien theorem due to Nori}.
\newblock Preprint, 2004

\bibitem[DG80]{MR563524}
\textsc{Demazure}, Michel ; \textsc{Gabriel}, Peter:
\newblock \emph{North-Holland Mathematics Studies}. Bd.~39: {\emph{Introduction
  to algebraic geometry and algebraic groups}}.
\newblock Amsterdam : North-Holland Publishing Co., 1980

\bibitem[DM82]{MR654325}
\textsc{Deligne}, Pierre ; \textsc{Milne}, James~S.:
\newblock Tannakian categories.
\newblock {In: }\emph{Deligne, Pierre and Milne, James S. and Ogus, Arthur and
  Shih, Kuang-yen: Hodge cycles, motives, and {S}himura varieties} Bd. 900.
\newblock Springer-Verlag, 1982, S. 101--228

\bibitem[Eis95]{MR1322960}
\textsc{Eisenbud}, David:
\newblock \emph{Commutative algebra with a view toward algebraic geometry}.
\newblock New York : Springer-Verlag, 1995 (Graduate Texts in Mathematics)

\bibitem[Fre64]{Freyd1964}
\textsc{Freyd}, Peter:
\newblock \emph{Abelian categories. {A}n introduction to the theory of
  functors}.
\newblock New York : Harper \& Row Publishers, 1964 (Harper's Series in Modern
  Mathematics)

\bibitem[Gro]{G}
\textsc{Grothendieck}, Alexander:
\newblock Récoltes et Semailles, Réflexions et témoignages sur un passé de
  mathématicien (1985-86).
\newblock {In:
  }\emph{http://www.math.jussieu.fr/~leila/grothendieckcircle/recoltesetc.php}

\bibitem[Har77]{Har}
\textsc{Hartshorne}, R.:
\newblock \emph{Algebraic Geometry}.
\newblock Springer, 1977 (Graduate Texts in Mathematics)

\bibitem[Hat02]{MR1867354}
\textsc{Hatcher}, Allen:
\newblock \emph{Algebraic topology}.
\newblock Cambridge : Cambridge University Press, 2002

\bibitem[HMS11]{paper}
\textsc{Huber}, Annette ; \textsc{Müller-Stach}, Stephan:
\newblock On the relation between Nori Motives and Kontsevich Periods.
\newblock {In: }\emph{arXiv:1105.0865v2 [math.AG]}  (2011)

\bibitem[Hub95]{MR1439046}
\textsc{Huber}, Annette:
\newblock \emph{Lecture Notes in Mathematics}. Bd. 1604: {\emph{Mixed motives
  and their realization in derived categories}}.
\newblock Berlin : Springer-Verlag, 1995

\bibitem[Lan02]{MR1878556}
\textsc{Lang}, Serge:
\newblock \emph{Algebra}.
\newblock third.
\newblock New York : Springer-Verlag, 2002 (Graduate Texts in Mathematics)

\bibitem[Lev05]{MR2182598}
\textsc{Levine}, Marc:
\newblock Mixed motives.
\newblock {In: }\emph{Handbook of {$K$}-theory. {V}ol. 1, 2} Bd.~12.
\newblock Springer, 2005, S. 429--521

\bibitem[Mac72]{MR0354799}
\textsc{MacLane}, Saunders:
\newblock \emph{Kategorien. {B}egriffssprache und mathematische {T}heorie}.
\newblock Berlin : Springer-Verlag, 1972

\bibitem[Nor]{Nori}
\textsc{Nori}, M.V.:
\newblock Mitschrieb von Vorlesung am TIFR Mumbai, unveröffentlicht, 32
  Seiten.
\newblock

\bibitem[SR72]{MR0338002}
\textsc{Saavedra~Rivano}, Neantro:
\newblock \emph{Cat\'egories {T}annakiennes}.
\newblock Berlin : Springer-Verlag, 1972 (Lecture Notes in Mathematics, Vol.
  265)

\bibitem[Str07]{MR2294803}
\textsc{Street}, Ross:
\newblock \emph{Quantum groups}.
\newblock Cambridge : Cambridge University Press, 2007 (Australian Mathematical
  Society Lecture Series)

\bibitem[Wed04]{MR2101076}
\textsc{Wedhorn}, Torsten:
\newblock On {T}annakian duality over valuation rings.
\newblock {In: }\emph{J. Algebra} 282 (2004), Nr. 2, S. 575--609

\end{thebibliography}

\cleardoublepage

\vfill
\subsection*{Erklärung}

Hiermit versichere ich, die vorliegende Arbeit selbstständig angefertigt und keine anderen als die angegebenen Quellen und Hilfsmittel benutzt zu haben.

\vspace{1cm}

Freiburg, den 30. Juli 2011
\begin{flushright}\vspace{-0.35cm}$\overline{\qquad\quad\text{ (Jonas von Wangenheim) }\qquad\quad}\quad$\end{flushright}

\end{document}